\documentclass[11pt,reqno,letterpaper]{amsart}
\usepackage{amsfonts}
\usepackage{amssymb}
\usepackage{amsmath}
\usepackage{mathrsfs}
\usepackage{color}
\usepackage[nospace,noadjust]{cite}
\usepackage{verbatim}
\usepackage{dsfont}
\usepackage{bbm}
\usepackage{enumerate}
\usepackage{graphicx}
\usepackage{xcolor}
\usepackage{slashed}
\usepackage[titletoc,title]{appendix}
\usepackage{wrapfig}
\usepackage{marginnote}
\usepackage{bm}
\usepackage{enumitem}
\usepackage{hyperref}
\usepackage{enumerate}

\usepackage{fancyhdr}
\usepackage{geometry}

\definecolor{deepgreen}{cmyk}{1,0,1,0.5}
\definecolor{violetrev}{rgb}{0.55,0.0,0.7}

\newcommand{\Del}[1]{}

\numberwithin{equation}{section}

\newtheorem{theorem}{Theorem}[section]
\newtheorem{corollary}[theorem]{Corollary}
\newtheorem{lemma}[theorem]{Lemma}
\newtheorem{proposition}[theorem]{Proposition}
\newtheorem{claim}[theorem]{Claim}
\newtheorem{remark}[theorem]{Remark}

\newcommand{\norm}[1]{\Vert#1\Vert}

\newcommand{\set}[1]{\{#1\}}

\newcommand{\angles}[2]{\langle #1,#2\rangle}

\newcommand{\be}{\begin{equation}}
\newcommand{\ee}{\end{equation}}
\newcommand{\bpr}{\begin{proof}}
\newcommand{\epr}{\end{proof}}
\newcommand{\bel}{\begin{equation}\label}
\newcommand{\eeq}{\end{equation}}
\newcommand{\ba}{\begin{aligned}}
\newcommand{\ea}{\end{aligned}}
\newcommand{\bpm}{\begin{pmatrix}}
\newcommand{\epm}{\end{pmatrix}}

\newcommand{\bp}{\begin{pmatrix}}
\newcommand{\ep}{\end{pmatrix}}

\newcommand{\bea}{\begin{eqnarray}}
\newcommand{\eea}{\end{eqnarray}}
\newcommand{\bee}{\begin{eqnarray*}}
\newcommand{\eee}{\end{eqnarray*}}
\newcommand{\ben}{\begin{enumerate}}
\newcommand{\een}{\end{enumerate}}

\newcommand{\sgn}{{\mathrm{sgn}}}

\newcommand{\bfu}{{\bf u}}

\newcommand{\bfw}{{\bf w}}

\newcommand{\bfH}{{\bf H}}

\newcommand{\bfY}{{\bf Y}}
\newcommand{\bfZ}{{\bf Z}}

\newcommand{\bfphi}{\mathbf{\phi}}

\newcommand{\bark}{{\overline k}}

\newcommand{\barx}{{\overline x}}

\renewcommand{\hbar}{{\underline h}}

\newcommand{\kbar}{{\underline k}}

\newcommand{\calA}{\mathcal A}
\newcommand{\calB}{\mathcal B}
\newcommand{\calC}{\mathcal C}

\newcommand{\calG}{\mathcal G}
\newcommand{\calH}{\mathcal H}
\newcommand{\calI}{\mathcal I}
\newcommand{\calJ}{\mathcal J}
\newcommand{\calK}{\mathcal K}
\newcommand{\calL}{\mathcal L}
\newcommand{\calM}{\mathcal M}

\newcommand{\calO}{\mathcal O}

\newcommand{\calY}{\mathcal Y}

\newcommand{\tila}{{\tilde{a}}}

\newcommand{\tilb}{{\tilde{b}}}

\newcommand{\wtilf}{{\widetilde{f}}}

\newcommand{\tilu}{{\tilde{u}}}

\newcommand{\tilw}{{\tilde{w}}}

\newcommand{\tilx}{{\tilde{x}}}

\newcommand{\wtilC}{{\widetilde{C}}}

\newcommand{\wtilH}{{\widetilde{H}}}

\newcommand{\wtilJ}{{\widetilde{J}}}

\newcommand{\wtilL}{{\widetilde{L}}}

\newcommand{\tilN}{{\tilde{N}}}

\newcommand{\wtilQ}{{\widetilde{Q}}}

\newcommand{\checka}{{\check{a}}}
\newcommand{\checkb}{{\check{b}}}

\newcommand{\checku}{{\check{u}}}

\newcommand{\checkN}{{\check{N}}}

\DeclareMathOperator{\arcosh}{arcosh}

\newcommand{\px}{\partial_x}

\newcommand{\pt}{\partial_t}

\newcommand{\bfvareps}{\boldsymbol{\varepsilon}}

\DeclareMathOperator{\sech}{sech}

\newcommand{\R}{{\mathbb R}}

\newcommand{\N}{{\mathbb N}}

\newcommand{\eps}{{\varepsilon}}

\providecommand{\MR}{\relax\ifhmode\unskip\space\fi MR }

\providecommand{\href}[2]{#2}

\begin{document}

\title[Yang-Mills kink dynamics]{Kink dynamics for the Yang-Mills field in an extremal Reissner-Nordstr\"om black hole}

\author{Ignacio Acevedo}
\address{Laboratoire de Math\'{e}matiques d’Orsay (UMR 8628), Universit\'{e} Paris-Saclay, CNRS, 91405 Orsay Cedex, France.}
\email{ignacio-andres.acevedo-ramos@cnrs.fr}

\thanks{
	I.A. was partially funded by Chilean research grant ANID/MAGISTER NACIONAL/2022-22221702, Fondecyt 1231250 and ANID Exploraci\'on 13220060. 
}

\author[Claudio Mu\~noz]{Claudio Mu\~noz}  
\address{Departamento de Ingenier\'{\i}a Matem\'atica and Centro
	de Modelamiento Matem\'atico (UMI 2807 CNRS), Universidad de Chile, Casilla
	170 Correo 3, Santiago, Chile.}
\email{cmunoz@dim.uchile.cl}
\thanks{C.M. was partially funded by Chilean research grants FONDECYT 1191412, 1231250, ANID Exploraci\'on 13220060 and Basal CMM FB210005.}

\begin{abstract}
	Considered in this work is the Yang-Mills field in an extremal Reissner-Nordstr\"om black hole, a physically motivated mathematical model introduced by Bizo\'n and Kahl. The kink is a fundamental, strongly unstable stationary solution in this non-perturbative, variable coefficients model, with a polynomial tail and no explicit form. In this paper, we introduce and extend several virial techniques, adapt them to the inhomogeneous medium setting, and construct a finite codimensional manifold of the energy space where the kink is asymptotically stable. In particular, we handle, using virial techniques, the emergence of a weak threshold resonance in the description of the stable manifold.
\end{abstract}


\maketitle 

\tableofcontents

\section{Introduction}

\subsection{Setting} 
The exterior of the extremal Reissner-Nordstr\"om black hole is a globally hyperbolic static spacetime $(\mathcal M,\hat g)$ with metric
\[
 \hat g := -\left( 1-\frac{M}r\right)^2dt^2+ \left( 1-\frac{M}r\right)^{-2} dr^2 +r^2 (d\theta^2 + \sin^2\theta d\phi^2),
\]
with $t\in\mathbb R$, $r>M$, $(\theta,\phi)\in \mathbb S^2$, and $M>0$ a positive constant. Extremal black holes have recently  become of great importance in Physics and Astronomy because it is believed that supermassive black holes in the center of galaxies are precisely characterized by extremal or near to extremal properties \cite{Daly}. Under the change of variables $\tau =\frac t{4M} \in \mathbb R$, $x= \log \left( \frac rM-1 \right) \in \mathbb R$, $\hat g = 16M^2 (1+e^{-x})^{-2} g$, a geodesically complete spacetime $(\mathcal{M}, g)$ is obtained, where the metric $g$ is given by ($\tau=t$ by simplicity)
\begin{equation*}
	g = -dt^2 + \cosh^4 \left(\frac x2\right)(dx^2 + d\theta^2 + \sin^2\theta d\varphi^2).
\end{equation*}
In \cite{BizKah16}, Bizo\'n and Kahl studied the static solutions of a Yang-Mills field placed in the exterior of an extremal Reissner-Nordstr\"om black hole defined by $g$ (see also \cite{BizChmRosR07} for previous work in the case of other black holes).  Proposing a spherically symmetric and purely magnetic $\text{SU}(2)$ Yang-Mills field propagating in $(\mathcal{M}, g)$ and having the specific form
\begin{equation*}
	A(t, x) = \varphi(t, x)\omega(\tau_1, \tau_2) + \tau_3 \cos\theta d\varphi,
\end{equation*}
where $\omega(\tau_1, \tau_2) = \tau_1 d\theta + \tau_2 \sin\theta d\varphi$ and $\varphi =\varphi(t,x)$ is a real scalar field. Here, $\{\tau_k\}_{k=1}^3$ are the $2\times 2$ complex matrix generators of $SU(2)$ such that $[\tau_k,\tau_l]=i \epsilon_{klm}\tau_m$. With this ansatz, we obtain the variable coefficients Lagrangian density
\begin{equation}\label{eq:lagrangian_density}
\ba
	& \mathcal{L}[x,\varphi, \partial_x\varphi,\partial_t \varphi]  = -\frac12 \cosh^2 \left(\frac x2\right) (\partial_t \varphi)^2 + \frac12 \sech^2 \left(\frac x2\right) \left(  (\partial_x \varphi)^2 + \frac12 (1-\varphi^2)^2 \right).
\ea
\end{equation}
The associated Euler-Lagrange equation  for the field $\varphi$, equivalent to the associated Yang-Mills model, is given by
\begin{equation}\label{eq:EL}
	\partial_t^2 \varphi - Q\partial_x (Q\partial_x \varphi) + Q^2 (\varphi^2 - 1)\varphi = 0,
\end{equation}
obtained after the time rescaling $\varphi(t,x)  \mapsto \varphi(\frac32 t, x)$, where $Q$ is the standard KdV soliton:
\begin{equation}\label{Q}
	Q(x) := \frac 32\sech^2\left(\frac x2\right).
\end{equation} 
Unlike standard scalar field models, \eqref{eq:EL} has neither Lorentz nor space translation invariance, and the theory of asymptotic stability developed in \cite{KMMV20} does not apply. However, the time translation invariance induces a Hamiltonian structure. Indeed, from the Lagrangian density \eqref{eq:lagrangian_density}, the \emph{energy}
\begin{equation*}
	\mathcal E[\varphi, \varphi_t] = \int \left( \frac12 Q^{-1}(\partial_t \varphi)^2 + \frac12 Q \left( (\partial_x \varphi)^2 + \frac12(1 - \varphi^2)^2 \right) \right)dx
\end{equation*}
is formally conserved along the flow, thanks to the associated continuity equation
\begin{equation*}
	Q^{-1}\partial_t (\partial_t \varphi)^2 + Q \partial_t \left( (\partial_x \varphi)^2 + \frac12(1 - \varphi^2)^2 \right) + \partial_x (Q \partial_t \varphi \partial_x \varphi) = 0.
\end{equation*}
Since there is no space translation invariance over the system, there is a lack of conservation for the natural physical momentum
\begin{equation}\label{momentum_law}
	\mathcal P[\varphi, \partial_t \varphi] = \int Q^{-1}\partial_t\varphi \partial_x\varphi dx.
\end{equation}	
However, a particular version of this quantity will be essential for the proof of our main results.

\subsection{Kinks}
Static solutions $H=H(x)$ of \eqref{eq:EL} solve
\begin{equation}\label{eq:static-EL}
	H'' - \tanh\left(\frac{x}{2} \right)H' + H(1-H^2) = 0, \quad x\in\mathbb R. 
\end{equation}
The first non-trivial solution to this equation is given by  \cite{BizKah16}
\begin{equation}\label{eq:H}
	H(x)  = \tanh\left(\frac{x}{2}\right)\,.
\end{equation}
We call $\bfH = (H, 0)$ the \textit{kink} associated to this model. The physical meaning of kinks and their key importance in High Energy Physics and General Relativity have been described in detail in the literature, the reader can consult the monographs \cite{ManSut04,Vac06,Sre07}. The mathematical structure of kink solutions has achieved an impressive knowledge during the past years. Among them, the kink of the integrable sine-Gordon equation has garnered attention due to its complexity and the absence of kink asymptotic stability in the energy space \cite{MP,AMP,CLL,LuS1,ChenLu}. See \cite{KMM17_Survey,CucMae_survey,Germain} for detailed surveys on the long-time behavior and asymptotic of nonlinear waves.  

More generally, in \cite{BizKah16} a countable family of time-independent smooth finite energy solutions $H_n(x)$, $n\geq 0$ of \eqref{eq:static-EL} was found. These are characterized by $H_0:=1$, $H_1=H$, with $H_n$ having $n$ zeros, $|H_n(x)|<1$ for all $x$, $\lim_{|x|\to \infty} |H_n(x)| =1$, $H_n$ is even (odd) for even (odd) $n$, and $\lim_{n\to +\infty} H_n(x)=0$. They also provided strong evidence that $L_n$, the linearized operator at the ``kink'' $H_n$, has exactly $n$ negative eigenvalues. Finally, they introduced the hyperboloidal formulation $s= t -\frac12 (\cosh x +\log (2\cosh x)),$ $z=\tanh \left(\frac{x}2\right)$ for the variable coefficients nonlinear wave problem and proved that, after a compactification of space, there is a decreasing energy. In these coordinates, $H_1(x) =H(x)=z$.

Following Bizo\'n and Kahl \cite{BizKah16}, we introduce the function 
\begin{equation}\label{eq:alpha}
	\alpha(x) = \frac13 (\sinh x + x),
\end{equation}
which is strictly monotone and bijective from $\mathbb R$ onto itself. Its inverse function, denoted $\alpha^{-1}$, has no exact closed form and exhibits only logarithmic growth. Define the distorted soliton and kink as
\begin{equation}\label{tilde Q tilde H}
	\wtilQ(x) = Q(\alpha^{-1}(x)), \qquad \wtilH(x) = H(\alpha^{-1}(x)),
\end{equation}
with $Q$ and $H$ as in \eqref{Q} and \eqref{eq:H}, respectively. Both functions have only polynomial rate of convergence at infinity, with
\begin{equation}\label{tails}
0\leq \wtilQ(x) \lesssim \frac{1}{|x|}, \qquad \left| \wtilH (x)\mp 1 \right| \lesssim \frac{1}{|x|}, \qquad \hbox{as}\quad  x\to \pm \infty.
\end{equation}
If $\varphi$ is a solution of the equation \eqref{eq:EL}, then $\phi = \varphi \circ \alpha^{-1}$ solves
\begin{equation}\label{eq:varEL}
	\partial_{t}^2 \phi - \partial_x^2 \phi - \wtilQ^2 (\phi -\phi^3) = 0.
\end{equation}  
Let $\boldsymbol{\phi} = (\phi, \partial_t \phi) = (\phi_1, \phi_2)$. Then \eqref{eq:varEL} becomes
\begin{equation}\label{eq:varELsystem}
	\begin{cases}
		\partial_t\phi_1 = \phi_2 \\[0.2cm]
		\partial_t\phi_2 = \partial_x^2 \phi_1 + \wtilQ^2(x) (1-\phi_1^2) \phi_1 .
	\end{cases}
\end{equation}
Notice that $\widetilde{\bfH}=(\widetilde H,0)$ is an exact solution to this model. The conserved energy reads now
\begin{equation}\label{eq:energy}
	E[\phi_1, \phi_2] =\frac12  \int  \left( \phi_2^2 + (\partial_x \phi_1)^2 + \frac12 \wtilQ^2(1-\phi_1^2)^2 \right)dx.
\end{equation}
The states $(\pm 1,0)$ are global minima of $E[\phi_1, \phi_2]$: $0 = E[\pm 1,0] < E[\widetilde H,0] = \frac12 \int Q H'^2 +\frac14 \int Q(1-H^2)^2 =\frac65.$  Due to the dissipation of energy by dispersion, solutions of the system \eqref{eq:varELsystem} are expected to settle down to critical points of the potential energy. The energy is well-defined for the set of functions 
\[
\begin{aligned}
	\bold{E} := &~{} \Big\{\boldsymbol{\phi}=(\phi_1,\phi_2)\in (L^1_{\textup{loc}}(\R))^2 : \ \px \phi_1 \in L^2(\R), \  \wtilQ(1 - \phi_1^2) \in L^2(\R) \ , \ \phi_2 \in L^2(\R) \Big\}.
\end{aligned}
\]
To study the stability of $\widetilde\bfH$, we introduce the following metric structure. We consider the weighted Sobolev space
\[
H_0(\R) := \left\{ \phi \in L^1_{\textup{loc}}(\R) : \ \px\phi \in L^2(\R), ~ \wtilQ \phi \in L^2(\R) \right\}.
\]
which we endow with the Hilbert norm
\[
\|\phi\|_{H_0(\R)}^2 := \|\px\phi\|_{L^2(\R)}^2 + \|\wtilQ \phi\|_{L^2(\R)}^2.
\]
Due to the equivalence of norms in Lemma \ref{claim:Linfty}, the rough estimate $|\phi(x)| \leq |\phi(0)| + \|\px \phi\|_{L^2}|x|^{\frac12}$, and the polynomial decay of $\wtilQ$ in \eqref{tails}, we have that the energy space $\bold{E}$ appears as the subset of $H_0(\R)\times L^2(\R)$ given by
\[
\bold{E} =  \Big\{\boldsymbol{\phi}=(\phi_1,\phi_2)\in H_0(\R)\times L^2(\R) : \ \wtilQ(1 - \phi_1^2) \in L^2(\R) \Big\}.
\]
We endow the energy space with the metric structure given by
\begin{equation}\label{energy_norm}
	\|\boldsymbol{\phi}\|_{(H_0\times L^2)(\R)}^2 := \| \phi_1\|_{H_0(\R)}^2 +\| \phi_2\|_{L^2(\R)}^2.
\end{equation}
Notice that the energy norm $\|\cdot \|_{(H_0\times L^2)(\R)}^2$ need not be similar to the standard $H^1\times L^2$ norm. In particular, perturbations of the kink need not be necessarily bounded in space. By standard fixed-point arguments, the system \eqref{eq:varELsystem} is locally well-posed for arbitrary finite energy data; however, the global existence of solutions for initial data with small energy is not obvious. In what follows, we refers to \textit{global solution} of \eqref{eq:varELsystem} to a function $ \mathbf{\phi}\in C([0,\infty); \bold{E})$ that satisfies \eqref{eq:varELsystem} for all $t\geq 0$.

\subsection{Main results} In this work we shall address three main objectives. First, to analyze the long time evolution and stability of the Bizo\'n and Kahl \cite{BizKah16} 1D kink emerging in the setting of the Yang-Mills field in the extremal Reissner-Nordstr\"om black hole. Second, to describe the long time behavior of kinks in a non perturbative, inhomogeneous medium represented by a variable coefficients setting, with no restriction on the data except their perturbative character. Finally, we aim to describe the dynamics of a kink only presenting a polynomial tail.

Our main result establishes that, for globally defined perturbations of the kink $\widetilde{\bfH}$, stability in the energy space $\bold{E}$ (see \eqref{energy_norm}) implies asymptotic stability in a spatially localized energy norm.

\begin{theorem}\label{th:Main}
	There exists $\delta>0$  such that if a global solution $\boldsymbol{\phi} \in  \mathbf{E}$ of \eqref{eq:varELsystem} satisfies
	\begin{equation}\label{eq:vicinity}
	\sup_{t\geq 0} \left\|\boldsymbol{\phi}(t) - \widetilde{\bfH} \right\|_{(H_0\times L^2)(\R)}  < \delta,
	\end{equation}
	then for any $I$ bounded interval in $\mathbb R$,
	\begin{equation}\label{eq:AS}
		\lim_{t\to\infty}  \left\| \boldsymbol{\phi} (t) -\widetilde{\bfH} \right\|_{(H^1\times L^2)(I) } = 0\, .
	\end{equation}
\end{theorem}

Theorem \ref{th:Main} can be recast as the local asymptotic stability of the variable coefficients, unstable kink $H$. Compared with the classical $\phi^4$ model studied in \cite{KMM2017, KM2022, CucMae} through the use of virial identities, the $L^\infty$ norm of the perturbation is not \emph{globally in space} small in principle, meaning that nonlinear terms are as large as the linear ones: the contribution of nonlinear terms has to be measured equally with linear ones. 

\medskip

The case of kinks in variable coefficients scalar field models was first studied by Snelson in the $\phi^4$ case \cite{Snelson}, see also the recent results by Alammari and Snelson \cite{AS1,AS2} for general scalar field models around the zero solution. In this paper, Theorem \ref{th:Main} refers to the asymptotic stability of an unstable kink in a slowly decaying in space setting. In particular, the spectral theory of variable coefficients operators cannot be taken for granted, and it is independently performed in Section \ref{A:LINEAL-SPECTRAL-THEORY}.

Restricted to the constant coefficients case, kinks are better understood. Cuccagna \cite{Cuc} studied the stability of the $\phi^4$ kink in 3D using vector field methods. Komech and Kopylova \cite{KK1,KK2} established the asymptotic stability of kinks in highly degenerate scalar field theories under higher order weighted norms. Delort and Masmoudi \cite{DM} utilized Fourier analysis techniques to provide detailed asymptotics for odd perturbations of the kink up to times of order $O(\varepsilon^{-4})$, where $\varepsilon$ represents the size of the perturbation. It is worth noting that the analysis in \cite{KMM2017} was limited to odd data, and the stability in the general case remains an open question. In \cite{KMMV20}, a condition was proposed to describe the long-term dynamics of kink perturbations for any data in the energy space, encompassing many models of interest in Quantum Field Theory \cite{Lohe}, excluding the sine-Gordon and $\phi^4$ models. However, the modulation of kinks in terms of scaling and shifts in this scenario complicates computations. Cuccagna and Maeda introduced a new sufficient condition for asymptotic stability in the case of odd data \cite{CucMae}. 

Let us review some relevant works related to the Yang-Mills mathematical theory. Chen-Ning Yang and Robert Mills presented the first concepts of a gauge theory for non-abelian groups that could explain strong interactions in Physics \cite{YangMills54}. This constituted the beginning of the so-called Yang-Mills theory, present now in the foundations of the Standard Model, a theory that describes the interactions between fundamental particles. The global dynamics of a Yang-Mills field propagating in a 4-dimensional Minkowski spacetime is well-understood in the case of a smooth initial data \cite{EarMon82,Chris81}, as well as the global in time regularity in any globally hyperbolic 4-dimensional curved spacetime \cite{ChrSha97}. The hyperbolic energy critical case, where the instanton plays a threshold role, has been successfully addressed in a series of works \cite{OT1,OT2,OT3}.

Of particular interest is the comparison of the results presented in this paper with the energy critical equivariant reduction of the Yang-Mills model for a field $\phi=\phi(t,r)$ in 1+4 dimensions
\[
\partial_{t}^2 \phi - \partial_r^2 \phi -\frac1{r} \partial_r \phi - \frac{2}{r^2} (\phi -\phi^3) = 0, \quad t\in \mathbb R, r>0.
\]
The associated static solution (better known as the instanton) is explicit and given by $H(r)= \frac{1-r^2}{1+r^2}$. In this case, a precise stable blow up mechanism around the kink was showed in \cite{RR}, while other blow up rates are constructed in \cite{KST}. In this work, we construct an asymptotically stable manifold for $\wtilH$, but the understanding of a possible blow up mechanism outside this manifold remains an interesting open question. Conversely, our results open a path towards a better understanding of the (asymptotically) stable manifold for the equivariant Yang-Mills instanton $H(r)$, in the sense of giving a better description on the existence and characterization of this manifold.

\medskip

For the sake of completeness, and following the construction described in \cite{KMM19}, we provide an explicit description of a set of initial data leading to global solutions satisfying \eqref{eq:vicinity}. It turns out that, unlike other kinks \cite{HPW}, the linearized problem around $\wtilH$ has a strongly unstable direction \cite{BizKah16}. Let us consider a perturbation in \eqref{eq:varELsystem} over $\widetilde \bfH$ of the form $\boldsymbol{\phi} = \widetilde\bfH + \bfw$. Explicitly,
\[
\phi_1(t,x) = \wtilH(x) + w_1(t,x), \quad \phi_2(t,x) = w_2(t,x)\,.
\]
Then $\bfw$ satisfies the following system:
\begin{equation}\label{eq:BKsystem}
	\begin{cases}
		\partial_t w_1 = w_2 \\
		\partial_t w_2 = - Lw_1 - \wtilQ^2 ( 3\wtilH w_1^2 + w_1^3 ),
	\end{cases}
\end{equation}
where we have defined the linear operator
\begin{equation}\label{eq:L}
	Lw = -\partial_x^2 w + V(x) w, \quad \textup{with} \quad V(x) = 2\wtilQ^{2}(1- \wtilQ).
\end{equation}
Consequently, for the well-understanding of the problem we require to study the second order operator $L$. In Section \ref{A:LINEAL-SPECTRAL-THEORY} we will show that $L$ has an even eigenfunction $\phi_0(x)$ of unit norm, associated with the first simple and negative eigenvalue $-\mu_0^2 $ (numerically studied by Bizo\'n and Kahl in \cite{BizKah16}). Moreover, $\phi_0$ satisfies (Lemma \ref{lemma:Lproperties})
\begin{equation}\label{eq:properties-eigenvalL}
	L\phi_0 = -\mu_0^2\phi_0, \quad  \left| \partial_x^{k}\phi_0(x) \right| \lesssim e^{-\frac{\sqrt{2}}{2}\mu_0 |x|}, \ \ k =0, 1, 2.
\end{equation}
The negative eigenvalue of the linearized operator $L$ introduces exponentially stable and unstable modes for the dynamics in the neighborhood of the kink. Let
\begin{equation}\label{eq:pmY}
	\bold Y_{\pm} = \begin{pmatrix} \phi_0 \\ \pm \mu_0\phi_0 \end{pmatrix}, \qquad \bfZ_{\pm} = \begin{pmatrix} \phi_0 \\ \pm \mu_0^{-1}\phi_0 \end{pmatrix},
\end{equation}
and $\delta_0>0$, let $\calA_0$ be the manifold given by
\begin{equation}\label{eq:A0}
	\calA_0 = \big\{ \boldsymbol{\eps}
	\text{ such that }  \|\boldsymbol{\eps}\|_{(H_0\times L^2)(\R)}< \delta_0 \text{ and } \langle \boldsymbol{\eps}, \bfZ_+ \rangle = 0 \big\}.
\end{equation}
Notice that some work is required to ensure that $\langle \boldsymbol{\eps}, \bfZ_+ \rangle$ is well-defined, but \eqref{eq:pmY} and \eqref{eq:properties-eigenvalL} are sufficient to conclude. 

\begin{theorem}\label{th:Manifold}
There exist $C, \delta_0>0$ and a Lipschitz function $h:\calA_0\to \R$ with $h(0) = 0$ and $|h(\boldsymbol{\eps})|\leq C\|\boldsymbol{\eps}\|_{H_0\times L^2} ^{3/2}$, such that denoting
\begin{equation}\label{eq:calM}
\calM = \big\{ {\bf \wtilH} + \boldsymbol{\eps} + h(\boldsymbol{\eps})\bfY_+ \text{ with }\boldsymbol{\eps} \in \calA_0 \big\}
\end{equation}
the following holds:
\begin{enumerate}
	\item[(i)] If $(\phi, \pt\phi)(0)\in \calM$ then the solution $(\phi, \pt\phi)$ of \eqref{eq:varEL} with initial data $(\phi, \pt\phi)(0)$ is global and satisfies, for all $t\geq 0$,
	\begin{equation}\label{eq:cont initial data}
		\begin{aligned}
		\|\boldsymbol{\phi}(t) - \widetilde{\bfH} \|_{(H_0\times L^2)(\R)}  \leq  C \|\boldsymbol{\phi}(0) - \widetilde{\bfH} \|_{(H_0\times L^2)(\R)}.
		\end{aligned}
	\end{equation}
	\item[(ii)] If a global solution $\bfphi$ of \eqref{eq:varEL} satisfies, for all $t\geq0$,
	\[
	\|\boldsymbol{\phi}(t) - \widetilde{\bfH} \|_{(H_0\times L^2)(\R)}  \leq \frac{\delta_0}2, 
	\]
	then for all $t\geq0$, $(\phi, \partial_t \phi)(t)\in \calM$.
\end{enumerate}
\end{theorem}

Although it seems very similar to previous constructions done in \cite{KMM19,Mau23}, the proof of Theorem \ref{th:Manifold} requires important changes in the specific deep description of the manifold $ \calM$. We mention some of them in the following lines.

\subsection{Main difficulties} The proofs of Theorem \ref{th:Main} and \ref{th:Manifold} are mainly based on the previously published works \cite{KMM2017,KMM19,KMMV20,KM2022} whose main ingredient is the use of combined virial estimates to leverage the convergence of perturbations of the kink at large times. Despite the remarkable stability of this theory in many models, in this work we will require several improvements and/or extensions of this set of techniques due to the lack of important basic properties of the kink in the considered scalar field model, and that we proceed to explain now.

\medskip

{\it Lack of standard $L^\infty$ smallness}. Working with small 1D perturbations in the energy space $H^1\times L^2$ possesses several advantages, among them   the $L^\infty$ smallness that allows one in virial estimates to control quadratic and cubic nonlinear terms in terms of estimates for the linear ones. An important issue in this paper is related to the lack of suitable $L^\infty$ control on the perturbations. As a consequence of this fact, as far as we understand, nonlinear terms must be treated in estimates as elements with sizes as large as the linear ones. As an example, terms such as $\wtilQ^2 ( 3\wtilH w_1^2 + w_1^3 )$ in \eqref{eq:BKsystem} are as large as $Lw_1$. We have found a particular positivity structure in Bizo\'n-Kahl's problem, related to the quartic potential, and which becomes a key actor to either estimate nonlinearities jointly with linear terms as a whole, or to absorb them in terms of classical virial estimates. 

\medskip

{\it A degenerate energy}. Deeply related to the previous issue is the fact that the classical energy does not enjoy a natural coercivity structure as in standard kink problems. This is probably caused by the supercritical character of the problem, and it is both a fundamental and technical issue essentially saying that the second variation of the energy $E$ is in practice different to the bilinear operator represented by $L$, the latter being the case in classical scalar field models. We have found a correct representative for the energy around the kink $\wtilH$ for large scales, given by a modified linearization denoted $\widetilde L$ (see \eqref{mL}), an operator satisfying $\widetilde L <L$ (essentially strictly below $L$), under which the value of eigenvalues decrease, but an improved algebra appears: for example $\widetilde L\wtilH =0$. Additionally, $\widetilde L$ does not posses spectral gap, and coercivity estimates must be always placed in weighted spaces. Then, naturally $H_0$ becomes the correct space to describe the long time behavior.

\medskip

{\it Existence of a resonance}. Precisely, $\widetilde L$ is an operator with an ``$L^2$ threshold resonance'' at zero, with generalized eigenfunction $\wtilH$. This fact makes the decay analysis hard enough, since under $\langle \phi_0 , u \rangle =0$  one only has $\langle \widetilde L u , u \rangle \geq 0$, meaning that even in the energy space $\bold{E}$ the influence of the resonance is strong.  Even proving this last fact requires a delicate construction of solutions to the equation $\widetilde L\phi_1 = \phi_0$ and showing that $\langle \phi_1,\phi_0\rangle <0$. While doing this, we have realized two surprising findings: $\phi_1$ can be chosen even and in $L^2$ (despite $\widetilde L$ not having spectral gap), and $\phi_0$ is actually orthogonal to the full kernel of $\widetilde L$. 

Resonances induce natural weak instability directions and, as far as we know, have not been treated using virial methods. The reason is deeply related to the fact that local virial estimates ``feel'' resonances, even if they are outside the energy space. Additionally, resonances announce the existence of breathers, periodic in time solutions that contradict the asymptotic stability, for at least one possible nonlinearity in the model. This makes them complicated to handle with techniques only placed in the energy space. Here we propose a first direction to handling them for all times \emph{using just virial techniques}, namely for data in the energy space only. See also the works by Palacios and Pusateri \cite{PP} for an approach to resonances and asymptotic stability via mixed virial/distorted Fourier transform techniques in the case of nonlinear cubic Klein-Gordon up to exponentially large but finite time, and the recent work by Chen and Luhrmann on sine-Gordon considering the kink odd resonant mode in weighted Sobolev spaces \cite{ChenLu}. In our case, because of the resonance $\widetilde H$, orbital stability is not clear as in standard cases even under orthogonal conditions with respect to the negative eigenvalues (notice that shifts are not present here). Consequently, the presence of the resonance makes our setting more involved than the one studied in \cite{KMM19}. Indeed, we will show that at an initial time the manifold \eqref{eq:calM} has the particular structure
\[
(\phi, \pt\phi)(0) = (1+a(0)){\bf \widetilde H} + (\tilde u_1, u_2)(0)  + b_-(0)\bfY_- + h(\bfvareps)\bfY_+,
\]
where $(u_1, u_2)(0)$ are error terms, $a(0)$ is a new modulation term representing the resonant mode associated to $\widetilde L\wtilH =0$, and $\tilde u_1$ results from the decomposition of the error term $u_1$ into resonant and nonresonant terms. In principle, looking at the energy in \eqref{eq:energy} one realizes that there is no actual topological obstruction on the kink ${\bf \widetilde H}$ and $a(0)$ may be later growing in time destroying the orbital stability. Therefore, an important part of the proof will be devoted to show that the instability direction associated to the resonance stays bounded in time, and the manifold indeed exists. This being said, without using shift modulations. A new setting involving a careful choice of new orthogonalities in the decomposition of the stable manifold will be the first action towards a good control of the energy norm. Then, a second step will involve a suitable decomposition of the energy functional profiting of the fact that the model is quartic to get new positivity bounds, in the sense that roughly speaking
\[
\|u_2\|_2^2 + \|\tilde u_1\|_{H_0}^2+  a^2 \lesssim \|u_2\|_2^2 + \langle \widetilde L \tilde u_1, \tilde u_1\rangle + \int \wtilQ^2(u_1^2 + 2\wtilH u_1)^2 \lesssim \delta_0^2 .
\]
 This fact is also deeply related to the first point above, because nonlinear terms are as large as linear ones, and no actual control on the resonance amplitude is obtained without finding a hidden ``defocusing'' behavior. In other words, resonances may be handled via hidden positivities in cubic and quartic order terms. Putting all this together, it will allow us to ensure the boundedness and decay of $a(t)$, i.e., the control of the resonance modulation, and therefore the existence of a stable manifold. Finally, the asymptotic stability will be ensured by improved primal and dual estimates, where we have control of every good sign term (Propositions \ref{prop:virial I} and \ref{prop:virial II}). Indeed, we need to get track of good-sign weighted $L^2$ norms in both virial estimates, reducing to its minimal value bad sign terms, since we do not have full control on nonlinear terms. It will be the case that bad terms will have improved decay properties, allowing us to prove the convergence without the necessity of decomposing the dynamics into resonant and nonresonant parts. Consequently, the constructed manifold will satisfy convergence to zero locally in space (or in a subspace of $H_0$) as time tends to infinity, also implying the convergence of the resonant modulation.

\medskip

{\it No explicit kink solution}.  Another issue present in the considered model is the lack of an explicit representation for the kink $\wtilH$ that permits effective computations for spectral analysis and by consequence explicit control of virial estimates. In particular, this lack of explicit knowledge poses interesting challenges for the understanding of the associated point spectrum theory for $L$. By using well-chosen test functions, we have computed suitable estimates on the spectrum of $L$, its smallest eigenvalue (Lemma \ref{valor mu0}), and obtained suitable coercivity estimates by partial local estimates valid for each particular region of space. A particular issue to be mentioned is the one related to the so called ``transformed problem'', where the associated potential has no explicit representation at all. Section \ref{A:LINEAL-SPECTRAL-THEORY} provides a rigorous description of the functional setting related to this operator, that we believe could be used in other models with no explicit kinks. We emphasize that all our proofs do not use extended numerical computations to describe the spectral theory, except by some simple evaluations of certain explicit functions at some particular points, which are done with standard mathematical programs and enjoy great accuracy. An example of this type of numerical computation is to find the solutions of the equation $\alpha^{-1}(x)=1$, or the zeros/solutions of the equation $\wtilQ(x)=1$.

\medskip

{\it Lack of an exponential tail in the kink solution.} Previous works in the field  \cite{KMM2017,KMM19,KMMV20,KM2022,LuS1} consider a kink or soliton solution with an exponential convergence at infinity, representing in this case a quickly converging tail. In this work, this is not the case (see \ref{tails}) and only a slightly above the minimally sufficient (in terms of spectral theory) polynomial decay is present in our setting. This is in some sense equivalent to the degenerate setting $W'' =0$ at the spatial infinite limit of the Lohe's kink solutions \cite{Lohe}, which is indirectly mentioned but not treated in \cite{KMMV20} (special cases are some $\phi^8$ models with polynomial tail kinks).  The polynomial character of the kink $\wtilH$ imposes restrictions in several standard estimates, which are not satisfied now and which must to consider any possible gain in decay. This is for instance the case of coercivity estimate \eqref{eq:coercivity}, which is only valid if one imposes a strong weight of order at least $O(|x|^{-6})$. Following a series of estimates, we will track weighted estimates with weights as optimal as one can get. Examples as this one are present in many places in this paper (see e.g. \eqref{eq:N0est}, Lemma \ref{claim:Linfty}, to mention a few in the first part of the paper), leading to the introduction of several new estimates that must consider polynomially decaying functions.

\subsection{Related literature} We finish this introduction with some final comments on related results. An alternative perspective, equivalent to considering kinks under symmetry assumptions (essentially no shifts or Lorentz boosts), involves studying 1D nonlinear Klein-Gordon models with variable coefficients. Foundational works in 3D were conducted by Soffer and Weinstein \cite{SW2,SW3}, and scattering studies and dispersive decay include those by Lindblad and Soffer \cite{LS1,LS2,LS3}, Hayashi and Naumkin \cite{HN2,HN3,HN4}, Bambusi and Cuccagna \cite{Bam_Cucc}, Lindblad and Tao \cite{Lin_Tao}, and Lindblad et al. \cite{LLS,LLS2,LLSS}, among several other works. Recent enhancements include considerations of quadratic nonlinearities, exemplified by the work of Germain and Pusateri \cite{GP}, and related studies \cite{GPZ}. On the other hand, non-topological solitons in nonlinear Klein-Gordon models have been a focal point of research since the recent works on the description of the stable and unstable soliton manifold by Krieger-Nakanishi-Schlag \cite{KNS}, Nakanishi-Schlag \cite{NS}, alongside earlier results by Ibrahim, Masmoudi, and Nakanishi \cite{IMN}; see also former results in references therein. Subcritical dynamics around solitons have been extensively explored, particularly in the presence of at least one unstable mode, see details in \cite{BCS,KMM19,KMM17_Survey,GJ1,GJ2,LiLuhrmann22,KP,LP,LP2,LuS2}.

Another interesting comparison is related to the long time behavior in energy critical equivariant wave maps. Here a much more detailed description of the so-called soliton resolution conjecture is available, see e.g. \cite{DKMM,JL}. There is an interesting relation among these models, specially from the fact that the solutions $H_n$ in our case can be related to equivariant wave maps in different topological classes. There is probably a soliton resolution conjecture associated to our problem, as Bizo\'n has personally communicated to us. This comparison needs to be though in more detail because it is only weakly understood from a rigorous point of view. Several differences appear with the model under attention here, and probably the most relevant is the lack of fixed topological classes which makes the kink worked here more inclined to be destroyed by general perturbations. Additionally, the existence of a scaling symmetry is also relevant in the critical setting. In our case, such structure is not present, but it is weakly mimicked by the existence of the mild resonance. 

\subsection*{Organization of this paper} This paper is organized as follows. In Section \ref{Sect:1} we introduce preliminary estimates and concepts essential for the proof of Theorem \ref{th:Main}. Section \ref{Sect:3} introduces the first virial estimates. Section \ref{sec:second virial estimate} is concerned with dual virial estimates. Section \ref{sec:proof theorem 1} proves Theorem \ref{th:Main} and Section \ref{sec:proof_manifold} proves Theorem \ref{th:Manifold}. Next, Section \ref{A:LINEAL-SPECTRAL-THEORY} is devoted to the deep understanding of the operator $L$. Section \ref{B:POSITIVITY-POTENTIAL} proves the repulsivity of the associated virial operator.

\subsection*{Acknowledgments}  We would like to thank the referee for its careful reading of the manuscript, pointing out several inaccuracies and typos from a first version of this manuscript, and helping to improve the quality of this work.

\section{Preliminaries}\label{Sect:1}

\emph{Notation}. The standard $\lesssim $ symbol means that there exists $C>0$ such that $a(x) \leq C b(x)$, $C$ independent of $x$.

\medskip

We shall start with some basic properties about the function $\alpha$ defined in \eqref{eq:alpha}, and the modified soliton $\widetilde{Q}$ in \eqref{tilde Q tilde H}, deeply involved in the spectral analysis of $L$.
\begin{lemma}
	The function $\alpha(x)$ is strictly monotone, bijective. Moreover, if $\alpha^{-1}$ denotes the inverse of $\alpha$,
	\begin{equation}\label{eq:dalpha}
		\px\alpha(x) = Q^{-1}(x), \quad \px\alpha^{-1}(x) = \wtilQ(x),
	\end{equation}
	and
	\begin{equation}\label{derivada}
	\partial_x \wtilQ (x) = - \wtilQ^2(x) \wtilH(x), \quad \partial_x^2 \wtilQ (x) =2\wtilQ^3(x) - \frac53\wtilQ^4(x).
	\end{equation}
	
\end{lemma}
\begin{proof}
	By direct computation one has $\alpha'(x) = \frac13 (\cosh x + 1) = \frac23 \cosh^2\left(\frac{x}{2}\right) = \frac1{Q(x)}$,  proving that $\alpha(x)$ is strictly monotone and bijective, since $\alpha'(x)$ grows with $x$. For the inverse of $\alpha$ we have
	\[
	(\alpha^{-1})'(x) = \frac{1}{\alpha'(\alpha^{-1}(x))} = Q(\alpha^{-1}(x)) = \wtilQ(x).
	\]
	This ends the proof of \eqref{eq:dalpha}. In order to prove \eqref{derivada}, notice that from \eqref{Q} one has $Q'(x) = -\frac32  \sech^2 \left( \frac{x}2\right) \tanh \left( \frac{x}2\right)=-Q(x)H(x)$. Then,  using \eqref{tilde Q tilde H},
	\begin{equation}\label{derivada de tQ}
	\partial_x \wtilQ (x)= Q'(\alpha^{-1}(x))(\alpha^{-1})'(x) = - \wtilQ^2(x) \wtilH(x).
	\end{equation}
	Finally, since $\wtilH'(x) =\frac13 \wtilQ^2(x)$ and $\wtilH^2=1-\frac23 \wtilQ$,
	\[
	\begin{aligned}
	\partial_x^2 \wtilQ (x) = &~{} -2 \wtilQ(x)\wtilQ'(x) \wtilH(x) -  \wtilQ^2(x) \wtilH'(x)  \\
	=&~{}  2\wtilQ^3(x) \wtilH^2(x)-\frac13 \wtilQ^4(x) = 2\wtilQ^3(x) - \frac53\wtilQ^4(x). 
	\end{aligned}
	\]
The proof is complete.
\end{proof}

\begin{lemma}\label{lemma:estimation}
	The functions $\alpha^{-1}(x)$, $\wtilH(x)$ and $\wtilQ(x)$ are odd, odd and even, respectively, and they have the following asymptotic descriptions.
	
	For $|x| \ll 1$,
	\begin{equation}\label{eqn:equivalencias2}
	\begin{gathered}
		\alpha^{-1}(x) = \frac32 x + \calO(x^2), \quad \wtilQ(x) = \frac32 - \frac{27}{32}x^2 + \calO(x^4), \\
		 \wtilH(x)= \frac34x - \frac{9}{32}x^3 + \calO(x^5).
	\end{gathered}
	\end{equation}
	
	For $|x| \gg 1$, we have the limits
	\begin{equation}\label{eqn:equivalencias}
	\begin{gathered}
		\lim_{x\to \pm \infty} \frac{|\alpha^{-1}(x)|}{\ln(|x|)} =1, \quad \lim_{x\to \pm \infty} (1 + |x|)\wtilQ(x) = 1, \\
		\lim_{x\to \pm \infty} (1 + |x|)|\wtilH(x) \mp 1| =\frac13,
	\end{gathered}
	\end{equation}
	and for every $\epsilon\in(0,1)$ there exists $M_\epsilon>0$ such that
	\begin{equation}\label{eq:lowerboundQ}
		\epsilon \wtilQ(x) \leq \frac{1}{|x|} \qquad \text{for all } x \geq M_\epsilon.
	\end{equation}
	Even more, the integral $\int \wtilQ^{1+\varepsilon}dx$ is finite for any $\varepsilon>0$.
\end{lemma}
\begin{proof}
Let us first prove \eqref{eqn:equivalencias}. Recall that $\widetilde Q(x)= Q(\alpha^{-1}(x))$. Employing the fact that $x=\alpha(y)$ is continuous bijective, and goes to $\pm \infty$ when $y \to \pm \infty$, as well as \eqref{eq:alpha}, we have that
\begin{align*}
	\lim_{x\to\pm\infty} x \wtilQ(x) &= \lim_{y \to \pm \infty} \alpha(y)Q(y)  = \lim_{y \to \pm \infty} \frac12 \frac{\sinh y + y}{ \cosh^2\left(\frac{y}{2}\right)  } \\
	& = \lim_{y \to \pm \infty} \frac{\cosh y + 1}{\sinh y} =\pm 1,
\end{align*}
where in the second line we have used a simple L'H\^opital's rule.
On the other hand, using \eqref{eq:dalpha},
\begin{align*}
	\lim_{x\to\pm\infty} \frac{|\alpha^{-1}(x)|}{\ln(|x|)} = \lim_{x\to\pm\infty} |x|\wtilQ(x) = 1.
\end{align*}
This proves the first limit of \eqref{eqn:equivalencias}, and $\wtilQ \lesssim |x|^{-1}$.

\medskip

Now we restrict our analysis of $\wtilQ$, by parity, to the positive real numbers. From definition \eqref{eq:alpha} we obtain for $x>0$, $e^{x} = e^{-x} - 2x + 6\alpha(x).$ Employing this,
\[
\sech^2\left(\frac{x}{2}\right) = \frac{1}{\cosh^2\left(\frac{x}{2} \right) } = \frac{4}{e^{x} + 2 + e^{-x}} = \frac{4}{3e^{-x} + 2 - 2x + 6\alpha(x) }.
\]
Replacing in \eqref{tilde Q tilde H}, and using that $|\alpha^{-1}| \sim \frac12\ln(|x|)$, we have for any $x>0$
\[
\begin{aligned}
	\wtilQ(x) = &~{} \frac{6}{3e^{-\alpha^{-1}(x)} + 2 - 2\alpha^{-1}(x) + 6x} \leq   \frac{3}{1 - \alpha^{-1}(x) + 3x}.
\end{aligned}
\]
Analogously, since $0 < e^{-\alpha^{-1}(x)} \leq 1$ for $x\geq 0$, the denominator is bounded above by $5 - 2\alpha^{-1}(x) + 6x$, hence
\[
\begin{aligned}
	\wtilQ(x) \geq &~{}  \frac{6}{5 - 2\alpha^{-1}(x) + 6x}=\frac{3}{\tfrac{5}{2} - \alpha^{-1}(x) + 3x}\geq \frac{3}{\tfrac{5}{2} + 3x}.
\end{aligned}
\]
Therefore $\lim_{x\to +\infty} (1+x)\wtilQ(x)  =1.$ The case $x\to -\infty$ is obtained by parity,
which proves \eqref{eqn:equivalencias} in the case of $\wtilQ$. Finally, we consider the case of $\wtilH(x) =H(\alpha^{-1}(x))$. We have
\[
\lim_{x\to \pm \infty} x (\wtilH(x) \mp 1)  = \lim_{y\to \pm \infty} \alpha(y) \left( \tanh\left( \frac{y}2\right) \mp 1\right) =-\frac13.
\]
To prove \eqref{eq:lowerboundQ}, we restrict again the analysis by parity to the positive real numbers. Let $\epsilon\in (0,1)$. Notice that the statement is equivalent to prove that there exists $M_\epsilon>0$ such that
\[
\epsilon Q(y) \alpha(y) \leq 1
\]
for all $y\geq \alpha^{-1}(M_\epsilon)$. By the definitions \eqref{Q} and \eqref{eq:alpha} we have
\[
Q(y) \alpha(y) = \frac{e^{y} - e^{-y} + 2y}{e^{y} + e^{-y} + 2} = 1 + 2\frac{y-1-e^{-y}}{e^{y}+e^{-y}+2}.
\]
For $y\geq 2$ the last term can be bounded easily as follows
\[
Q(y) \alpha(y) \leq 1 + 2ye^{-y} :=1 + g(y).
\]
Since $g$ is a strictly positive and decays exponentially for all $y>2$, there exist $M_0>0$ such that
\[
1+g(y) \leq \epsilon^{-1}
\]
for all $y\geq M_0$. Then it suffices to take $M_\epsilon = \alpha(M_0)$ to conclude \eqref{eq:lowerboundQ}.

Now we prove \eqref{eqn:equivalencias2}. The proof is based in a simple Taylor expansion in second and fourth order around $x=0$.
\begin{align*}
	\alpha^{-1}(x) = \alpha^{-1}(0) + \px \alpha^{-1}(0) x + \calO(x^2) = \frac32 x + \calO(x^2). 
\end{align*}
Also,
\begin{align*}
	\wtilQ(x) =&~{} \wtilQ(0) + \wtilQ'(0)x + \frac12 \wtilQ''(0) x^2 + \frac16 \wtilQ'''(0)x^3 + \calO(x^4) \\
	=&~{} \frac32 - \frac{27}{32}x^2 + \calO(x^4).
\end{align*}
and
\begin{align*}
	\wtilH(x) &= \wtilH(0) + \wtilH'(0)x + \frac12 \wtilH''(0) x^2 + \frac16 \wtilH'''(0)x^3 +\frac1{24} \wtilH''''(0)x^4 + \calO(x^5) \\
	& = \frac34x - \frac{9}{32}x^3 + \calO(x^5).
\end{align*}
In the previous expansions we have used that $ \wtilQ'(x) = -\wtilQ^2(x)\wtilH(x)$, $\wtilQ''(x) = 2\wtilQ^3(x) - \frac{5}{3}\wtilQ^4(x)$ (see \eqref{derivada}), $\wtilH'(x) =\frac13 \wtilQ^2(x) $, $\wtilH''(x)= -\frac23\wtilQ^3 (x) \wtilH(x) $, and $\wtilH'''(x) = 2\wtilQ^4 \wtilH^2(x)  -\frac29\wtilQ^5 $, and that $\wtilQ$ is even and $\wtilH$ is odd. Finally, by \eqref{eq:dalpha} we have $\int \wtilQ^{1+\varepsilon}(x)dx = \int Q^\varepsilon(s)ds <+\infty.$
\end{proof}

\subsection{Expansion of the conserved energy around the kink.}

We have $\phi_1(t,x) = \wtilH + \bar w_1(t,x)$, $\phi_2(t,x) = \bar w_2(t,x)$, and
	\[
	(1 - \phi_1^2)^2 = (1-\wtilH^2)^2 - 2(1-\wtilH^2)(2\wtilH \bar w_1 + \bar w_1^2) + (2\wtilH \bar w_1 + \bar w_1^2)^2.
	\]
Replacing in \eqref{eq:energy}, and using that $\wtilH$ is the static solution of \eqref{eq:varEL}, we obtain
	\[
	\begin{aligned}
		E[\phi_1, \phi_2] = &~{} \int  \left( \frac12 \phi_2^2 + \frac12 (\partial_x \phi_1)^2 + \frac14 \wtilQ^2(1-\phi_1^2)^2 \right)dx \\
		= &~{} E[\wtilH, 0] + \frac12\int \bar w_2^2 + \frac12\int (\px \bar w_1)^2 - \int \wtilH'' \bar w_1 \\
		&~{} - \frac12\int \wtilQ^2(1-\wtilH^2)(2\wtilH \bar w_1 + \bar w_1^2) +  \frac14\int \wtilQ^2(2\wtilH \bar w_1 + \bar w_1^2)^2 \\
		= &~{} E[\wtilH, 0] + \frac12\int \bar w_2^2 + \frac12\int \bar w_1 \left( -\px^2 \bar w_1 + 2\wtilQ^2(1-\wtilQ)\bar w_1 \right) \\
		&~{} + \frac14\int \wtilQ^2 \left( 4\wtilH \bar w_1^3 + \bar w_1^4 \right).
	\end{aligned}
	\]
	Therefore,
	\begin{equation}\label{mL0}
	\begin{aligned}
		2\{ E(\phi_1, \phi_2) - E(\wtilH, 0)\} =&~{} \int \bar w_2^2 + \langle L \bar w_1, \bar w_1\rangle +  \frac12\int \wtilQ^2(4\wtilH \bar w_1^3 + \bar w_1^4).
	\end{aligned}
	\end{equation}

\section{Virial estimate at large scale}\label{Sect:3}
The first step is to consider a small perturbation of the modified kink $(\wtilH, 0)$. In what follows we describe this decomposition, introduce some notation, and develop a first virial estimate.

\subsection{Decomposition of the solution in a vicinity of the kink}\label{subsec:decomposition}
Let $(\phi, \partial_t \phi)$ be a solution of \eqref{eq:varEL} satisfying \eqref{eq:vicinity} for some $\delta > 0$. Let $(\mu_0,\phi_0)$ be given in \eqref{eq:properties-eigenvalL}. Using $\bold Y_+$ from \eqref{eq:pmY}, we decompose $(\phi, \partial_t \phi)$ as follows
\begin{equation}\label{eq:decomposition}
	\begin{cases}
		\phi(t, x) - \wtilH = a_1(t)\phi_0(x) + u_1(t,x) \\
		\hspace{0.5cm} \partial_t \phi(t, x) = \mu_0 a_2(t) \phi_0(x) + u_2(t,x),
	\end{cases}
\end{equation}
where we define (see \eqref{eq:properties-eigenvalL})
\[
\begin{aligned}
a_1(t) = &\langle \phi(t) - \wtilH, \phi_0\rangle = -\frac{1}{\mu_0^2}\langle \phi(t) - \wtilH, L[\phi_0]\rangle ,\\
a_2(t) = & \frac{1}{\mu_0}\langle \partial_t \phi(t), \phi_0\rangle = -\frac{1}{\mu_0^3}\langle \partial_t \phi(t), L[\phi_0]\rangle,
\end{aligned} 
\]
such that
\begin{equation}\label{eq:perp-conditions}
	\langle u_1(t), \phi_0\rangle = 0 = \langle u_2(t), \phi_0\rangle.
\end{equation}
Additionally, we set the variables
\begin{equation}\label{eq:b1b2}
	b_+ = \frac12 (a_1 + a_2),  \quad b_- = \frac12 (a_1 - a_2).
\end{equation}

\begin{lemma}\label{lema_energia_separacion}
Under \eqref{eq:vicinity}, there exists $C>0$ fixed such that one has, for all $t\in \R_+$,
\begin{equation}\label{eq:bound}
\begin{gathered}
\|u_1(t)\|_{H_0} + \|u_2(t)\|_{L^2} + |a_1(t)| + |a_2(t)| + |b_+(t)| + |b_-(t)|
 \leq C\delta.
\end{gathered}
\end{equation}
\end{lemma}

\begin{proof}
In what follows, we will require the stability hypothesis \eqref{eq:vicinity}, and the decomposition \eqref{eq:decomposition}. First, using \eqref{eq:perp-conditions} we have
\begin{equation}\label{la clave 0}
\begin{aligned}
\|\phi_2\|_{L^2}^2 = & ~{} \mu_0^2 |a_2|^2 \|\phi_0\|_{L^2}^2 + \mu_0 a_2\langle u_2(t),\phi_0\rangle + \|u_2(t)\|_{L^2}^2 \\
= &~{} \mu_0^2 |a_2|^2 \|\phi_0\|_{L^2}^2 + \|u_2(t)\|_{L^2}^2 \leq \delta^2.
\end{aligned}
\end{equation}
This implies that $|a_2|, \|u_2(t)\|_{L^2} \lesssim \delta$. Let $R>0$ be a large number. Since $\|\partial_x(\phi_1 - \widetilde H) \|_{L^2(\R)}^2+\|\wtilQ(\phi_1 - \widetilde H)\|_{L^2(\R)}^2 \leq \delta^2$, one has $\int_{-R}^R (a_1 \phi_0 +   u_1 )^2  \lesssim R^2 \delta^2 $, and therefore
\[
a_1^2 +  \int_{-R}^R   u_1^2  \le CR^2 \delta^2 +  C|a_1|\left| \int_{-R}^R \phi_0 u_1 \right|. 
\]
Since $\langle u_1(t), \phi_0\rangle = 0$ and \eqref{eq:properties-eigenvalL} holds, one has 
\[
\begin{aligned}
& C |a_1| \left| \int_{-R}^R \phi_0 u_1 \right| \\
&~{}\le    C |a_1|\|\wtilQ^{-1}\phi_0\|_{L^2(|x|\geq R)} \|\wtilQ(a_1\phi_0 + u_1)\|_{L^2(|x|\geq R)} + C|a_1|^2 \|\phi_0\|_{L^2(|x|\geq R)}^2 \\
&~{} \le   \delta^2 + C e^{-2c_0 R}a_1^2.
\end{aligned}
\]
Consequently, fixing $R_0$ large,
\[
|a_1(t)| \le C\delta,  \quad |b_+(t)| + |b_-(t)| \leq C\delta,
\]
and for all $R>R_0$
\[
 \| u_1(t)\|_{L^2(-R,R)} \le  CR^2\delta.
\]
Now, using that $\|\wtilQ(a_1 \phi_0 + u_1)\|_{L^2(\R)} \leq \delta$, we obtain $\|\wtilQ u_1\|_{L^2(\R)} \leq C\delta$. Finally,
since $\|a_1 \phi_0' + \partial_x u_1 \|_{L^2(\R)} \leq \delta$, we arrive to $\| \partial_x u_1 \|_{L^2(\R)} \leq C\delta$.
\end{proof}

\begin{lemma}\label{claim:Linfty}
	For all $p\in [2, \infty]$ one has $\|\wtilQ^{\frac{p+2}{2p}}u\|_{L^p} \leq \sqrt2\|u\|_{H_0}$. In particular, $\|\cdot\|_{H_0}$ is equivalent to the norm
	\begin{equation}\label{norm_cubic}
		\|u\|^2:= \|\px u\|_{L^2(\R)}^2 + \|\wtilQ^{\frac32} u\|_{L^2(\R)}^2.
	\end{equation}
\end{lemma}
\begin{proof}
	Defining $u(x)=v(\alpha^{-1}(x))$, we obtain that it holds 
	\[
	\partial_x u(x)= \partial_y v(\alpha^{-1}(x))  Q(\alpha^{-1}(x)).
	\] 
	Therefore, applying a change of variable
	\[
	\int(\px u)^2  = \int Q(\partial_y v)^2, \quad \int \wtilQ^{k} u^2 = \int Q^{k-1} v^2, \quad \hbox{ for } k=2,3.
	\]
	Now, defining $g = Q^{1/2}v$, one has $Q^{1/2}\partial_y v = \partial_y g + \frac12 H g$. Replacing and integrating by parts, we get

	\[
	\begin{aligned}
	 \|u\|_{H_0}^2 =&~{} \int \left(  (\partial_y g)^2 + Hg\partial_y g + \frac14 H^2 g^2  +  g^2 \right) \\
	=&~{}\int (\partial_y g)^2 + \left(\frac54 -  \frac13 Q \right)g^2
	\geq \frac12\|g\|_{H^1}^2.
	\end{aligned}
	\]
	Since $\|\wtilQ^{\frac{p+2}{2p}}u\|_{L^p} = \|g\|_{L^p} \leq \|g\|_{H^1}$ for all $p\in [2, \infty]$ (by the standard Sobolev embedding $H^1(\R)\hookrightarrow L^p(\R)$ for $p\geq 2$), we obtain the first result.
	Using that $\wtilQ$ is bounded we have $\|u\| \lesssim \|u\|_{H_0}$. Next, applying the change of variable in \eqref{norm_cubic} and computing we get
	\[
	\frac12\int(\px u)^2 + \int \wtilQ^3 u^2  = \frac12 \int (\partial_y g)^2 + \left(\frac18 + \frac56 Q \right)g^2 \geq \frac18 \|g\|_{H^1}^2.
	\]
	From the Sobolev embedding for $p=2$ we have $ \|\wtilQ u\|_{L^2(\R)} \leq \|g\|_{H^1}$. This implies $\|u\|_{H_0}\lesssim \|u\|$.
\end{proof}

Using \eqref{eq:BKsystem}, \eqref{eq:properties-eigenvalL} and \eqref{eq:decomposition}, we obtain that $(a_1, a_2)$ satisfies the following differential system
\begin{equation}\label{eq:a1a2}
\ba
& 	\begin{cases}
		\dot{a}_1(t) = \mu_0 a_2(t)\\[0.1cm]
		\dot{a}_2(t) = \mu_0 a_1(t) - \dfrac{N_0}{\mu_0},
	\end{cases}
	\qquad	 \text{or equivalently } 
\\
& 	\begin{cases}
		\dot{b}_+(t) = \mu_0 b_+(t) - \dfrac{N_0}{2\mu_0}\\[0.3cm]
		\dot{b}_-(t) = -\mu_0 b_-(t) + \dfrac{N_0}{2\mu_0},
	\end{cases}
\ea
\end{equation}
where
\begin{equation}\label{eq:N}
	N =  \wtilQ^2 \left( 3\wtilH(a_1\phi_0 + u_1)^2 + (a_1\phi_0 + u_1)^3 \right),
\end{equation}
and
\begin{equation}\label{eq:Np}
	N^{\perp} = N-N_0 \phi_0, \quad \textup{and}\quad N_0 = \langle N,\phi_0\rangle.
\end{equation}
Then, $(u_1, u_2)$ satisfies the following system
\begin{equation}\label{eq:system-perturbated}
	\begin{cases}
		\dot{u}_1 = u_2 \\[0.1cm]
		\dot u_2 = - Lu_1 - N^{\perp}.
	\end{cases}
\end{equation}

\subsection{Local well-posedness in a neighborhood of the kink} Let $\delta >0$ and $T>0$ small enough to be chosen. We consider an initial data $ \mathbf{\phi}(0)\in \mathbf{E}$ such that
\begin{equation}
	\|\mathbf{\phi}(0) - \widetilde{\bfH} \|_{(H_0\times L^2)(\R)} \leq \delta.
\end{equation}
We decompose \eqref{eq:varELsystem} around $\widetilde \bfH$ in the form $\mathbf{\phi}(t) = \widetilde\bfH + \bfw(t)$ where
\[
\phi_1(t,x) = \wtilH(x) + w_1(t,x), \quad \phi_2(t,x) = w_2(t,x).
\]
Then we are reduced to solve
\begin{equation}\label{eq:local system}
	\begin{cases}
		\partial_t w_1 = w_2 \\
		\partial_t w_2 = \px^2 w_1 - F(t,x,w_1),
	\end{cases}
\end{equation}
where $F(t, x, w_1) =  2\wtilQ^2(1-\wtilQ)w_1 + \wtilQ^2 ( 3\wtilH w_1^2 + w_1^3 ).$ Invoking Lemma \ref{claim:Linfty}, we will solve this model in the space in $H_0 \times L^2$. If we denote by $S(t)(\bfw_0)$ the solution to the linear wave equation on $[-T,T]\times \R$, thanks to Lemma \ref{lemma:estimation} one can prove that $(S(t))_{t\in [-T,T]}$ defines a strongly continuous group of contractions in $H_0\times L^2$. In addition, there exists $C>0$ such that for any $w_1$, $\tilw_1$, if $\|\wtilQ^{1/2}w_1\|_{L^\infty}\leq 1$ and $\|\wtilQ^{1/2}\tilw_1\|_{L^\infty}\leq 1$ then
\[
|F(t, x, w_1) - F(t, x, \tilw_1)| \leq  C\wtilQ|w_1 - \tilw_1|.
\]
By Lemma \ref{claim:Linfty} and standard arguments, for $T$ and $\delta$ small enough, there exists a local in time solution $(w_1, w_2)$ of \eqref{eq:local system} in $H_0 \times L^2$. In this paper we will only work with the above notion of solution $\mathbf{\phi} = (\phi_1, \phi_2)$ of \eqref{eq:varELsystem}.

\subsection{Notation for virial argument}
In this paper, the notation $F \lesssim G$ means that $F \leq CG$ for some constant $C>0$ independent of F and G. Unless otherwise indicated, the implicit constant $C > 0$ is supposed to be independent of the parameters $A$, $B$, $\gamma$ and $\delta$ introduced below. As in \cite{LiLuhrmann22,KM2022}, it is convenient to define a \emph{modified space} $\calY$ of smooth functions $f:\R\to\R$ with the property that for any $k\geq 0$, there exists a constant $C_k>0$ such that
\[ |f^{(k)}(x)|\leq C_k \wtilQ(x)^3 \quad \text{for all } x \in \R.
\]
It is important to stress that $\wtilQ$ and $V$ in \eqref{eq:L} have only polynomial decay, consequently the definitions of $\mathcal Y$ and the virial type functions $\zeta$ need some care in our case. Note for example that $\wtilQ, h_0', V \in \calY$.

Let $\chi\in C^{\infty}_c(\mathbb{R})$ be a smooth even function satisfying
\begin{equation}\label{eq:chi}
	\chi(x) = 1 \textup{ for } |x|\leq 1, \quad \chi(x) = 0 \textup{ for } |x|\geq 2, \quad \chi'(x)\leq 0 \textup{ for } x\geq 0.
\end{equation}
For $A>0$, we define the function $\zeta_A$ and $\phi_A$ as follows
\begin{equation}\label{eq:zeta}
	\begin{aligned}\displaystyle
	&	\zeta_A^2(x) = \exp\left(-\dfrac{1}{A}|\alpha^{-1}(x)|(1-\chi(x)) \right), \quad \varphi_A(x) = \int_0^x \wtilQ\zeta_A^2(y)dy, \ x\in\R.
	\end{aligned}
\end{equation}
Moreover, we introduce the weight function
\begin{equation}\label{def:sigmaA}
	\sigma_A(x) = \sech\left(\frac1A \alpha^{-1}(x)\right).
\end{equation}
Notice that $\zeta_A\lesssim \sigma_A \lesssim \zeta_A$. Also, $\varphi_A' \sim  \wtilQ  \sigma_A^2.$ For $B>0$, we also define
\begin{equation}\label{eq:virial II notation}
	\begin{gathered}
		\zeta_B^2(x) =  \exp\left(-\dfrac{1}{B}|\alpha^{-1}(x)|(1-\chi(x)) \right), \quad \varphi_B(x) = \int_0^x \wtilQ\zeta_B^2(y)dy, \ x\in\R, \\[0.1cm]
		\tilde\chi_A(x)=\chi\left(\frac{\alpha^{-1}(x)}{A}\right), \quad \psi_{A,B}(x) = \tilde\chi_A^2(x)\varphi_B(x), \\
		\tilde\chi_B(x)=\chi\left(\frac{\alpha^{-1}(x)}{B^2}\right).
	\end{gathered}
\end{equation}
These functions will be used in two distinct virial arguments to prove Proposition \ref{prop:virial I} and Proposition \ref{prop:virial II} with different scales
\begin{equation}\label{eq:scales}
	1 \ll B\lll A.
\end{equation}
The choice of the switch function $\varphi_A$ is specifically adapted to the decay rate of the potential of the linear operator in \eqref{eq:BKsystem} and \eqref{eq:transformed system}.
We denote by $^{\sim}$ the composition with $\alpha^{-1}$ (i.e., $\wtilf(x) = (f \circ \alpha^{-1})(x)$).

\subsection{Virial estimate at large scale}\label{subsec:virial I}

Following \cite{KMM19}, and having in mind \eqref{momentum_law} in our new coordinates, we introduce the time dependent virial functional $\calI(t)$ defined by
\begin{equation}\label{eq:I}
	\mathcal{I} = \int \left( \varphi_A \partial_x u_1 + \frac12 \varphi_A' u_1 \right) u_2,
\end{equation}
and introduce the variables
\begin{equation}\label{tildeu}
	w_i = \zeta_A u_i,\quad  i=1,2.
\end{equation}
Here, as in \cite{KMM19}, $(w_1,w_2)$ represent a localized version of $(u_1, u_2)$ at scale $A$.

\begin{proposition}\label{prop:virial I}
	There exist $C_0, C > 0$ and $\delta_1 > 0$ such that for any $0<\delta \leq \delta_1$, the following holds. Fix 
	\begin{equation}\label{eq:choiceA}
		A=\delta^{-1/4}
	\end{equation}
	and assume that for all $t\geq 0$, \eqref{eq:bound} holds. Then for all $t\geq 0$, the functional $\mathcal I$ in \eqref{eq:I} satisfies the estimate
	\begin{equation}\label{eq:1virial}
		\frac{d}{dt}\mathcal{I} \leq - \frac12 C_0\int \wtilQ [(\partial_x w_1)^2 + \wtilQ^2 w_1^2] + C \int \wtilQ^7 u_1^2 + C |a_1|^4.
	\end{equation}
\end{proposition}
\begin{remark}
	Estimate \eqref{eq:1virial} does not involve any type of spectral analysis. Its purpose is to give weighted control of $(u_1,\px u_1)$ on a large scale $A$ in terms of a weighted and localized $H^1$ norm of $u_1$.
\end{remark}
The rest of this section is devoted to the proof of Proposition \ref{prop:virial I}. We start with the following basic lemma.

\begin{lemma}\label{lemma:virial I linear}
	Let $(u_1,u_2) \in H^1(\R)\times L^2(\R)$ be a solution of \eqref{eq:system-perturbated}. Consider $\varphi_A = \varphi_A(x)$ the smooth bounded function defined by \eqref{eq:zeta}. Then
	\begin{equation}\label{eq:dIdt}
		\frac{d}{dt}\mathcal{I} = I_{L}(t) - \int \left(\varphi_A \partial_x u_1 + \frac12 \varphi_A' u_1 \right)N^{\perp},
	\end{equation}
	where $I_L$ denotes the linear contribution, given by
		\begin{equation}\label{df:Ilinear}
			I_{L}(t):= - \int \varphi_A' (\partial_x u_1)^2 + \frac14 \int \varphi_A''' u_1^2 + \frac12 \int \varphi_A V' u_1^2.
		\end{equation}
\end{lemma}
\begin{proof}
	 The proof relies on direct substitution and integration by parts. The novelty lies in the presence of the algebraic weight in the virial functional. For clarity of exposure, we present the detailed computations.
	
	We define the integrals
	\begin{equation*}
		\mathcal{I}_1(t) = \int \varphi_A u_2 \partial_x u_1, \qquad \mathcal{I}_2(t) = \frac12 \int \varphi_A' u_1 u_2.
	\end{equation*}
	Differentiating $\calI_1$ with respect to time and using \eqref{eq:system-perturbated},
	\begin{align*}
		\frac{d}{dt}\mathcal{I}_1(t) &= \int \varphi_A (\dot u_2 \partial_x u_1 +  u_2 \partial_x \dot u_1)  \\
		&= - \int \varphi_A \left(L[u_1] + N^{\perp} \right) \partial_x u_1 +   \int \varphi_A u_2 \partial_x u_2\\
		&= - \int \varphi_A L[u_1] \partial_x u_1 - \int \varphi_A \px u_1 N^{\perp} - \frac12 \int \varphi_A' u_2^2.
	\end{align*}
	For the first integral on the RHS,
	\begin{align*}
		\int \varphi_A L[u_1] \partial_x u_1 =&~{} \int \varphi_A (-\partial_x^2 u_1 + V u_1) \partial_x u_1 \\
		=&~{} - \frac12 \int \varphi_A \px(\partial_x u_1)^2  + \frac12 \int \varphi_A V \px u_1^2 \\
		=&~{} \frac12 \int \varphi_A' (\partial_x u_1)^2 - \frac12 \int \varphi_A' V u_1^2 - \frac12 \int \varphi_A V' u_1^2.
	\end{align*}
	Then, replacing we obtain
	\begin{align}\label{eq:dI1dt}
		\begin{split}
			\frac{d}{dt}\mathcal{I}_1 &= - \frac12 \int \varphi_A' (\partial_x u_1)^2 + \frac12 \int \varphi_A' V u_1^2 + \frac12 \int \varphi_A V' u_1^2 - \frac12 \int  \varphi_A' u_2^2 - \int \varphi_A \px u_1 N^{\perp}.
		\end{split}
	\end{align}
	Now, for the second virial term $\calI_2$ we proceed analogously. Taking time derivative and using \eqref{eq:system-perturbated}:
	\begin{align*}
		\frac{d}{dt}\mathcal{I}_2 &= \frac12 \int \varphi_A' (\dot u_1 u_2 + u_1 \dot u_2) = \frac12 \int \varphi_A' u_2^2 - \frac12 \int \varphi_A' u_1 \left( L[u_1] + N^{\perp} \right) \\
		&= \frac12 \int \varphi_A' u_2^2 - \frac12 \int \varphi_A' u_1 L[u_1] - \frac12 \int \varphi_A' u_1 N^{\perp}.
	\end{align*}
	For the second integral above we have
	\begin{align*}
		\int \varphi_A' u_1 L[u_1] &= \int \varphi_A' u_1 (- \partial_x^2 u_1 + Vu_1) \\
		&= \int (\varphi_A' u_1)_x \partial_x u_1 + \int \varphi_A' Vu_1^2 \\
		&= \int \varphi_A'' u_1 \partial_x u_1 + \int \varphi_A' (\partial_x u_1)^2 + \int \varphi_A' Vu_1^2 \\
		&= - \frac12 \int \varphi_A''' u_1^2 + \int \varphi_A' (\partial_x u_1)^2 +  \int \varphi_A' V u_1^2.
	\end{align*}
	Then, replacing we obtain
	\begin{align}\label{eq:dI2dt}
		\begin{split}
			\frac{d}{dt} \mathcal{I}_2(t) &= \frac12 \int \varphi_A' u_2^2 + \frac14 \int \varphi_A''' u_1^2 - \frac12 \int \varphi_A' (\partial_x u_1)^2  - \frac12 \int \varphi_A' V u_1^2 - \frac12 \int \varphi_A' u_1 N^{\perp}.
		\end{split}
	\end{align}
	Finally, adding \eqref{eq:dI1dt} and \eqref{eq:dI2dt} we arrive at the equation,
	\begin{align*}
		\frac{d}{dt}\mathcal{I}(t) &= - \int \varphi_A' (\partial_x u_1)^2 + \frac14 \int \varphi_A'''u_1^2  + \frac12 \int \varphi_A V' u_1^2 - \int \left(\varphi_A \partial_x u_1 + \frac12 \varphi_A' u_1 \right)N^{\perp},
	\end{align*}
	which is precisely \eqref{eq:dIdt}.
\end{proof}

Unlike previous results in the area, the nonlinear term poses several problems in estimates. This is mainly due to the lack of dispersion for the one-dimensional wave equation, and the weak spatial decay of the linear part with respect to the non-linearity. For this reason, we treat the nonlinear term first. Recall that the nonlinear term is
\[
- \int \left(\varphi_A \partial_x u_1 + \frac12 \varphi_A' u_1 \right)N^{\perp},
\]
where $N^\perp$ was introduced in \eqref{eq:N}-\eqref{eq:Np}. We have the following intermediate result.
\begin{lemma}\label{lem:virial I non linear}
	There exists a universal constant $C>0$ such that 
	\begin{equation}\label{eq:virial I non linear}
		\begin{aligned}
			& - \int \left(\varphi_A \partial_x u_1 + \frac12 \varphi_A' u_1 \right) \left(\wtilQ^2 \left( 3\wtilH(a_1\phi_0 + u_1)^2 + (a_1\phi_0 + u_1)^3 \right) - N_0 \phi_0 \right)    \\
			& \hspace{2cm} \leq ~{} C a_1^4 + C\int \wtilQ^7 u_1^2 + C A \| \wtilQ^{1/2} u_1\|_{L^\infty}\int \wtilQ^3 w_1^2 + \frac14\int \wtilQ^3 \wtilH^2  w_1^2 \\
			& \hspace{2cm} \quad + 2\int \wtilQ^3  \left|\varphi_A \wtilH \right|\wtilH^6 u_1^2.
		\end{aligned}
	\end{equation}
\end{lemma}
\begin{remark}\label{rem:quadratic-nonlinear}
	Notice that the last two terms in \eqref{eq:virial I non linear} are nothing but quadratic, revealing that the purely nonlinear terms are not that small as one usually has in NLKG models. Precisely, these terms will be added to the ``quadratic part'' in \eqref{eq:dIdt}.
\end{remark}

\begin{proof}
	We decompose the first integral of \eqref{eq:virial I non linear} into several parts and write
	\begin{align*}
		& \int \left(\varphi_A \partial_x u_1 + \frac12 \varphi_A' u_1 \right) \wtilQ^2 \left( 3\wtilH(a_1\phi_0 + u_1)^2 + (a_1\phi_0 + u_1)^3 \right) \\
		& \hspace{3cm} = a_1^2\int \wtilQ^2(3\wtilH + a_1 \phi_0)\phi_0^2 \left(\varphi_A \partial_x u_1 + \frac12 \varphi_A' u_1 \right)  \\
		& \hspace{3cm} \quad + 3a_1 \int \wtilQ^2(2\wtilH + a_1 \phi_0) \phi_0 u_1 \left(\varphi_A \partial_x u_1 + \frac12 \varphi_A' u_1 \right) \\[0.1cm]
		& \hspace{3cm} \quad + 3 \int \wtilQ^2 (\wtilH + a_1\phi_0)u_1^2 \left(\varphi_A \partial_x u_1 + \frac12 \varphi_A' u_1 \right) \\
		& \hspace{3cm} \quad + \int \wtilQ^2 u_1^3 \left(\varphi_A \partial_x u_1 + \frac12 \varphi_A' u_1 \right) \\[0.1cm]
		& \hspace{3cm} = I_1 + I_2 + I_3 + I_4.
	\end{align*}
	For the first term, using integration by parts, the Cauchy-Schwarz inequality, the decay estimates on $\wtilQ$ and $\phi_0$, and noticing that for all $x\in\R$, $|\varphi_A'(x)|\leq \wtilQ$ and $|\varphi_A(x)|\leq |\alpha^{-1}(x)|$,
	\begin{equation}\label{eq:I1est}
		\begin{aligned}
			|I_1| \leq &~{} a_1^2 \int |\px(\wtilQ^2(3\wtilH + a_1\phi_0)\phi_0^2) \varphi_A u_1| + \frac12 a_1^2\int |\wtilQ^2(3\wtilH + a_1\phi_0)\phi_0^2 \varphi_A' u_1| \\
			\lesssim &~{} a_1^2 \left[ \left(\int \wtilQ\phi_0^4 |\alpha^{-1}|^2\right)^{\frac12} + \left(\int \wtilQ^{3}|\phi_0 \phi_0' \alpha^{-1}|^2\right)^{\frac12}
			\right] \left(\int \wtilQ^7 u_1^2 \right)^{\frac12} \\
			\lesssim &~{}  a_1^4 + \int \wtilQ^7 u_1^2.
		\end{aligned}
	\end{equation}
	For the second integral, by integration by parts, using the exponential decay \eqref{eq:properties-eigenvalL}, $\phi_A(x) \lesssim |\alpha^{-1}(x)|$, and in addition $|a_1|<1$ (see \eqref{eq:bound}), we obtain
	\begin{equation}\label{eq:I2est}
		\begin{aligned}
			|I_2| = &~{} \frac32 |a_1|\left|\int \px(\wtilQ^2(2\wtilH + a_1\phi_0)\phi_0 ) \varphi_A u_1^2 \right| \\
			\lesssim &~{} |a_1|\int |\alpha^{-1}(x)|(\wtilQ\phi_0 + \phi_0') \wtilQ^2 u_1^2 \lesssim \int \wtilQ^7 u_1^2.
		\end{aligned}
	\end{equation}
	For $I_3$ the control is more subtle. We decompose $I_3:=I_{3,1} +I_{3,2}$, where
	\[
	\begin{aligned}
		&	I_{3,1}:= 3 \int \wtilQ^2 \wtilH  u_1^2 \left(\varphi_A \partial_x u_1 + \frac12 \varphi_A' u_1 \right), \\
		&	I_{3,2}:= 3 a_1\int \wtilQ^2 \phi_0 u_1^2 \left(\varphi_A \partial_x u_1 + \frac12 \varphi_A' u_1 \right).
	\end{aligned}
	\]
	One easily has from the exponential decay of $\phi_0$,
	\begin{equation}\label{eq:I32est}
		\left| I_{3,2}\right| \lesssim A|a_1|\int \wtilQ^2 (|\phi_0|+|\phi_0'|) |u_1|^3 \lesssim A \| \wtilQ^{1/2} u_1\|_{L^\infty}\int \wtilQ^3 w_1^2 .
	\end{equation}
	On the other hand, using that $ \wtilQ' = -\wtilQ^2\wtilH $ and $\wtilH'  = \frac13 \wtilQ^2$, one has $ (\wtilQ^2 \wtilH)' = \frac13 \wtilQ^4 -2\wtilQ^3 \wtilH^2$, and
	\begin{equation}\label{eq:I31}
		\begin{aligned}
			I_{3,1}=&~{} \int   \left( \frac12 \varphi_A'   \wtilQ^2 \wtilH  - \varphi_A (\wtilQ^2 \wtilH)' \right) u_1^3 \\
			=&~{} \int   \left( \frac12\varphi_A' \wtilQ^2 \wtilH -\varphi_A \wtilQ^3 \left( \frac13 \wtilQ -2 \wtilH^2\right) \right) u_1^3.
		\end{aligned}
	\end{equation}
	To estimate $I_{3,1}$, we will exploit the positivity of $I_4$. Integrating by parts, we find
	\begin{equation}\label{eq:I4est}
		\begin{aligned}
			I_4 =&~ \int \wtilQ^2 u_1^3 \left(\varphi_A \partial_x u_1 + \frac12 \varphi_A' u_1 \right) \\
			=&~ \frac14 \int\wtilQ^2 \varphi_A' u_1^4 + \frac12 \int \wtilQ^3\wtilH\varphi_A u_1^4.
		\end{aligned}
	\end{equation}
	Note that each term in $I_4$ is non-negative. We have from this last identity and \eqref{eq:I31},\label{22}
		\[
		\begin{aligned}
			I_{3,1}+ I_4 = &~{}  \frac12 \int  \varphi_A'   \wtilQ^2 \wtilH  u_1^3+ \frac14 \int \varphi_A' \wtilQ^2 u_1^4  \\
			&~{} +  2 \int  \wtilQ^3 \left|\varphi_A \wtilH \right| \wtilH \left(1 - \frac23 \wtilQ\right) u_1^3  + \frac12 \int \wtilQ^3 \left| \varphi_A \wtilH \right| u_1^4  \\
			&~{} + \frac43 \int \wtilQ^4 \left|\varphi_A \wtilH \right|\wtilH u_1^3 -   \frac13 \int \wtilQ^4 \varphi_A u_1^3,
		\end{aligned}
		\]
		where we have added a cubic term with enough decay to be controlled and absorbed into the error estimates later. Now, we complete the square as follows
		\[
		\begin{aligned}
			I_{3,1}+ I_4 = &~{} \frac14 \int  \varphi_A' \wtilQ^2 u_1^2  \left( u_1+ \wtilH \right)^2  -\frac14 \int \varphi_A'\wtilQ^2 \wtilH^2  u_1^2  \\
			&~{} +  \frac12 \int \wtilQ^3  \left|\varphi_A \wtilH \right|u_1^2 \left( u_1 + 2\wtilH^3 \right)^2  - 2 \int \wtilQ^3  \left|\varphi_A \wtilH \right|\wtilH^6 u_1^2  \\
			&~{} + \frac43\int \wtilQ^4 \left|\varphi_A \wtilH \right|\wtilH u_1^3 -   \frac13 \int \wtilQ^4 \varphi_A u_1^3,
		\end{aligned}
		\]
		We bound the last two integrals using an $L^{\infty}$ estimate
		\[
		\frac43\int \wtilQ^4 \left|\varphi_A \wtilH^2 u_1^3\right|  +   \frac13 \int \wtilQ^4 \left|\varphi_A u_1^3\right| \lesssim A \|\wtilQ^{1/2} u_1\|_{L^\infty} \int \wtilQ^{7/2}u_1^2 \lesssim A \|\wtilQ^{1/2} u_1\|_{L^\infty} \int \wtilQ^3 w_1^2.
		\]
		Thus, we conclude that
		\begin{equation}\label{eq:I34est}
			\begin{aligned}
				I_{3,1}+ I_4 \geq &~{} - \frac14 \int \varphi_A'\wtilQ^2 \wtilH^2  u_1^2  - 2 \int \wtilQ^3  \left|\varphi_A \wtilH \right|\wtilH^6 u_1^2  - CA \|\wtilQ^{1/2} u_1\|_{L^\infty} \int  \wtilQ^3 w_1^2.
			\end{aligned}
		\end{equation}
	The last term that we treat from \eqref{eq:virial I non linear} is $N_0  \int\phi_0 \left(\varphi_A \partial_x u_1 + \frac12 \varphi_A' u_1 \right)$. By a point-wise estimate in \eqref{eq:N},
	\begin{equation}\label{eq:Nexplicit}
		N=  \wtilQ^2 \left( 3\wtilH(a_1^2\phi_0^2  +2a_1 \phi_0 u_1 + u_1^2) + a_1^3 \phi_0^3 + 3 a_1^2 \phi_0^2 u_1 + 3 a_1\phi_0 u_1^2 + u_1^3 \right)
	\end{equation}
	and using that  $|a_1|\lesssim 1$ (see \eqref{eq:bound}),
	\begin{equation}\label{eq:Nest}
		|N| \lesssim \wtilQ^2(a_1^2\phi_0^2 + u_1^2+|u_1|^3+u_1^4),
	\end{equation}
	and thus, by the decay estimates on $\wtilQ$ and $\phi_0$, $\| \wtilQ^{1/2}u_1\|_{L^{\infty}}\lesssim \| \wtilQ^{1/2} u_1\|_{H^1}\lesssim 1$, $A\geq 2$, it holds that \eqref{eq:Nest} implies
	\begin{equation}\label{eq:N0}
		|N_0| =\left| \langle \phi_0, N \rangle\right| \lesssim a_1^2 + \int \wtilQ^2 \phi_0 u_1^2 \lesssim a_1^2 + \int \wtilQ^7 u_1^2.
	\end{equation}
	Integrating by parts $ - \int\phi_0 \left(\varphi_A \partial_x u_1 + \frac12 \varphi_A' u_1 \right) = \int u_1 \left( \varphi_A \phi_0' + \frac12 \varphi_A' \phi_0 \right).$ Note that from the exponential decay of $\phi_0$, $\phi_0'$, and from the polynomial decay of $\wtilQ$, $\zeta_A$ we have
	\[
	\left| \varphi_A \phi_0' + \varphi_A' \phi_0 \right| \lesssim \alpha^{-1}(x)\phi_0' + \wtilQ\zeta_A^2\phi_0 \lesssim \wtilQ^7.
	\]
	Thus, using \eqref{eq:N0}, the Cauchy-Schwarz inequality and Lemma \ref{lemma:estimation},
	\begin{align}\label{eq:N0est} 
		\begin{split}
			\left| N_0  \int\phi_0 \left(\varphi_A \partial_x u_1 + \frac12 \varphi_A' u_1 \right) \right| \lesssim&~  \left(a_1^2 + \int \wtilQ^7 u_1^2\right) \int \wtilQ^7|u_1| \\
			\lesssim&~  \left(a_1^2 + \int \wtilQ^7 u_1^2\right) \left(\int \wtilQ^7 u_1^2\right)^{\frac12}\left(\int \wtilQ^7\right)^{\frac12} \\
			\lesssim&~  a_1^4 + \int \wtilQ^7 u_1^2.
		\end{split}
	\end{align}
	Gathering \eqref{eq:I1est}, \eqref{eq:I2est}, \eqref{eq:I32est}, \eqref{eq:I34est} and \eqref{eq:N0est}, we obtain for a constant $C>0$,
	\[
	\begin{aligned}
		- \int \left(\varphi_A \partial_x u_1 + \frac12 \varphi_A' u_1 \right) N^{\perp} \leq &~{} Ca_1^4 + C\int \wtilQ^7 u_1^2 + CA \| \wtilQ^{1/2} u_1\|_{L^\infty}\int \wtilQ^3 w_1^2  \\
		&~{} + \frac14 \int \wtilQ^3 \wtilH^2  w_1^2 + 2 \int \wtilQ^3  \left|\varphi_A \wtilH \right| \wtilH^6  u_1^2,
	\end{aligned}
	\]
	which is nothing but \eqref{eq:virial I non linear}.
\end{proof}

In the following lemma we rewrite the linear part of the virial identity \eqref{df:Ilinear} using the new variables $(w_1, w_2)$.

\begin{lemma}\label{lem:identidades}
	It holds that
	\begin{equation}\label{eq:virial1}
		\begin{aligned}
			I_{L}(t) = &~{} -\int  \wtilQ(\partial_x w_1)^2 + \frac12 \int \left[\frac{\zeta_A''}{\zeta_A} - \left(\frac{\zeta_A'}{\zeta_A}\right)^2 \right]\wtilQ w_1^2 + \frac14\int \wtilQ'' w_1^2 + \frac12\int \varphi_A V' u_1^2,
		\end{aligned}
	\end{equation}
	where $I_L(t)$ is defined in \eqref{df:Ilinear}. Moreover,
	\begin{equation}\label{eq:422}
		\frac{\zeta_A''}{\zeta_A} - \left(\frac{\zeta_A'}{\zeta_A}\right)^2 =
		\frac{1}{A} \left[ \chi'' |\alpha^{-1}| + 2\chi'\wtilQ  \, \sgn(\alpha^{-1}) + (1-\chi)\wtilQ^2\wtilH \, \sgn(\alpha^{-1}) \right].
	\end{equation}
	Additionally, 
	\begin{equation}\label{eq:estimations}
		\left|\frac{\zeta_A'}{\zeta_A}\right| \lesssim \frac{1}{A}\wtilQ\mathbf{1}_{\{|x|\geq 1\}}, \qquad \left| \frac{\zeta_A''}{\zeta_A} - \left(\frac{\zeta_A'}{\zeta_A}\right)^2 \right| \lesssim \frac1A \wtilQ^2\mathbf{1}_{\{|x|\geq 1\}}.
	\end{equation}
\end{lemma}
\begin{remark}
	Unlike previous works using this type of virial function, we obtain an expression in terms of $w_1$ with weight $\wtilQ$, and an extra term $\frac14\int\wtilQ''w_1^2$. This is due to the particular definition of $\zeta_A$ and $\varphi_A$ in \eqref{eq:zeta} chosen to handle the specific polynomial decay of the linearized potential. Another relevant feature is the loss of a compact support for the second expression in \eqref{eq:estimations}, which shall be controlled by the specific decay from $V'$ in \eqref{eq:1virial}.
\end{remark}
\begin{proof}
	Considering $w_1 = \zeta_A u_1$, and $\varphi_A' = \wtilQ \zeta_A^2$, we have,
	\begin{align*}
		\int \varphi_A'(\partial_x u_1)^2 = &~{} \int \wtilQ\left(\partial_x w_1 - \dfrac{ \zeta'_A }{ \zeta_A }w_1\right)^2 \\
		= &~{} \int \wtilQ(\partial_x w_1)^2 - 2\int \wtilQ\dfrac{ \zeta'_A }{ \zeta_A } w_1 \partial_x w_1 + \int  \left(\dfrac{ \zeta_A' }{ \zeta_A }\right)^2 \wtilQ w_1^2 \\
		= &~{} \int \wtilQ(\partial_x w_1)^2 + \int\left[\left(\frac{\wtilQ\zeta_A'}{\zeta_A}\right)' + \wtilQ \left(\frac{\zeta_A'}{\zeta_A}\right)^2 \right]w_1^2 \\
		= &~{} \int \wtilQ(\partial_x w_1)^2 + \int \frac{(\wtilQ\zeta_A')'}{\zeta_A}w_1^2,
	\end{align*}
	and
	\begin{align*}
		\int \varphi_A''' u_1^2 = \int \left[\wtilQ'' + 2\wtilQ'\left(\frac{\zeta_A'}{\zeta_A}\right)+2\frac{(\wtilQ\zeta_A')'}{\zeta_A} + 2\wtilQ\left(\frac{\zeta_A'}{\zeta_A} \right)^2\right]w_1^2.
	\end{align*}
	Then, 
	\[
	\begin{aligned}
		- \int \varphi'_A (\partial_x u_1)^2 + \frac14 \int \varphi'''_A u_1^2 =& -\int  \wtilQ(\partial_x w_1)^2 + \frac14\int \wtilQ'' w_1^2  \\
		&~ +  \frac12 \int \left[\wtilQ'\left(\frac{\zeta_A'}{\zeta_A}\right) - \frac{(\wtilQ\zeta_A')'}{\zeta_A} + \wtilQ\left(\frac{\zeta_A'}{\zeta_A} \right)^2\right] w_1^2 \\
		=& -\int  \wtilQ(\partial_x w_1)^2 + \frac14\int \wtilQ'' w_1^2 + \frac12 \int \wtilQ \left[\frac{\zeta_A''}{\zeta_A} - \left(\frac{\zeta_A'}{\zeta_A}\right)^2\right]w_1^2.
	\end{aligned}
	\]
	Replacing the above identities we obtain \eqref{eq:virial1}. By elementary computations of \eqref{eq:zeta}, we have
	\begin{align}
		\begin{split}
			\frac{\zeta_A'}{\zeta_A} &= \frac{1}{A} \left[ \chi'|\alpha^{-1}| - \sgn(\alpha^{-1})(\alpha^{-1})'(1 - \chi)\right]  \nonumber \\[0.1cm] 
			\frac{\zeta_A''}{\zeta_A} &= \left(\frac{\zeta_A'}{\zeta_A}\right)^2 + \frac{1}{A} \left[\chi''|\alpha^{-1}| + 2\chi'(\alpha^{-1})'  \sgn(\alpha^{-1}) - (1-\chi)(\alpha^{-1})'' \sgn(\alpha^{-1}) \right].
		\end{split}
	\end{align}
	Hence, replacing with \eqref{eq:dalpha}, we get \eqref{eq:422} and the first inequality of \eqref{eq:estimations}.
	
	\medskip
	
	Now, we describe in more detail the behavior of \eqref{eq:422}, which will differ from previous works on the subject. First, for $1\leq |x|\leq 2$, we can see that
	\[ \left| \frac{\zeta_A''}{\zeta_A} - \left(\frac{\zeta_A'}{\zeta_A}\right)^2 \right| \lesssim \frac{1}{A}. \]
	For $|x|\geq 2$, using \eqref{eq:dalpha}
	\[ \left| \frac{\zeta_A''}{\zeta_A} - \left(\frac{\zeta_A'}{\zeta_A}\right)^2 \right| = \frac{1}{A} \wtilQ^2|\wtilH| \leq \frac{1}{A}\wtilQ^2(x). \]
	Then one can see that
	\[ \left| \frac{\zeta_A''}{\zeta_A} - \left(\frac{\zeta_A'}{\zeta_A}\right)^2 \right| \lesssim  \frac{\wtilQ^2 \mathbf{1}_{\{|x|\geq 1\}}}{A},
	\]
	which proves the second estimate of \eqref{eq:estimations}.
\end{proof}

Finally, we focus on the control of the quadratic terms obtained from the non-linear part in \eqref{eq:virial I non linear} using the negativity of the linear part in \eqref{eq:virial1}. For this, we will employ the repulsivity of $V$ when $x$ is far from the origin and, unlike previous results, a Hardy-type inequality. This reflects the weak dispersion of the problem.

\begin{lemma}\label{lemma:virial I quadratic control}
	Let $(u_1,u_2)$ be a solution of \eqref{eq:system-perturbated}. Then, for $A$ large enough, there exist positive constants $C_0, C>0$ such that
	\begin{equation}
		\begin{aligned}
		& I_L(t) +\frac14\int \wtilQ^3 \wtilH^2  w_1^2 + 2 \int \wtilQ^3  \left|\varphi_A \wtilH \right| \wtilH^6 u_1^2\\
		& \hspace{2cm} \leq ~ -C_0\int  \wtilQ[(\partial_x w_1)^2 + \wtilQ^2 w_1^2] + C \int \wtilQ^7 u_1^2.
		\end{aligned}
	\end{equation}
\end{lemma}
\begin{proof}
	From \eqref{eq:virial1} and \eqref{eq:estimations}, grouping terms according to their spatial decay, there exist $C>0$ such that
	\begin{equation}\label{linearControl1}
		\begin{aligned}
		& I_L +\frac14\int \wtilQ^3 \wtilH^2  w_1^2 + 2 \int \wtilQ^3  \left|\varphi_A \wtilH \right| \wtilH^6 u_1^2 \\
		& \hspace{2.5cm} \leq ~ -\int  \wtilQ(\partial_x w_1)^2 + \frac{C}{A} \int_{|x|\geq1} \wtilQ^3 w_1^2   \\
		& \hspace{2.5cm} \quad ~ + \frac14 \int \left\{ (\wtilQ'' + \wtilH^2 \wtilQ^3)\zeta_A^2 + 2\varphi_A V' + 8\left|\varphi_A \wtilH\right| \wtilH^6 \wtilQ^3 \right\} u_1^2. 
		\end{aligned}
	\end{equation}
	In the rest of the proof, the constant $C$ may be redefined by a larger one, but remains bounded independently of the parameters.
	First, we study the weighted $L^2$ terms involved in the last integral. Using the definition of $\wtilQ$ and $V$, in addition to \eqref{eq:dalpha} and \eqref{derivada}, we compute
	\begin{equation}\label{linearControl2}
		\begin{aligned}
		&	(\wtilQ'' + \wtilH^2 \wtilQ^3)\zeta_A^2 + 2\varphi_A V' + 8\left|\varphi_A \wtilH\right| \wtilH^6 \wtilQ^3 \\
		&\hspace{2cm} = ~ \left(2 - \frac53\wtilQ + \wtilH^2 \right)\wtilQ^3\zeta_A^2 - \left(8 - 12\wtilQ - 8\left(1 - \frac23\wtilQ \right)^3  \right)\wtilQ^3 \varphi_A \wtilH \\
		&\hspace{2cm}	= ~ \left(3 - \frac73\wtilQ \right)\wtilQ^3\zeta_A^2 - \left(4 - \frac{32}{3}\wtilQ + \frac{64}{27}\wtilQ^2 \right)\wtilQ^4 \varphi_A \wtilH.
		\end{aligned}
	\end{equation}
	Notice that the last term on the RHS is not localized at scale $A$. Nevertheless, the factor in parenthesis is strictly positive for $|x|$ sufficiently large, giving it a favorable sign and so it can be dropped easily from the upper bound. We thus focus on controlling the first term, which has the same scale that the variable $w_1$. For this, we will apply a Hardy inequality to the weighted $\dot{H}^1$ term from the linear part in \eqref{linearControl1}.
	
	We proceed to localize the integral as follows. Let $\chi_K(\cdot):= \chi(\frac{\cdot}{K})$ be a cut-off function at scale $K>0$ to be determined later, where $\chi$ is defined by \eqref{eq:chi}. For simplicity, we denote $\eta_K^2:= 1 - \chi_K$. 
	
	\medskip
	
	A direct computation yields
	\[
	\begin{aligned}
		\left[\partial_x(\wtilQ^{\frac12} \eta_K w_1)\right]^2 = &~ \left[(\wtilQ^{\frac12})'\right]^2\eta_K^2 w_1^2 + (\wtilQ^{\frac12})'\wtilQ^{\frac12}\eta_K^2 \px (w_1^2) + \wtilQ\eta_K^2 (\px w_1)^2 \\
		&~ + 2(\wtilQ^{\frac12})'\wtilQ^{\frac12}\eta_K\eta_K' w_1^2 + \wtilQ \eta_K\eta_K'\px (w_1^2) + \wtilQ(\eta_K')^2 w_1^2.
	\end{aligned}
	\]
	Thus, after integration by parts we get
	\begin{equation}\label{dotH1}
		\begin{aligned}
			\int \wtilQ (1 - \chi_K)(\px w_1)^2 = &~ \int \left[\px(\wtilQ^\frac12 \eta_K w_1)\right]^2 + \int (\wtilQ^{\frac12})''\wtilQ^{\frac12}(1 - \chi_K) w_1^2 \\
			&~ + \int \left(\wtilQ' \eta_K' + \wtilQ \eta_K''  \right)\eta_K w_1^2.
		\end{aligned}
	\end{equation}
	For the first integral in the RHS of \eqref{dotH1} we can employ the one dimensional Hardy's inequality. For the second integral, we compute the term explicitly using \eqref{derivada}, and the last integral is localized at $\{K\leq |x|\leq 2K\}$, and so bounded by a reminder term. Thus, estimating from below, we have
	\[
	\begin{aligned}
		\int \wtilQ (1 - \chi_K)(\px w_1)^2 \geq &~ \frac14 \int \frac{\wtilQ}{x^2} (1 - \chi_K) w_1^2 + \frac{3}{4} \int \left(1 - \frac{8}{9} \wtilQ\right)\wtilQ^3 (1 - \chi_K)w_1^2 \\
		&~ - C \int_{\{K\leq |x|\leq 2K\}} \wtilQ w_1^2 \\
		\geq &~ \frac14 \int_{\{|x|\geq K\}} \frac{\wtilQ}{x^2}w_1^2 + \frac{3}{4} \int_{\{|x|\geq K\}} \left(1 - \frac{8}{9} \wtilQ\right)\wtilQ^3 w_1^2\\
		&~ - C \int \wtilQ^7 w_1^2.
	\end{aligned}
	\]
	Now, employing the estimate \eqref{eq:lowerboundQ} from Lemma \ref{lemma:estimation} with $\epsilon^2 = \tfrac{13}{19}$ and denoting $M_0>0$ the associated constant, we conclude that for $K\geq M_0$,
	\begin{equation}\label{hardy}
		\begin{aligned}
			\int \wtilQ (1 - \chi_K)(\px w_1)^2 \geq &~ \frac{13}{76} \int_{\{|x|\geq K\}} \wtilQ^3 w_1^2 + \frac{3}{4}\int_{\{|x|\geq K\}} \left(1 - \frac{8}{9} \wtilQ\right)\wtilQ^3 w_1^2  - C \int \wtilQ^7 w_1^2 \\
			= &~ \frac{35}{38} \int_{\{|x|\geq K\}} \left(1 - \frac{76}{105} \wtilQ\right)\wtilQ^3 w_1^2  - C \int \wtilQ^7 u_1^2.
		\end{aligned}
	\end{equation}
	Now we are in a position to prove the lemma. First, let $M_1>0$ be a fixed constant such that $1 - \frac83 \wtilQ(x) + \frac{16}{27}\wtilQ^2(x)>0$ for all $|x|\geq M_1$ (see Lemma \ref{lemma:estimation}). Fixing $K = \max\{M_0, M_1 \}$ and from \eqref{linearControl2}, we have
	\[
	\begin{aligned}
		& \frac14 \int \left\{ (\wtilQ'' + \wtilH^2 \wtilQ^3)\zeta_A^2 + 2\varphi_A V' + 8\left|\varphi_A \wtilH\right| \wtilH^6 \wtilQ^3 \right\} u_1^2  \\
		& \hspace{1cm} = ~ \frac14 \int_{\{|x|\leq K\}}  \left(3 - \frac73\wtilQ \right)\wtilQ^3 w_1^2 - \int_{\{|x|\leq K\}} \left(1 - \frac{8}{3}\wtilQ + \frac{16}{27}\wtilQ^2 \right)\wtilQ^4 \varphi_A \wtilH u_1^2 \\
		& \hspace{1cm}\quad + \frac14 \int_{\{|x|\geq K\}}  \left(3 - \frac73\wtilQ \right)\wtilQ^3 w_1^2 - \int_{\{|x|\geq K\}} \left(1 - \frac{8}{3}\wtilQ + \frac{16}{27}\wtilQ^2 \right)\wtilQ^4 \varphi_A \wtilH u_1^2 \\
		& \hspace{1cm} \leq C\int \wtilQ^7 u_1^2 + \frac34 \int_{\{|x|\geq K\}} \left(1 - \frac79\wtilQ \right)\wtilQ^3 w_1^2.
	\end{aligned}
	\]
	
	Now, we decompose the first integral in \eqref{linearControl1} and using \eqref{hardy}, we have
	\[
	\begin{aligned}
		\int  \wtilQ(\partial_x w_1)^2 \geq &~ \frac{1}{20} \int  \wtilQ(\partial_x w_1)^2 + \frac{19}{20} \int  \wtilQ (1-\chi_K)(\partial_x w_1)^2 \\
		\geq &~  \frac{1}{20} \int  \wtilQ(\partial_x w_1)^2 + \frac{7}{8} \int_{\{|x|\geq K\}} \left(1 - \frac{76}{105} \wtilQ\right)\wtilQ^3 w_1^2  - C \int \wtilQ^7 u_1^2.
	\end{aligned}
	\]
	Gathering these two estimates,  we bound \eqref{linearControl1} obtaining
	\[
	\begin{aligned}
		& I_L +\frac14\int \wtilQ^3 \wtilH^2  w_1^2 + 2 \int \wtilQ^3  \left|\varphi_A \wtilH \right| \wtilH^6 u_1^2 \\
		& \hspace{1cm} \leq ~ -\frac{1}{20} \int  \wtilQ(\partial_x w_1)^2 - \frac{7}{8} \int_{\{|x|\geq K\}} \left(1 - \frac{76}{105} \wtilQ\right)\wtilQ^3 w_1^2 \\
		&\hspace{1cm} \quad   + \frac34 \int_{\{|x|\geq K\}} \left(1 - \frac79\wtilQ \right)\wtilQ^3 w_1^2 + C\int \wtilQ^7 u_1^2 + \frac{C}{A} \int_{|x|\geq1} \wtilQ^3 w_1^2 \\
		&\hspace{1cm} = ~ -\frac{1}{20} \int  \wtilQ(\partial_x w_1)^2 - \frac{1}{8}\int_{\{|x|\geq K\}} \left( 1 - \frac{2}{5} \wtilQ\right)\wtilQ^3 w_1^2 + C\int \wtilQ^7 u_1^2 + \frac{C}{A} \int_{|x|\geq1} \wtilQ^3 w_1^2 \\
		&\hspace{1cm} \leq ~ -\frac{1}{20} \int  \wtilQ(\partial_x w_1)^2 - \frac{1}{20}\int_{\{|x|\geq K\}} \wtilQ^3 w_1^2 + C\int \wtilQ^7 u_1^2 + \frac{C}{A} \int_{|x|\geq1} \wtilQ^3 w_1^2.
	\end{aligned}
	\]
	The proof is completed by taking $C_0 = \frac{1}{40}$ and choosing $A\geq 40C$.
\end{proof}

\subsection{End of Proposition \ref{prop:virial I}}
Applying Lemmas \ref{lemma:virial I linear}, \ref{lem:virial I non linear} and \ref{lemma:virial I quadratic control}, there exist constants $C_0, C>0$ such that
\begin{align*}
	\frac{d}{dt}\mathcal{I} =&~{}  -\int \wtilQ (\partial_x w_1)^2 - \frac12 \int \left[\frac{\zeta_A''}{\zeta_A} - \left(\frac{\zeta_A'}{\zeta_A}\right)^2 \right]\wtilQ w_1^2 + \frac14\int \wtilQ'' w_1^2\\
	&~{} + \int \varphi_A V' u_1^2  - \int \left(\varphi_A \partial_x u_1 + \frac12 \varphi_A' u_1 \right)N^{\perp} \\
	\leq &~{}  - C_0\int \wtilQ [(\partial_x w_1)^2 + \wtilQ^2 w_1^2] + C \int \wtilQ^7 u_1^2 + C |a_1|^4 + C A \| \wtilQ^{1/2} u_1\|_{L^\infty}\int \wtilQ^3 w_1^2.
\end{align*}
Using $A = \delta^{-\frac14}$  (from \eqref{eq:choiceA}) and $\| \wtilQ^{1/2} u_1\|_{L^\infty}\lesssim \delta$ (from \eqref{eq:bound}), for $\delta_1$ small enough, we obtain \eqref{eq:1virial}.

\section{Transformed problem and second virial estimates}\label{sec:second virial estimate}

\subsection{Transformed problem}\label{subsec:transformedproblem}
We refer to \cite[Section~3]{Chang08} for more details about factorizations of Schr\"odinger operators and to \cite{KM2022,KMM19,LiLuhrmann22} for other uses in similar contexts.
Recall $L$ and $V$ from \eqref{def_L_V}, and let $L_0$, $U$, $U^*$ be defined as follows:
\begin{equation}\label{eq:L0}
	\begin{gathered}
		\hspace{.5cm} L_{0} = -\px^2 + V_0, \quad \text{with} \quad V_0 := 2\left(\frac{\px \phi_0}{\phi_0} \right)^2 - 2\mu_0^2 - V, \\[0.2cm]
		U = \phi_0 \cdot \px \cdot \phi_0^{-1}, \quad U^* = - \phi_0^{-1} \cdot \px \cdot \phi_0. 
	\end{gathered}
\end{equation}
An important point to remark here is the unknown character of the terms forming $L_0$ in \eqref{eq:L0}.  

\medskip

Then, the operators $L$ and $L_0$ rewrite as $L=U^* U - \mu_0^2$, $L_{0} = UU^* - \mu_0^2$ and it follows that
\[
UL = L_{0}U.\]
Let $(u_1, u_2)$ be a solution of the linear part of \eqref{eq:system-perturbated}, and set $v_1 = U u_1$, $v_2 = U u_2$. Then,
\begin{equation}\label{sist:L0}
	\begin{cases}
		\dot v_1 = v_2 \\
		\dot v_2 = - L_{0}[v_1].
	\end{cases}
\end{equation}
Our analysis relies in the crucial fact that the potential of $L_0$ is positive and repulsive. These properties happens to be the only spectral information needed for the proof of Theorem \ref{th:Main}. See Section \ref{B:POSITIVITY-POTENTIAL} for more details and the prove of these statements.

With respect to the above heuristic, we must take care of the loss of one derivative due to the operator $U$, without destroying the special algebra described. Therefore we need a regularization procedure of the functions involved, as in \cite{KMM19}. For this purpose, and for $\gamma > 0$ small to be defined later, we define the operator $(1 - \gamma \px^2)^{-1} :L^2(\R) \to H^2(\R)$ via its Fourier transform representation: for $h\in L^2$,
\begin{equation*}
	\widehat{(1 - \gamma \px^2)^{-1} h}(\xi) = \frac{\hat h(\xi)}{1 + \gamma \xi^2}.
\end{equation*}
Set
\begin{equation}\label{eq:def v1 v2}
	\begin{cases}
		v_1 = (1 - \gamma \px^2)^{-1} U(\tilde\chi_B u_1), \\
		v_2 = (1 - \gamma \px^2)^{-1} U(\tilde\chi_B u_2).
	\end{cases}
\end{equation}
where $\tilde\chi_B$ is defined in \eqref{eq:virial II notation}. We need this localization since the term $\int \wtilQ^7 u_1^2$ from Proposition \ref{prop:virial I} provides a localized estimate of $u_1$, and so the functions $(v_1, v_2)$ also must have a certain localization to compete against this term.

From the system \eqref{eq:system-perturbated} for $(u_1, u_2)$, follows that $(v_1, v_2)\in (H_0\cap \dot{H}^2)(\R)\times H^1(\R)$, and satisfies  the system
\begin{equation}\label{corrected}
	\begin{cases}
		\dot v_1 = v_2 \\
		\dot v_2 = -(1 - \gamma \px^2)^{-1} U(\tilde\chi_B Lu_1) - (1 - \gamma \px^2)^{-1}U(\tilde\chi_B N^{\perp}).
	\end{cases}
\end{equation}
In order to rewrite the first term in the right-hand side in terms of $v_1$, we first note that
\[
\tilde\chi_B Lu_1 = L(\tilde\chi_B u_1) + 2\tilde\chi_B' \px u_1 + \tilde\chi_B'' u_1.
\]
Second, we note that $UL = L_0 U$, then
\begin{align*}
	-(1 - \gamma \px^2)^{-1}UL(\tilde\chi_B u_1) =&~{} -(1 - \gamma \px^2)^{-1} L_0 U(\tilde\chi_B u_1) \\
	=&~{} - (1 - \gamma \px^2)^{-1} L_0[(1 - \gamma \px^2)v_1]\\
	=&~{} - (1 - \gamma \px^2)^{-1} (-\px^2 + V_0)(1 - \gamma \px^2)v_1\\
	=&~{} \px^2 v_1 - (1 - \gamma \px^2)^{-1}[V_0(1 - \gamma \px^2)v_1].
\end{align*}

Since
\begin{align*}
	(1 - \gamma \px^2)[V_0 v_1] &= V_0 v_1 - \gamma( V_0''v_1 + 2 V_0' \px v_1 + V_0 \px^2 v_1) \\
	&= V_0 (1 - \gamma \px^2) v_1 - \gamma (V_0''v_1 + 2 V_0' \px v_1),
\end{align*}
we obtain
\[ -(1 - \gamma \px^2)^{-1} UL(\tilde\chi_B u_1) = -L_0 v_1 - \gamma (1 - \gamma \px^2)^{-1}(V_0'' v_1 + 2 V_0' \px v_1). \]
Therefore, we have obtained the following system for $(v_1, v_2)$ (compare with \eqref{sist:L0}):
\begin{equation}
	\begin{cases}\label{eq:transformed system}
		\dot v_1 = v_2 \\
		\dot v_2 = -L_0 v_1 - \gamma (1 - \gamma \px^2)^{-1}(V_0'' v_1 + 2V_0' \px v_1) 
		\\[0.1cm] 
		\qquad - (1 - \gamma \px^2)^{-1}U[2\tilde\chi_B'\px u_1 + \tilde\chi_B'' u_1]
		- (1 - \gamma \px^2)^{-1}U (\tilde\chi_B N^{\perp}).
	\end{cases}
\end{equation}
An important point to be stressed now is that system \eqref{eq:transformed system}, unlike previous systems obtained recently in the field, has unknown function $V_0$. We do not assume any specific spectral property on $V_0$, but we will succeed to show the required repulsivity conditions on \eqref{eq:transformed system} by making interesting computations on its local and global behavior.

\subsection{Virial functional for the transformed problem}\label{subsec:virial II}
Recall $(v_1,v_2)$ from \eqref{eq:def v1 v2}. Set
\begin{equation}\label{def:J}
	\calJ (t) = \int \left(\psi_{A,B} (x) \px v_1(t,x) + \frac12\psi_{A,B}' (x) v_1 (t,x) \right)v_2 (t,x) dx
\end{equation}
where we recall that $\psi_{A,B}=\tilde\chi_{A}^2\varphi_B$ (see \eqref{eq:zeta} and \eqref{eq:virial II notation}), and define the localized version of the function $v_1$ at scale $B$ as follows
\begin{equation}\label{eq:z}
	z: = \tilde\chi_A \zeta_B v_1.
\end{equation}
This scale is intermediate, and $\calJ$ involves a cut-off at scale $A$, which will allow us to obtain an estimate in the same scale than the information obtained in Proposition \ref{prop:virial I}, needed to bound some bad error and nonlinear terms; see \cite{KMM19,KMMV20,MM23} for similar procedure.
\begin{proposition}\label{prop:virial II}
	There exist $C_2, C_3 > 0$ and $\delta_2>0$ such that for $\gamma$ small enough and for any $0<\delta\leq \delta_2$, the following holds. Fix 
	\begin{equation}\label{eq:choiceB}
		B=\alpha^{-1}(\delta^{-1/8}),
	\end{equation}
	and assume that for all $t\geq 0$, \eqref{eq:bound} holds. Then, for all $t\geq 0$, $\mathcal J$ in \eqref{def:J} satisfies
	\begin{equation}\label{eqn:bound_dtJ}
		\frac{d}{dt}\calJ \leq  -C_2\int \wtilQ[(\px z)^2 + \wtilQ^2 z^2] + C_3\ln(\delta^{-\frac18})^{-1} \int \wtilQ[(\px w_1)^2 + \wtilQ^2 w_1^2] + C_3\delta^\frac12|a_1|^3.
	\end{equation}
\end{proposition}

The rest of this section is devoted to the proof of Proposition \ref{prop:virial II}, which has been divided in several subsections.

\subsection{Proof of Proposition \ref{prop:virial II}: first computations}
Analogously to the computation of $\dot \calI$ in the proof of Proposition \ref{prop:virial I}, we have from \eqref{eq:transformed system},
\begin{equation}\label{los Ji}
	\begin{aligned}
		\frac{d}{dt}\calJ &= \int \left(\psi_{A,B} \px v_1 + \frac12\psi_{A,B}' v_1 \right)\dot v_2 \\
		&= - \int \left(\psi_{A,B} \px v_1 + \frac12\psi_{A,B}' v_1 \right)L_0 v_1 \\
		&\quad  - \gamma \int \left(\psi_{A,B} \px v_1 + \frac12\psi_{A,B}' v_1 \right)(1 - \gamma \px^2)^{-1}( V_0''v_1 + 2 V_0' \px v_1) \\
		&\quad - \int \left(\psi_{A,B} \px v_1 + \frac12\psi_{A,B}' v_1 \right)(1 - \gamma \px^2)^{-1}U[2\tilde\chi_B'\px u_1 + \tilde\chi_B'' u_1] \\
		&\quad -  \int \left(\psi_{A,B} \px v_1 + \frac12\psi_{A,B}' v_1 \right)(1 - \gamma \px^2)^{-1}U(\tilde\chi_B N^{\perp}) \\
		&=:  J_1 + J_2 + J_3 + J_4.
	\end{aligned}
\end{equation}
First, using the definition of $L_0$ and integrating by parts such as in the proof of Lemma \ref{lemma:virial I linear}, we have
\[ J_1 = -\int \psi_{A,B}'(\px v_1)^2 + \frac14\int \psi_{A,B}''' v_1^2 - \int \left(\psi_{A,B} \px v_1 + \frac12\psi_{A,B}' v_1 \right) V_0 v_1.  \]
By definition of $\psi_{A,B}$ (see \eqref{eq:virial II notation}), it follows that
\begin{equation}\label{psi_computations}
	\begin{aligned}
		\psi_{A,B}' &= \wtilQ\tilde\chi_A^2 \zeta_B^2 + (\tilde\chi_A^2)' \varphi_B\\
		\psi_{A,B}'' &= \wtilQ'\tilde\chi_A^2 (\zeta_B^2) + \wtilQ \tilde\chi_A^2 (\zeta_B^2)' + 2\wtilQ(\tilde\chi_A^2)'\zeta_B^2 + (\tilde\chi_A^2)''\varphi_B \\
		\psi_{A,B}''' &= \wtilQ'' \tilde\chi_A^2 \zeta_B^2 + 3\wtilQ'(\tilde\chi_A^2)' \zeta_B^2 + 2\wtilQ'\tilde\chi_A^2 (\zeta_B^2)' + 3\wtilQ(\tilde\chi_A^2)'(\zeta_B^2)' \\
		&\quad + 3\wtilQ(\tilde\chi_A^2)''\zeta_B^2 + \wtilQ \tilde\chi_A^2 (\zeta_B^2)'' + (\tilde\chi_A^2)'''\varphi_B.
	\end{aligned}
\end{equation}
Thus,
\[
\begin{aligned}
	&	-\int \psi_{A,B}' (\px v_1)^2 + \frac14 \int\psi_{A,B}''' v_1^2\\
	&\hspace{3cm} = - \int \wtilQ \tilde\chi_A^2 \zeta_B^2 (\px v_1)^2 + \frac14 \int \wtilQ''\tilde\chi_A^2\zeta_B^2 v_1^2 + \frac14 \int \wtilQ \tilde\chi_A^2 (\zeta_B^2)'' v_1^2 \\
	&\hspace{3cm} \quad + \frac34 \int \wtilQ'(\tilde\chi_A^2)'\zeta_B^2v_1^2 + \frac34 \int \wtilQ (\tilde\chi_A^2)'(\zeta_B^2)'v_1^2 + \frac34 \int\wtilQ(\tilde\chi_A^2)''\zeta_B^2 v_1^2 \\
	&\hspace{3cm} \quad + \frac12\int \wtilQ'\tilde\chi_A^2(\zeta_B^2)'v_1^2 - \int (\tilde\chi_A^2)'\varphi_B (\px v_1)^2 + \frac14\int (\tilde\chi_A^2)''' \varphi_B v_1^2.
\end{aligned}
\]
For the first term of this integral, by the definition of $z$ in \eqref{eq:z} and proceeding as in the proof of Lemma \ref{lemma:virial I linear}, we have
\begin{align*}
	\int \wtilQ \tilde\chi_A^2 \zeta_B^2 (\px v_1)^2 = &~{} \int \wtilQ(\px z)^2 + \int (\wtilQ(\tilde\chi_A\zeta_B)')'\tilde\chi_A\zeta_B v_1^2 \\
	= &~{} \int \wtilQ(\px z)^2 + \int \wtilQ\frac{\zeta_B''}{\zeta_B} z^2 + \int \wtilQ\tilde\chi_A''\tilde\chi_A \zeta_B^2 v_1^2 + \frac12\int \wtilQ(\tilde\chi_A^2)'(\zeta_B^2)'v_1^2 \\
	&~{} + \frac12\int \wtilQ' (\tilde\chi_A^2)'\zeta_B^2v_1^2 + \frac12\int \wtilQ' \tilde\chi_A^2(\zeta_B^2)'v_1^2,
\end{align*}
and
\begin{equation*}
	\frac14 \int \wtilQ\tilde\chi_A^2 (\zeta_B^2)'' v_1^2 = \frac12 \int \wtilQ\left( \frac{\zeta_B''}{\zeta_B} + \frac{(\zeta_B')^2}{\zeta_B^2} \right)z^2.
\end{equation*}
Thus,
\begin{equation*}
	\begin{aligned}
		& 	-\int \psi_{A,B}' (\px v_1)^2 + \frac14 \int\psi_{A,B}''' v_1^2  \\
		& \hspace{2.cm} = - \left\{ \int \wtilQ(\px z)^2 - \frac14 \int \wtilQ''z^2 + \frac12\int \wtilQ\left( \frac{\zeta_B''}{\zeta_B} - \frac{(\zeta_B')^2}{\zeta_B^2} \right)z^2 \right\} + \wtilJ_1,
	\end{aligned}
\end{equation*}
where we have set
\begin{equation}\label{tJ1}
	\begin{aligned}
		\wtilJ_1 &= \frac14 \int \wtilQ(\tilde\chi_A^2)'(\zeta_B^2)' v_1^2 + \frac14 \int \wtilQ'(\tilde\chi_A^2)'\zeta_B^2 v_1^2 + \frac12 \int\wtilQ[3(\tilde\chi_A')^2 + \tilde\chi_A''\tilde\chi_A] \zeta_B^2 v_1^2 \\
		&\quad - \int (\tilde\chi_A^2)'\varphi_B(\px v_1)^2 + \frac14\int (\tilde\chi_A^2)'''\varphi_B v_1^2.
	\end{aligned}
\end{equation}
Recalling \eqref{eq:z}, \eqref{eq:zeta}, \eqref{eq:virial II notation} and integrating by parts,
\[
\int \left(\psi_{A,B} \px v_1 + \frac12\psi_{A,B}' v_1 \right) V_0 v_1 = \frac12 \int V_0 \px(\psi_{A,B} v_1^2) = -\frac12 \int \frac{\varphi_{B}}{\zeta_B^2} V_0' z^2.
\]
Therefore, we define the potential
\begin{equation}\label{def:V_B}
	V_B = - \frac14 \wtilQ'' + \frac12 \wtilQ\left( \frac{\zeta_B''}{\zeta_B} - \frac{(\zeta_B')^2}{\zeta_B^2} \right) - \frac12 \frac{\varphi_B}{\zeta_B^2} V_0'.
\end{equation}

For convenience, we split this potential into two main parts, given by
\begin{equation}\label{def:V_BI V_BII}
	\begin{aligned}
		V_B = &~{} \left[\frac12 \wtilQ\left( \frac{\zeta_B''}{\zeta_B} - \frac{(\zeta_B')^2}{\zeta_B^2} \right) - \frac{1}{10} \frac{\varphi_B}{\zeta_B^2} V_0'\right] + \left[ - \frac14 \wtilQ'' - \frac25 \frac{\varphi_B}{\zeta_B^2} V_0' \right] \\
		=:  &~{} V_B^\text{I} + V_B^\text{II}.
	\end{aligned}
\end{equation}
Thus, the main part of the virial term can be written as
\begin{equation*}
	J_1 = - \int \left[ \wtilQ(\px z)^2 + V_B^\text{I} z^2 + V_B^\text{II}z^2 \right] + \wtilJ_1,
\end{equation*}
with $V_B^\text{I}$, $V_B^\text{II}$ in \eqref{def:V_BI V_BII}. The following result simplifies the use of $V_B^\text{I}$ in some extent.

\begin{lemma}\label{lem:V_BI}
	There exists $B_0>0$ such that for all $B\geq B_0$, $V_B^\text{I}\geq0$ on $\R$. More precisely, there exists $C_1'>0$ such that
	\begin{equation}\label{eq:bound V1}
		V_B^\text{I} \geq V_1 \quad \text{where} \quad V_1 =  {\color{black}C_1'} \wtilQ^3(x) \mathbf{1}_{ \{ |x|\geq 1 \} }(x),
	\end{equation}
	for all $x\in\R$.
\end{lemma}
\begin{proof}
	First, from \eqref{eq:estimations} (with $A$ replaced by $B$), it holds
	\[ 
	\left| \frac{\zeta_B''}{\zeta_B} - \left(\frac{\zeta_B'}{\zeta_B}\right)^2 \right| \leq \frac{C}B \wtilQ^2(x)\mathbf{1}_{ \{ |x|\geq 1 \} }(x) ,
	\]
	for some $C>0$.
	
	Second, since for $x\in [0, +\infty) \to \zeta_B(x)$ is non-increasing, applying a change of variables, we have for $x\geq 0$,
	\begin{equation}\label{eq:bound alpha}
		\frac{\varphi_B}{\zeta_B^2} = \frac{1}{\zeta_B^2}\int_0^{\alpha^{-1}(x)}\zeta_B^2(\alpha(s))ds \geq \alpha^{-1}(x).
	\end{equation}

	Now we will need some technical results about decay, positivity and repulsivity of $V_0$ that will be proved in Section \ref{B:POSITIVITY-POTENTIAL}.
	From Lemma \ref{lem:repulsivity} we have that $V_0'\leq 0$ for all $x\geq0$. Using the above inequalities and decomposing,
	\begin{align}\label{eq:VB}
		\begin{split}
			V_B^\text{I}(x) \geq &~{} \frac{1}{10} \alpha^{-1}(x)|V_0'(x)| - \frac{C}{B}\wtilQ^3(x) \mathbf{1}_{ \{ |x|\geq 1 \} }(x) \\
			\geq &~{} \left( \frac{1}{20}\alpha^{-1}(x)|V_0'(x)| - \frac{C}{B}\wtilQ^3(x) \right)\mathbf{1}_{\{ 1\leq x\leq x_{2,2} \} }(x) + \frac{1}{20}\alpha^{-1}(x)|V_0'(x)| \\
			&~{}+ \left( \frac{1}{20}\alpha^{-1}(x)|V_0'(x)| - \frac{C}{B}\wtilQ^3(x) \right)\mathbf{1}_{ \{ x\geq x_{2,2} \} }(x) \\
		\end{split}
	\end{align}
	where $x_{2,2} > 1$ is the second positive root of $V''$ (see Lemma \ref{def: roots}).
	
	For $x\in (1, x_{2,2})$, since by Lemma \ref{lem:repulsivity} we know $|V_0'(x)| > 0$, we have that there exist $\wtilC>0$ such that
	\[
	\frac{1}{20} \alpha^{-1}(x)|V_0'(x)| \geq \wtilC.
	\]
	Then, taking $B_1= \frac{27}{4}\frac{C}{\wtilC}$ we obtain
	\[
	\frac{1}{20} \alpha^{-1}(x)|V_0'(x)| - \frac{C}{B}\wtilQ^3 \geq \wtilC - \frac{27}{8}\frac{C}{B} \geq \frac12 \wtilC > 0,
	\]
	for all $B\geq B_1$.
	
	For $x\in (x_{2,2}, \infty)$, using Lemma \ref{lem:decay of V0}, the definition of $V$ and Lemma \ref{lemma:estimation}, we have that
	$\wtilQ^3 \lesssim |V_0'| \lesssim \wtilQ^3.$
	In particular, there exists $C'>0$ such that $C'\wtilQ^3 \leq |V_0'(x)|$ for all $x\geq x_{2,2}$. Using this, we obtain
	\be\label{814_bis}
	\frac{1}{20}\alpha^{-1}(x)|V_0'(x)| - \frac{C}{B}\wtilQ^3 \geq \left(\frac{C'}{20}\alpha^{-1}(x) - \frac{C}{B} \right)\wtilQ^3.
	\ee
	Thus, since by \eqref{eq:alpha} for $x\in [x_{2,2}, +\infty)\longmapsto \alpha^{-1}(x)$ is increasing, we have
	\[
	\frac{1}{20}\alpha^{-1}(x)|V_0'(x)| - \frac{C}{B}\wtilQ^3 \geq \left(\frac{C'}{20}\alpha^{-1}(x_{2,2}) - \frac{C}{B} \right)\wtilQ^3.
	\]
	Taking $B_2 = \frac{10}{\alpha^{-1}(x_{2,2})}\frac{C}{C'}$, it holds
	\[
	\frac{1}{20}\alpha^{-1}(x)|V_0'(x)| - \frac{C}{B}\wtilQ^3 \geq \frac12 C' \wtilQ^3,
	\]
	for all $B\geq B_2$.
	
	Defining $B_0 = \max\{B_1, B_2\}$, collecting the previous estimates in \eqref{eq:VB} and using again that $\alpha^{-1}:\R_+\mapsto\R_+$ is an increasing positive function,
	\[
	\begin{aligned}
		V_B^\text{I}(x) \geq &~{} \frac12 \wtilC\mathbf{1}_{\{ 1\leq x \leq \tilx \} }(x) +   \frac12 C' \wtilQ^3\mathbf{1}_{\{ x\geq \tilx \} }(x) +  \frac{1}{20}\alpha^{-1}(x)|V_0'(x)| \\[0.1cm]
		\geq &~{}  \frac12 \wtilC\mathbf{1}_{\{ 1\leq x \leq \tilx \}}(x) +   \frac12 C'\wtilQ^3\mathbf{1}_{\{ x\geq \tilx \} }(x),
	\end{aligned}
	\]
	for all $B\geq B_0$. We conclude that there exists $C_1'>0$ such that
	\[
	V_B^\text{I}(x) \geq C_1' \wtilQ^3\mathbf{1}_{ \{ x\geq 1 \} }(x),
	\]
	for all $x\geq0$, $B\geq B_0$.  By parity, this estimate holds for any $x\in\R$, obtaining \eqref{eq:bound V1}.
\end{proof}

Now, we have to obtain some estimate for the potential $V_B^\text{II}$ in \eqref{def:V_BI V_BII}. For this, we prove the following result.
\begin{lemma}\label{lem:V_BII}
	The potential $V_B^\text{II}$ is strictly positive on $\R$. Even more, there exists $C_1''>0$ such that
	\begin{equation}\label{eq:bound V2}
		V_B^\text{II} \geq V_2 \quad \text{where} \quad V_2 =  C_1'' \wtilQ^3(x),
	\end{equation}
	for all $x\in \R$.
\end{lemma}
\begin{proof}
	By parity we restrict to $x\geq 0$. First, using \eqref{derivada} and the definition of $\wtilQ$, we have
	\begin{equation}\label{eq:extra term}
		-\frac14 \wtilQ'' = \frac12\wtilQ^3\left(\frac56 \wtilQ - 1 \right).
	\end{equation}
	We notice that \eqref{eq:extra term} is positive for $\wtilQ > \frac65$. If we denote $\bar x$ the unique positive root of \eqref{eq:extra term}, from the definition of $\wtilQ$ we have (solving $\wtilQ(\bar x)=\tfrac{6}{5}$, i.e., $\cosh(\alpha^{-1}(\bar x)/2)=\tfrac{\sqrt 5}{2}$)
	\[
	\barx =
	\alpha\left(2\hbox{arccosh}\left(\tfrac{\sqrt 5}{2}\right)\right)
	\sim 0.693,
	\]
	and we notice, recalling that $\wtilQ$ is a decreasing function on $\R_+$,  that  \eqref{eq:extra term} is positive for $|x|\leq \barx$. Using this, the repulsivity of $V_0$ and the definition of $V_B^\text{II}$, we have that
	\[
	V_B^\text{II}(x) > 0,
	\]
	for any $x\in [0, \barx)$.
	
	For $x\geq x_{2,2}$, where $x_{2,2}$ is the second positive root of $V''$ (see Lemma \ref{def: roots}), using \eqref{eq:bound alpha}, the upper bound decay estimate for $V_0'$ from Lemma \ref{lem:decay of V0}, and replacing \eqref{eq:dalpha} we obtain
	\be\label{814_bis_bis}
	\begin{aligned}
		V_B^\text{II}(x) = &~{} - \frac14 \wtilQ'' - \frac25 \frac{\varphi_B}{\zeta_B^2} V_0' \ge - \frac14\wtilQ'' -\frac15 \alpha^{-1}(x)V'(x) \\
		= &~{} \frac12\wtilQ^3\left(\frac56 \wtilQ - 1 \right) + \frac25 \alpha^{-1}(x)(2 - 3\wtilQ)\wtilQ^3 \wtilH \\ 
		= &~{} \left(\frac45 \alpha^{-1}(x)\wtilH - \frac12 \right)\wtilQ^3 + \left(\frac{5}{12} - \frac{6}{5}\alpha^{-1}(x)\wtilH \right)\wtilQ^4\\
		 =:&~{} k(\alpha^{-1}(x))\wtilQ^3,
	\end{aligned}
	\ee
	where we have defined the auxiliary function $k:\R_+ \mapsto \R$ as
	\[
	k(s) := \frac45 s H(s) - \frac12 + \left(\frac{5}{12} - \frac{6}{5}s H(s) \right)Q(s).
	\]
	Given \eqref{Q} and \eqref{eq:H}, this is an explicit function with two positive roots $s_1\sim 0.47$ and $s_2\sim 2.21$. Even more, from the asymptotic behavior of $k(s)$ for $s\to \infty$ we have that
	$k(s)> 0$ for all $s > s_2$. Using the bijectivity of $\alpha$, that $\wtilQ(x_{2,2}) \sim 0.49$, $Q(s_2)\sim 0.54$, this implies that $\alpha(s_2) < x_{2,2}$, and we conclude that $V_B^\text{II}(x) \gtrsim \wtilQ^3(x)$ for all $x\geq x_{2,2}$.  For $x\in (\barx, x_{2,2})$, computing we have that $V_B^\text{II}(x) > 0.$ Considering the above cases and by parity, there exist $C, \wtilC > 0$ such that
	\[
	V_B^\text{II}(x) \geq C \mathbf{1}_{ |x|\leq x_{2,2} }(x) + \wtilC \wtilQ^3 \mathbf{1}_{ |x|\geq x_{2,2} }(x),
	\]
	for all $x\in\R$. To sum up, we have that there exists $C_1''>0$ where it holds
	\[
	V_B^\text{II}(x) \geq C_1'' \wtilQ^3(x),
	\]
	for all $x\in \R$. This ends the proof of \eqref{eq:bound V2}.
\end{proof}

Using Lemmas \ref{lem:V_BI} and \ref{lem:V_BII}, the definition of $V_B$ in \eqref{def:V_B} and considering $C_1= \min\{C_1', C_1'' \}$, we obtain
\begin{equation}\label{eq:error terms}
	\frac{d}{dt}\calJ \leq  - \int \wtilQ \left[ (\px z)^2 + C_1 \wtilQ^2 z^2 \right] + \wtilJ_1 + J_2 + J_3 + J_4,
\end{equation}
with $J_2$, $J_3$ and $J_4$ as in \eqref{los Ji}, and  $\wtilJ_1$ as in \eqref{tJ1}. To control the terms $\wtilJ_1$, $J_2$, $J_3$ and $J_4$ we need some technical estimates.

\subsection{Technical estimates.}
The following estimates are already classical, but in our context, since the decay is only algebraic, we need some particular care.
We start out with estimates necessary to treat regularized functions. The proof of these are different from previous work due to the slow decay of the potential $V_0$. We first recall the following well-known result.
\begin{lemma}[See \cite{KMM19}]
	For any $\gamma\in(0,1)$ and $f\in L^2$,
	\begin{align}\label{eq:basic inequalities}
		\begin{split}
			&	\left\|(1-\gamma\px^2)^{-1}f\right\|_{L^2}\leq \|f\|_{L^2}, \quad \left\|(1-\gamma\px^2)^{-1} \px f\right\|_{L^2} \leq \gamma^{-\frac12}\|f\|_{L^2}, \\
			& \hspace{2cm}	\left\|(1-\gamma\px^2)^{-1}\px^2 f\right\|_{L^2} \leq \gamma^{-1}\|f\|_{L^2}.
		\end{split}
	\end{align}
\end{lemma}
Our third result uses the fact that, even if the decay is only polynomial, it is strong enough to perform commutator estimates.

\begin{lemma}\label{lem:commutativity}
	Let $\alpha(\cdot)$ be the function defined in \eqref{eq:alpha}. For any $0<K\leq 3$, $\gamma > 0$ small enough, and $f\in L^2(\R)$ one has
	\begin{equation}\label{eq:conmutativity sech}
	\begin{aligned}
	& \left\| \sech(K\alpha^{-1}(x))(1-\gamma\px^2)^{-1}f\right\|_{L^2} \le (1+m_0)\left\|(1 - \gamma\px^2)^{-1}[\sech(K\alpha^{-1}(x)) f]\right\|_{L^2},
	\end{aligned}
	\end{equation}
	where $m_0>0$ is any fixed small constant, and
	\begin{equation}\label{eq:conmutativity cosh}
		\left\|\cosh(K\alpha^{-1}(x))(1 - \gamma\px^2)^{-1}f\right\|_{L^2} \lesssim \left\|(1 - \gamma\px^2)^{-1} [\cosh(K\alpha^{-1}(x))f]\right\|_{L^2},
	\end{equation}
	where the implicit constant is independent of $\gamma$ and $K$.
\end{lemma}

Let us recall that in view of \eqref{eqn:equivalencias}, the term $\sech(K\alpha^{-1}(x))$ has only polynomial decay.
\begin{proof}
	We set $g = \sech(K\alpha^{-1})(1-\gamma\px^2)^{-1}f$ and $k = (1 - \gamma\px^2)^{-1}[\sech(K\alpha^{-1})f]$. We have
	\begin{align*}
		f = &~{} \cosh(K\alpha^{-1})(1 - \gamma\px^2)k = (1 - \gamma\px^2)[\cosh(K\alpha^{-1})g] \\
		= &~{}\cosh(K\alpha^{-1})g  - \gamma[\cosh(K\alpha^{-1})''g + 2\cosh(K\alpha^{-1})'\px g + \cosh(K\alpha^{-1})\px^2 g] \\
		= &~{} \cosh(K\alpha^{-1})(1 - \gamma\px^2)g - \gamma K\cosh(K\alpha^{-1})\wtilQ^2\left[K - \wtilH\tanh(K\alpha^{-1}) \right]g \\
		&~{} - 2\gamma K\cosh(K\alpha^{-1}) \wtilQ \tanh(K\alpha^{-1}) \px g.
	\end{align*}
	Thus,
	\[
	\begin{aligned}
		(1 - \gamma\px^2)k = &~{} (1 - \gamma\px^2)g  -  \gamma K\wtilQ^2\left[K - \wtilH\tanh(K\alpha^{-1}) \right]g - 2\gamma K\wtilQ \tanh(K\alpha^{-1}) \px g.
	\end{aligned}
	\]
	Applying the operator $(1 - \gamma\px^2)^{-1}$ to this id{} entity, we obtain
	\begin{align*}
		g = &~{} k + \gamma K (1 - \gamma\px^2)^{-1}\left\{\wtilQ^2\left[K - \wtilH\tanh(K\alpha^{-1}) \right]g \right\} \\
		&~{} + 2\gamma K (1 - \gamma\px^2)^{-1} \left[\wtilQ \tanh(K\alpha^{-1}) \px g\right].
	\end{align*}
	We have from \eqref{eq:basic inequalities} that for $\gamma\leq \frac12$,
	\begin{align}\label{eq:norm regularizer operator}
		\|(1 - \gamma\px^2)^{-1}\|_{\calL(L^2, L^2)} \le 1, \quad \|(1 - \gamma\px^2)^{-1}\px\|_{\calL(L^2, L^2)} \le \gamma^{-\frac12}.
	\end{align}
	Thus, for $0<K\leq 3$,
	\[
	\begin{aligned}
		\gamma K \left\| (1 - \gamma\px^2)^{-1}\left\{\wtilQ^2\left[K - \wtilH\tanh(K\alpha^{-1}) \right]g \right\}\right\|_{L^2} \le&  \gamma K  \left\| \wtilQ^2\left[K - \wtilH\tanh(K\alpha^{-1}) \right]g \right\|_{L^2} \\
		\le& (1+K)\gamma K  \| \wtilQ^2 g\|_{L^2},
	\end{aligned}
	\]
	and using again \eqref{eq:norm regularizer operator} and  \eqref{derivada de tQ},
	\begin{align*}
		\left\| (1 - \gamma\px^2)^{-1} \left[\wtilQ \tanh(K\alpha^{-1}) \px g\right] \right\|_{L^2}  \le&~   \left\|(1 - \gamma\px^2)^{-1}\px \left[\wtilQ \tanh \left( K\alpha^{-1} \right) g \right]\right\|_{L^2} \\
		& + \left\| (1 - \gamma\px^2)^{-1} \left[ \px \left( \wtilQ \tanh(K\alpha^{-1}) \right) g \right] \right\|_{L^2} \\
		\le&~{}  \gamma^{-\frac12} \left\| \wtilQ\tanh \left( K\alpha^{-1} \right) g\right\|_{L^2}  \\
		&~{} + \left\| \wtilQ^2\left(K\sech^2(K\alpha^{-1}) - \wtilH\tanh(K\alpha^{-1})  \right)g\right\|_{L^2} \\
		\le&~{} \gamma^{-\frac12} K\|  \wtilQ^2 g\|_{L^2} + \| \widetilde Q^2 g\|_{L^2} \le  3\gamma^{-\frac12}\| \wtilQ^2g\|_{L^2}.
	\end{align*}
	We obtain
	\[
	\|g\|_{L^2} \leq \|k\|_{L^2}+ (1+K)\gamma K  \| \wtilQ^2 g\|_{L^2}+6 K \gamma^{\frac12}\| \wtilQ^2g\|_{L^2}.
	\]
	We deduce that for any $m_0>0$ fixed and small,
	\[
	\|g\|_{L^2} \leq (1+m_0)\|k\|_{L^2},
	\]
	which implies \eqref{eq:conmutativity sech} for $\gamma$ small enough.
	
	\medskip
	
	We prove \eqref{eq:conmutativity cosh} similarly. Setting 
	\[
	g = \cosh(K\alpha^{-1})(1 - \gamma\px^2)^{-1}f\quad \hbox{ and } \quad k=(1 - \gamma\px^2)^{-1}[\cosh(K\alpha^{-1})f],
	\]
	we compute
	\[
	\begin{aligned}
		f = &~{} \sech(K\alpha^{-1})(1 - \gamma\px^2)k = (1 - \gamma\px^2)[\sech(K\alpha^{-1})g] \\
		= &~{} \sech(K\alpha^{-1})g - \gamma \left[ \sech(K\alpha^{-1})'' g + 2\sech(K\alpha^{-1})'\px g + \sech(K\alpha^{-1})\px^2g \right] \\
		= &~{} \sech(K\alpha^{-1})(1 - \gamma\px^2)g  \\
		&~{} - \gamma K\wtilQ\sech(K\alpha^{-1}) \left[ K\wtilQ(1 - 2\sech^2(K\alpha^{-1}))g - 2\tanh(K\alpha^{-1})\px g \right].
	\end{aligned}
	\]
	Thus, applying the operator $(1 - \gamma\px^2)^{-1}$ as before, we have 
	\[
	\begin{aligned}
		g =&~{}  k + \gamma K^2(1 - \gamma\px^2)^{-1} \left[ \wtilQ^2(1 - 2\sech^2(K\alpha^{-1}))g \right] \\
		&~{} - 2\gamma K(1 - \gamma\px^2)^{-1} \left[ \wtilQ\tanh(K\alpha^{-1})\px g \right].
	\end{aligned}
	\]
	Using $0<K\leq 3$ and \eqref{eq:norm regularizer operator}, it follows that
	\[
	\left\|(1 - \gamma\px^2)^{-1}[\wtilQ^2(1 - 2\sech^2(K\alpha^{-1}))g]\right\|_{L^2} \lesssim \left\|\wtilQ^2(1 - 2\sech^2(K\alpha^{-1}))g\right\|_{L^2} \lesssim \|g\|_{L^2},
	\]
	and
	\[
	\begin{aligned}
		&~{} \left\|(1 - \gamma\px^2)^{-1}[\wtilQ\tanh(K\alpha^{-1})\px g]\right\|_{L^2} \\ &~{} \hspace{0.5cm} \lesssim   \left\|(1 - \gamma\px^2)^{-1}\px[\wtilQ\tanh(K\alpha^{-1})g]\right\|_{L^2}  + \left\|(1 - \gamma\px^2)^{-1}[\px(\wtilQ\tanh(K\alpha^{-1})) g]\right\|_{L^2} \\
		&~{} \hspace{0.5cm} \lesssim \gamma^{-\frac12} \left\| \wtilQ\tanh(K\alpha^{-1})g\right\|_{L^2} + \left\|\wtilQ^2[K\sech^2(K\alpha^{-1}) - \wtilH\tanh(K\alpha^{-1})]g\right\|_{L^2} \\
		&~{} \hspace{.5cm}
		\lesssim \gamma^{-\frac12}\|g\|_{L^2}.
	\end{aligned}
	\]
	It follows that there exist $\wtilC > 0$ independent of $\gamma$ such that
	\[
	\|g\|_{L^2} \leq \|k\|_{L^2} + \wtilC\gamma^{\frac12}\|g\|_{L^2}.
	\]
	Considering $\gamma$ small enough we obtain \eqref{eq:conmutativity cosh}.
\end{proof}

\begin{remark}
	There are some interesting consequences of the previous results. Indeed, using \eqref{eq:conmutativity sech} and \eqref{eq:conmutativity cosh} for $K=\frac{n}{2} + \frac2A$ with $A\geq2$, \eqref{eq:basic inequalities} and $n=1,3$ implies the following inequalities
	\begin{equation}\label{eq: tech 1}
		\left\| \sech\left(\left(\frac32 + \frac1A\right)\alpha^{-1}\right) (1 - \gamma\px^2)^{-1}f\right\|_{L^2} \lesssim \left\| (1-\gamma\px^2)^{-1}\left[\sech\left(\left(\frac32 + \frac1A\right)\alpha^{-1}\right)f\right]\right\|_{L^2},
	\end{equation}
	and
	\begin{equation}\label{eq: tech 2}
		\left\| \sech\left(\left(\frac12 + \frac1A\right)\alpha^{-1}\right) (1 - \gamma\px^2)^{-1}f\right\|_{L^2} \lesssim \left\| (1-\gamma\px^2)^{-1}\left[\sech\left(\left(\frac12 + \frac1A\right)\alpha^{-1}\right)f\right]\right\|_{L^2}.
	\end{equation}
	Besides, recall that for $\sigma_A$ as in \eqref{def:sigmaA},
	\begin{equation}\label{eq:equivalence cosh}
		\sigma_A \wtilQ^{-\frac{1}{2}}\lesssim \cosh\left(\frac{2-A}{2A}\alpha^{-1}\right) \lesssim \sigma_A \wtilQ^{-\frac{1}{2}},
	\end{equation}
	for any $A\geq 2$. Using
	\eqref{eq:conmutativity cosh} for $K=\frac{2-A}{2A}$ with $A\geq4$, and \eqref{eq:equivalence cosh}, one gets
	\begin{equation}\label{eq: tech 4}
		\left\| \sigma_A \wtilQ^{-\frac12}(1 - \gamma\px^2)^{-1}f \right\|_{L^2} \lesssim
		\left\|\sigma_A \wtilQ^{-\frac12}f\right\|_{L^2}.
	\end{equation}
	\begin{equation}\label{eq: tech 3}
		\left\| \sigma_A \wtilQ^{-\frac12}(1 - \gamma\px^2)^{-1}\px f \right\|_{L^2} \lesssim
		\gamma^{-\frac12}\left\|\sigma_A \wtilQ^{-\frac12}f\right\|_{L^2}.
	\end{equation}
\end{remark}
The following result is a $\widetilde Q$ localized version of the radiation term.
\begin{lemma}\label{lem:localized estimation}
	For any $A\geq1$ large, any $\gamma > 0$ small and any $u$ measurable, if we define $v$ related with $u$ by
	\begin{equation*}
		v = (1 - \gamma\px^2)^{-1} Uu,
	\end{equation*}
	then
	\begin{equation}\label{eq:tech 11}
		\left\| \sigma_A \wtilQ^\frac32 v \right\|_{L^2} \lesssim \gamma^{-\frac12} \left\| \sigma_A \wtilQ^\frac32 u\right\|_{L^2},
	\end{equation}
	and
	\begin{equation}\label{eq:tech 22}
		\left\| \sigma_A \wtilQ^\frac12\px v\right\|_{L^2} \lesssim \gamma^{-\frac12} \left\| \sigma_A \wtilQ^\frac12 \px u\right\|_{L^2} + \left\| \sigma_A \wtilQ^\frac52 u\right\|_{L^2}.
	\end{equation}
\end{lemma}
\begin{remark}
	Estimates in \eqref{eq: tech 4}, \eqref{eq: tech 3} and Lemma \ref{lem:localized estimation} require the additional terms $\wtilQ^\frac12$, $\wtilQ^\frac32$ in order to control some nonstandard terms appearing in below estimates.
\end{remark}
\begin{proof}[Proof of Lemma \ref{lem:localized estimation}]
	By direct computations, we have $U = \px - h_0$,
	where the function $h_0$ is bounded (see Appendix Lemma \ref{lem:bounds h0}). In addition, using that
	\begin{equation}\label{eq:equivalence}
		\sigma_A \wtilQ^{\frac{n}{2}} \lesssim \sech\left(\left(\frac{n}{2} + \frac1A\right)\alpha^{-1}\right) \lesssim \sigma_A \wtilQ^{\frac{n}{2}}
	\end{equation}
	with $n=3$, the first estimate is consequence of \eqref{eq: tech 1} and \eqref{eq:basic inequalities},
	\begin{align*}
		\left\| \sigma_A\wtilQ^\frac32 v\right\|_{L^2} \lesssim&~{} \left\|\sech\left(\frac{3A + 2}{2A}\alpha^{-1}\right)v\right\|_{L^2} \\
		\lesssim&~{} \left\|\sech\left(\frac{3A + 2}{2A}\alpha^{-1}\right)(1 - \gamma\px^2)^{-1}\px u\right\|_{L^2} + \left\| \sech\left(\frac{3A + 2}{2A}\alpha^{-1}\right)(1-\gamma\px^2)^{-1}[h_0 u]\right\|_{L^2} \\
		\lesssim&~{} \left\|(1 - \gamma\px^2)^{-1}\left[\sech\left(\frac{3A + 2}{2A}\alpha^{-1}\right)\px u\right]\right\|_{L^2}   + \left\|(1-\gamma\px^2)^{-1}\left[ \sech\left(\frac{3A + 2}{2A}\alpha^{-1}\right)h_0 u \right]\right\|_{L^2} \\
		\lesssim&~{} \left\|(1 - \gamma\px^2)^{-1}\px\left[\sech\left(\frac{3A + 2}{2A}\alpha^{-1}\right) u\right]\right\|_{L^2}  \\
		&~{} + \frac{3A + 2}{2A}\left\|(1 - \gamma\px^2)^{-1}\left[\wtilQ\sech\left(\frac{3A + 2}{2A}\alpha^{-1}\right) u\right]\right\|_{L^2}  + \left\|\sech\left(\frac{3A + 2}{2A}\alpha^{-1}\right)h_0 u\right\|_{L^2}.
	\end{align*}
	Applying \eqref{eq:basic inequalities} again to the previous estimate,
	\begin{align*}
		\left\| \sigma_A\wtilQ^\frac32 v\right\|_{L^2} \lesssim&~{} \gamma^{-\frac12} \left\|\sech\left(\frac{3A + 2}{2A}\alpha^{-1}\right)u\right\|_{L^2} \\
		&~{} + \left\| \wtilQ\sech\left(\frac{3A + 2}{2A}\alpha^{-1}\right) u\right\|_{L^2} + \left\| \sigma_A \wtilQ^\frac32 h_0 u\right\|_{L^2} \\
		\lesssim&~{}  \gamma^{-\frac12}\left\|\sigma_A \wtilQ^\frac32 u\right\|_{L^2} + \left\|\sigma_A \wtilQ^\frac32 u\right\|_{L^2} + \left\| \sigma_A \wtilQ^\frac32 h_0 u\right\|_{L^2} \\
		\lesssim&~{}  \gamma^{-\frac12} \left\| \sigma_A \wtilQ^\frac32 u\right\|_{L^2}.
	\end{align*}
	This proves \eqref{eq:tech 11}.
	
	For the second estimate, we have
	\[
	\px U = \px^2 - h_0\px - h_0'.
	\]
	Using \eqref{eq:equivalence} with $n=1$ and \eqref{eq: tech 2}, plus the fact that $h_0$ is bounded and $|h_0'|\lesssim |V| \lesssim \wtilQ^2$ (see Lemma \ref{lem:h0 properties}, \eqref{eq:h0'}), analogously to the previous estimate we have
	\begin{align*}
		\|\sigma_A\wtilQ^\frac12\px v\|_{L^2} \lesssim & \left\| \sech\left(\frac{A + 2}{2A}\alpha^{-1}\right) \px v \right\|_{L^2} \\
		\lesssim & \left\| \sech\left(\frac{A + 2}{2A}\alpha^{-1}\right)(1 - \gamma\px^2)^{-1}\px^2 u \right\|_{L^2} \\
		& + \left\| \sech\left(\frac{A + 2}{2A}\alpha^{-1}\right)(1-\gamma\px^2)^{-1}[h_0\px u] \right\|_{L^2} \\
		&+ \left\| \sech\left(\frac{A + 2}{2A}\alpha^{-1}\right)(1 - \gamma\px^2)^{-1}[h_0'u] \right\|_{L^2} \\
		\lesssim &  \left\| (1 - \gamma\px^2)^{-1}\px\left[\sech\left(\frac{A + 2}{2A}\alpha^{-1}\right) \px u\right] \right\|_{L^2} \\
		&+ \left\| (1 - \gamma\px^2)^{-1}\left[\sech\left(\frac{A + 2}{2A}\alpha^{-1}\right)' \px u\right] \right\|_{L^2} \\
		& + \left\| (1-\gamma\px^2)^{-1} \left[ \sech\left(\frac{A + 2}{2A}\alpha^{-1}\right)h_0\px u \right] \right\|_{L^2} \\
		&+ \left\|(1-\gamma\px^2)^{-1} \left[ \sech\left(\frac{A + 2}{2A}\alpha^{-1}\right)h_0' u \right] \right\|_{L^2} \\
		\lesssim & \gamma^{-\frac12}\left\|  \sech\left(\frac{A + 2}{2A}\alpha^{-1}\right) \px u \right\|_{L^2} + \left\| \sech\left(\frac{A + 2}{2A}\alpha^{-1}\right)\wtilQ^2 u \right\|_{L^2} \\
		\lesssim&  \gamma^{-\frac12} \left\|\sigma_A \wtilQ^\frac12 \px u\right\|_{L^2} + \left\|\sigma_A \wtilQ^\frac52 u\right\|_{L^2},
	\end{align*}
	which proves \eqref{eq:tech 22}.
\end{proof}

\begin{lemma}\label{lem:tech estimates}
	One has
	\begin{enumerate}
		\item[(1)] \label{l:estimate w} Estimate on $w_1$.
		\begin{equation}\label{eq:tech w}
			\left\|\sigma_A \wtilQ^\frac12 \px(\tilde\chi_B u_1)\right\|_{L^2} \lesssim \left\|\wtilQ^\frac12\px w_1\right\|_{L^2} + \left\|\wtilQ^\frac32 w_1\right\|_{L^2}.
		\end{equation}
		\item[(2)] \label{l:estimate v} Estimates on $v_1$.
		\begin{gather}\label{eq:tech v1}
			\left\|\sigma_A\wtilQ^{\frac32} v_1\right\|_{L^2} \lesssim \gamma^{-\frac12}\left\|\wtilQ^{\frac32} w_1\right\|_{L^2}, \\[0.1cm] \label{eq:tech px v1}
			\left\|\sigma_A\wtilQ^\frac12 \px v_1\right\|_{L^2} \lesssim \gamma^{-\frac12}\left(\left\|\wtilQ^\frac12\px w_1\right\|_{L^2} + \left\|\wtilQ^\frac32 w_1\right\|_{L^2}\right).
		\end{gather}
	\end{enumerate}
\end{lemma}

\begin{remark}
	Compared with previous results in \cite{KMM19,KMMV20}, Lemma \ref{lem:tech estimates} contains new weighted estimates because of the variable coefficients in the model and the emergence of new weighted terms as well. 
\end{remark}

\begin{proof} Proof of \eqref{eq:tech w}. Using that $\sigma_A \lesssim \zeta_A$, $\tilde\chi_B' \lesssim \wtilQ $ and that from definition \eqref{eq:zeta} $\zeta_A'\lesssim A^{-1}\wtilQ\zeta_A$, we have
	\[
	\begin{aligned}
		\left\|\sigma_A \wtilQ^\frac12 \px (\tilde\chi_B u_1)\right\|_{L^2} \lesssim &~{} \left\|\zeta_A\wtilQ^\frac12 \px u_1\right\|_{L^2} + \left\|\zeta_A\wtilQ^\frac32 u_1\right\|_{L^2} \\
		\lesssim &~{} \left\| \wtilQ^\frac12 \px w_1\right\|_{L^2} + \left\|\wtilQ^\frac32 w_1\right\|_{L^2} + \left\|\wtilQ^\frac12 \zeta_A' u_1\right\|_{L^2} \\
		\lesssim &~{}  \left\|\wtilQ^\frac12 \px w_1\right\|_{L^2} + \left\|\wtilQ^\frac32 w_1\right\|_{L^2}.
	\end{aligned}	
	\]
	
	Proof of \eqref{eq:tech v1}. Estimate \eqref{eq:tech v1} is direct from \eqref{eq:tech 11}, using $\sigma_A \lesssim \zeta_A$ and  \eqref{tildeu}.
	
	Now, using \eqref{eq:tech 22} and \eqref{eq:tech w} we have
	\[
	\begin{aligned}
		\left\|\sigma_A \wtilQ^\frac12\px v_1\right\|_{L^2} \lesssim &~{} \gamma^{-\frac12} \left\| \sigma_A \wtilQ^\frac12 \px (\tilde\chi_B u_1)\right\|_{L^2} + \left\|\sigma_A \wtilQ^\frac52 \tilde\chi_B u_1\right\|_{L^2} \\
		\lesssim &~{} \gamma^{-\frac12}\|\wtilQ^\frac12\px w_1\|_{L^2} + \gamma^{-\frac12} \left\|\wtilQ^\frac32 w_1\right\|_{L^2} + \left\|\sigma_A \wtilQ^\frac52 u_1\right\|_{L^2} \\
		\lesssim &~{} \gamma^{-\frac12}\left( \left\| \wtilQ^\frac12\px w_1\right\|_{L^2} + \left\|\wtilQ^\frac32 w_1\right\|_{L^2} \right),
	\end{aligned}
	\]
	obtaining \eqref{eq:tech px v1}. 
\end{proof}

\subsection{Controlling error and nonlinear terms.}
Now we have in a position to control the error and nonlinear terms in \eqref{eq:error terms}. By the definition of $\zeta_B$ and $\tilde\chi_A$ in \eqref{eq:virial II notation}, it holds that
\begin{equation}\label{eq:estimates}
	\begin{gathered}
		\zeta_B(x) \lesssim e^{-\frac1B|\alpha^{-1}(x)|}, \quad |\zeta_B'(x)| \lesssim \frac{1}{B}\wtilQ e^{-\frac1B |\alpha^{-1}(x)|}, \quad |\varphi_B|\lesssim B, 	\\
		|\tilde\chi_A'|\lesssim \frac1A\wtilQ, \quad |(\tilde\chi_A^2)'|\lesssim \frac1A\wtilQ, \quad |\tilde\chi_A''|\lesssim \frac1A\wtilQ^2,  \quad  |(\tilde\chi_A^2)'''|\lesssim \frac{1}{A}\wtilQ^3.
	\end{gathered}
\end{equation}
Even more, from the definition of $\chi$ in \eqref{eq:chi} we have
\begin{equation}\label{eq:support chi}
	\tilde\chi_A'(x) = \tilde\chi_A''(x) = \tilde\chi_A'''(x) = 0,
\end{equation}
if $|\alpha^{-1}(x)|<A$ or if $|\alpha^{-1}(x)|>2A$.

\subsubsection{Control of $\wtilJ_1$.}
Let us now recall the definition of $\wtilJ_1$:
\begin{align}\label{def:J1}
	\begin{split}
		\wtilJ_1 =&~{}  \frac14 \int \wtilQ(\tilde\chi_A^2)'(\zeta_B^2)' v_1^2 + \frac14 \int \wtilQ'(\tilde\chi_A^2)'\zeta_B^2 v_1^2 + \frac12 \int\wtilQ[3(\tilde\chi_A')^2 + \tilde\chi_A''\tilde\chi_A] \zeta_B^2 v_1^2 \\
		& + \frac14\int (\tilde\chi_A^2)'''\varphi_B v_1^2 - \int (\tilde\chi_A^2)'\varphi_B(\px v_1)^2 \\
		= &~{} J_{1,1} + J_{1,2} + J_{1,3} + J_{1,4} + J_{1,5}.
	\end{split}
\end{align}

For the first four terms, using that $\sigma_A\gtrsim 1$ on $[-2\alpha(A), 2\alpha(A)]$, \eqref{eq:estimates} and \eqref{eq:support chi}, we have
\begin{equation}\label{eq:J1 estimates}
	\begin{gathered}
		|(\tilde\chi_A^2)'(\zeta_B^2)'|\lesssim \frac{1}{AB}e^{-2\frac{A}{B}}\sigma_A^2\wtilQ^2, \quad |(\tilde\chi_A^2)'\zeta_B^2|\lesssim \frac1A e^{-2\frac{A}{B}}\sigma_A^2\wtilQ, \\
		(\tilde\chi_A')^2\zeta_B^2 + |\tilde\chi_A''\tilde\chi_A|\zeta_B^2 \lesssim \frac1A e^{-2\frac{A}{B}}\sigma_A^2\wtilQ^2, \\
		|(\tilde\chi_A^2)'''\varphi_B| \lesssim \frac{B}{A}\sigma_A^2\wtilQ^3, \quad |(\tilde\chi_A^2)'\varphi_B|\lesssim \frac{B}{A}\sigma_A^2\wtilQ.
	\end{gathered}
\end{equation}
Thus, using the above estimates and \eqref{eq:tech v1}, we have for the terms in \eqref{def:J1},
\[
\begin{aligned}
	|J_{1,1}| + |J_{1,2}| + |J_{1,3}| + |J_{1,4}| \lesssim &~{} \frac{B}{A}\left\|\sigma_A\wtilQ^\frac32 v_1\right\|_{L^2}^2 \lesssim  \gamma^{-1}\frac{B}{A}\left\| \wtilQ^\frac32 w_1\right\|_{L^2}^2.
\end{aligned}
\]
In the case of $J_{1,5}$, using \eqref{eq:J1 estimates} and \eqref{eq:tech px v1}, we obtain
\[
\begin{aligned}
	|J_{1,5}| \lesssim &~{} \frac{B}{A}\left\| \sigma_A \wtilQ^\frac12 \px v_1\right\|_{L^2}^2
	\lesssim \gamma^{-1}\frac{B}{A}\left(\left\|\wtilQ^\frac12 \px w_1\right\|_{L^2}^2 +\left\|\wtilQ^\frac32 w_1\right\|_{L^2}^2\right).
\end{aligned}
\]
Therefore, we conclude for this term
\begin{equation}\label{eq:control J1}
	\left| \wtilJ_1 \right| \lesssim \gamma^{-1}\frac{B}{A} \left(\int \wtilQ(\px w_1)^2 + \int \wtilQ^3 w_1^2\right).
\end{equation}

\subsubsection{Control of $J_2$.} Recall $J_2$ from \eqref{los Ji}. 
First, by the Cauchy-Schwarz inequality,
\[
|J_2| \lesssim \gamma \left\| \wtilQ (1 - \gamma\px^2)^{-1}\left(\psi_{A,B} \px v_1 + \frac12\psi_{A,B}' v_1 \right) \right\|_{L^2} \left\| \wtilQ^{-1} (V_0''v_1 + V_0' \px v_1)\right\|_{L^2}.
\]
Using the commutativity estimate \eqref{eq:conmutativity sech}, \eqref{eq:basic inequalities} and $\wtilQ \lesssim \sech(\alpha^{-1}) \lesssim \wtilQ$,
\[
\begin{aligned}
	\|\wtilQ (1 - \gamma\px^2)^{-1}(\psi_{A,B} \px v_1)\|_{L^2} \lesssim &~{} \|\sech(\alpha^{-1})(1 - \gamma\px^2)^{-1}(\psi_{A,B} \px v_1)\|_{L^2} \\
	\lesssim &~{} \|(1 - \gamma\px^2)^{-1}(\sech(\alpha^{-1})\psi_{A,B} \px v_1)\|_{L^2} \\
	\lesssim &~{} \| \sech(\alpha^{-1})\psi_{A,B} \px v_1\|_{L^2} \lesssim  \|\wtilQ \psi_{A,B} \px v_1\|_{L^2}.
\end{aligned}
\]
From the definition of $z$ in \eqref{eq:z}, we have
\[
\px z = \tilde\chi_A \zeta_B \px v_1 + (\tilde\chi_A\zeta_B)'v_1 \quad \Longrightarrow \quad \tilde\chi_A^2 \zeta_B^2(\px v_1)^2 \lesssim (\px z)^2 + |(\tilde\chi_A\zeta_B)'v_1|^2.
\]
Using \eqref{eq:estimates} and again the definition of $z$
\[
|(\tilde\chi_A\zeta_B)'v_1|^2\tilde\chi_A^2 \lesssim \left(\frac1A + \frac1B\right)^2 \wtilQ^2 \tilde\chi_A^2 \zeta_B^2 v_1^2 \lesssim \frac{1}{B^2}\wtilQ^2 z^2,
\]
and so
\[
\tilde\chi_A^4 \zeta_B^2 (\px v_1)^2 \lesssim \tilde\chi_A^2(\px z)^2 + \frac{1}{B^2} \wtilQ^2 z^2.
\]
Thus, using $|\psi_{A,B}|\leq |\alpha^{-1}(x)|\tilde\chi_A^2$,
\[
\begin{aligned}
	\wtilQ^2|\psi_{A,B}\px v_1|^2 \lesssim &~{}  |\alpha^{-1}(x)|^2\wtilQ^2\tilde\chi_A^4(\px v_1)^2 \\
	\lesssim &~{} \wtilQ\tilde\chi_A^4\zeta_B^2(\px v_1)^2 \lesssim \wtilQ(\px z)^2 + \frac{1}{B^2} \wtilQ^3 z^2.
\end{aligned}
\]
So, it follows that
\begin{equation}\label{eq:psidv1}
	\left\| \wtilQ \psi_{A,B}\px v_1\right\|_{L^2} \lesssim \left(\wtilQ(\px z)^2 + \frac{1}{B^2}\wtilQ^3 z^2\right)^\frac12.
\end{equation}
Proceeding as before and using  \eqref{eq:conmutativity sech}, \eqref{eq:basic inequalities}, for the other term we obtain
\[
\left\|\wtilQ (1 - \gamma\px^2)^{-1}(\psi_{A,B}' v_1)\right\|_{L^2} \lesssim \left\|\wtilQ \psi_{A,B}'v_1\right\|_{L^2}.
\]
Now, we claim
\begin{equation}\label{eq:dpsi}
	(\psi_{A,B}')^2 \le \frac{21}{10} \wtilQ^2 \tilde\chi_A^2.
\end{equation}
Indeed, using \eqref{psi_computations} and \eqref{eq:estimates}, the definition of $\psi_{A,B}$ in \eqref{eq:virial II notation} and that $\tilde\chi_A = 0$ for $|\alpha^{-1}(x)|\geq 2A$,
\[
\begin{aligned}
	(\psi_{A,B}')^2 = &~{}  \left[ (\tilde\chi_A^2)'\varphi_B + \tilde\chi_A^2\wtilQ\zeta_B^2 \right]^2 \\
	\le &~{}  8 (\tilde\chi_A\tilde\chi_A' \varphi_B)^2 + 2\wtilQ^2\tilde\chi_A^4 \zeta_B^4 \le  \wtilQ^2 \tilde\chi_A^2 \left(  \frac{C}{A^2}B^2 + 2\right) \le \frac{21}{10}\wtilQ^2 \tilde\chi_A^2.
\end{aligned}
\]
Using \eqref{eq:dpsi}, we have that
\[
\left| \wtilQ^2(\psi_{A,B}')^2 v_1^2 \right| \le \frac{21}{10} \wtilQ^4 \tilde\chi_A^2 v_1^2 \le \frac{63}{20} \wtilQ^3 z^2,
\]
and so from \eqref{eq:z},
\begin{equation}\label{eq:dpsiv1}
	\left\|\wtilQ\psi_{A,B}v_1\right\|_{L^2}^2 \le 4 \int \wtilQ^3 z^2.
\end{equation}
Collecting \eqref{eq:psidv1} and \eqref{eq:dpsiv1} we have
\begin{equation}\label{eq:first term J2}
\ba
&	\left\| \wtilQ (1 - \gamma\px^2)^{-1}\left(\psi_{A,B} \px v_1 + \frac12\psi_{A,B}' v_1 \right) \right\|_{L^2} \\
&\qquad \lesssim \left( \int \wtilQ(\px z)^2 + \wtilQ^3 z^2 \right)^{\frac12}.
\ea
\end{equation}

Now we estimate the term related with the potential $V_0$. By Lemma \ref{lem:decay of V0} we have $|V_0'| \lesssim \wtilQ^3$, and using that
\[|h_0|\lesssim 1, \qquad h_0' = \frac{1}{4h_0}(V_0' + V'), \qquad h_0'' = V' - 2h_0h_0'\]
with
\[
V_0'' = 4(h_0')^2 + 4h_0 h_0'' - V'',
\]
one has $|V_0''|\lesssim \wtilQ^3$. Combining the above estimates,
\[
|V_0''v_1| + |V_0'\px v_1| \lesssim \wtilQ^3|v_1| + \wtilQ^3|\px v_1|,
\]
so
\[
\left\| \wtilQ^{-1} (V_0''v_1 + V_0' \px v_1)\right\|_{L^2} \lesssim \left\| \wtilQ^2 v_1\right\|_{L^2} + \left\| \wtilQ^2 \px v_1\right\|_{L^2}.
\]
From the definition of $z$ in \eqref{eq:z} and the particular polynomial decay of $\zeta_B$ and $\wtilQ$, we have
\begin{equation}\label{eq:v to z}
	\wtilQ^{\frac12} \tilde\chi_A^2 v_1^2  \lesssim \tilde\chi_A^2\zeta_B^2v_1^2 = z^2.
\end{equation}
Thus, using the above and from the definition of $\tilde\chi_A$,
\[
\begin{aligned}
	\wtilQ^4 v_1^2 = &~{} \wtilQ^4 v_1^2\tilde\chi_A^2 + \wtilQ^4 v_1^2(1 - \tilde\chi_A^2) \lesssim \wtilQ^\frac72 z^2 + e^{-\frac{A}{2}}\wtilQ^\frac72 v_1^2.
\end{aligned}
\]
From this, and using that $\wtilQ^\frac14\lesssim \sigma_A$ for $A$ large enough, it follows that
\[
\|\wtilQ^2 v_1\|_{L^2} \lesssim \|\wtilQ^\frac74 z\|_{L^2} + e^{-\frac{A}{2}}\|\wtilQ^\frac74 v_1\|_{L^2} \lesssim \|\wtilQ^\frac32 z\|_{L^2} + e^{-\frac{A}{4}}\|\sigma_A\wtilQ^\frac32 v_1\|_{L^2}.
\]
By estimate \eqref{eq:tech v1} we obtain
\begin{equation}\label{eq:Q2v1}
	\|\wtilQ^2 v_1\|_{L^2} \lesssim \|\wtilQ^\frac32 z\|_{L^2} + \gamma^{-\frac12}e^{-\frac{A}{4}}\|\wtilQ^\frac32 w_1\|_{L^2}.
\end{equation}
For the other term $\|\wtilQ^2\px v_1\|$, differentiating $z = \tilde\chi_A\zeta_B v_1$ we obtain
\[
\tilde\chi_A \zeta_B \px v_1 = \px z - \frac{\zeta_B'}{\zeta_B}z - \tilde\chi_A'\zeta_B v_1.
\]
Thus, from the properties of $\zeta_B$ and $\tilde\chi_A$ in \eqref{eq:estimations} and \eqref{eq:estimates} we get
\begin{equation}\label{eq:pxv to z}
	|\tilde\chi_A \zeta_B \px v_1| \lesssim \px z + \frac1B\wtilQ z.
\end{equation}
Replacing and using the polynomial decay of $\zeta_B$, we have
\[
\begin{aligned}
	\wtilQ^4(\px v_1)^2 = &~{} \wtilQ^4(\px v_1)^2\tilde\chi_A^2 + \wtilQ^4(\px v_1)^2(1 - \tilde\chi_A^2) \\
	\lesssim &~{} \wtilQ^\frac{7}{2}(\px z)^2 + \frac1B \wtilQ^\frac{11}{2} z^2 + e^{-A}\wtilQ^3(\px v_1)^2.
\end{aligned}
\]
Integrating over $\R$ and using \eqref{eq:tech px v1}, we obtain
\begin{equation}\label{eq:Q2pxV2}
	\begin{aligned}
		\| \wtilQ^2\px v_1\|_{L^2} \lesssim &~{} \|\wtilQ^\frac74\px z\|_{L^2} + \frac{1}{\sqrt{B}}\|\wtilQ^\frac{11}{4}z\|_{L^2} + e^{-\frac{A}{2}}\|\wtilQ^\frac74\px v_1\|_{L^2} \\
		\lesssim &~{} \|\wtilQ^\frac12\px z\|_{L^2} + \frac{1}{\sqrt{B}}\|\wtilQ^\frac{3}{2}z\|_{L^2} + e^{-\frac{A}{2}}\|\sigma_A \wtilQ^\frac32\px v_1\|_{L^2} \\
		\lesssim &~{} \|\wtilQ^\frac12\px z\|_{L^2} + \frac{1}{\sqrt{B}}\|\wtilQ^\frac{3}{2}z\|_{L^2} + \gamma^{-\frac12}e^{-\frac{A}{2}}\left(\|\wtilQ^\frac12 \px w_1\|_{L^2} + \|\wtilQ^\frac32 w_1\|_{L^2}\right).
	\end{aligned}
\end{equation}
It follows using \eqref{eq:Q2v1} and \eqref{eq:Q2pxV2} that
\begin{equation}\label{eq:V0v1}
\begin{aligned}
& 	\|\wtilQ^2 v_1\|_{L^2} + \|\wtilQ^2 \px v_1\|_{L^2} \\
&\hspace{2cm} \lesssim \|\wtilQ^\frac12 \px z\|_{L^2} + \|\wtilQ^\frac32 z\|_{L^2} + \gamma^{-\frac12}e^{-\frac{A}{4} }\left(\|\wtilQ^\frac12 \px w_1\|_{L^2} + \|\wtilQ^\frac32 w_1\|_{L^2} \right).
\end{aligned}
\end{equation}
Therefore, collecting the estimates \eqref{eq:first term J2} and \eqref{eq:V0v1} we conclude
\begin{equation}\label{eq:control J2}
	|J_2|\lesssim \gamma\left(\int \wtilQ (\px z)^2 + \wtilQ^3 z^2 \right) + e^{-\frac{A}{4}}\left(\int \wtilQ(\px w_1)^2 + \wtilQ^3 w_1^2 \right).
\end{equation}

\subsubsection{Control of $J_3$.} From \eqref{eq:virial II notation}, we recognize that $\psi_{A,B}$ and $\psi_{A,B}'$ are terms supported in $|\alpha^{-1}(x)| \leq 2A$ because of $\tilde \chi_A^2(x)$ and $\tilde \chi_A(x)\tilde \chi_A'(x)$. Using Cauchy-Schwarz inequality, we have
\begin{equation}\label{eq:CS J3}
	\begin{aligned}
		\left| J_3 \right|\lesssim &~{} \left(\left\|\wtilQ^\frac12 \tilde \chi_A^{-1}\psi_{A,B}\px v_1\right\|_{L^2} + \left\|\wtilQ^\frac12 \tilde \chi_A^{-1}\psi_{A,B}'v_1\right\|_{L^2} \right) \\
		&~{}\quad \left\|\wtilQ^{-\frac12} \tilde\chi_A (1 - \gamma\px^2)^{-1}U (\tilde\chi_B' \px u_1 + \tilde\chi_B'' u_1)\right\|_{L^2}.
	\end{aligned}
\end{equation}
For the term in parenthesis, using \eqref{eq:virial II notation}, $|\varphi_{B}|\lesssim B$, estimate \eqref{eq:tech px v1} and that $\zeta_A \gtrsim 1$ on $[-2\alpha(A), 2\alpha(A)]$,
\[
\begin{aligned}
	\left\|\wtilQ^\frac12\tilde \chi_A^{-1} \psi_{A,B} \px v_1\right\|_{L^2} \lesssim &~{} B\left\|\wtilQ^\frac12\tilde\chi_A \px v_1\right\|_{L^2} \lesssim  \gamma^{-\frac12}B \left( \left\|\wtilQ^\frac12\px w_1\right\|_{L^2} + \left\|\wtilQ^\frac32 w_1\right\|_{L^2} \right).
\end{aligned}
\]
On the other hand, since $\psi_{A,B}=\tilde\chi_{A}^2\varphi_B$ (see \eqref{eq:zeta} and \eqref{eq:virial II notation}), using \eqref{psi_computations} and \eqref{eq:estimates},
\begin{align}\label{psi_AB_p}
	|\psi_{A,B}'| \leq |(\tilde\chi_A^2)'\varphi_B| + \wtilQ\tilde\chi_A^2 \zeta_B^2 \lesssim \frac{B}{A}\wtilQ\tilde\chi_A1_{ \{ A \leq |\alpha^{-1}(x)| \leq 2A \} }   + \wtilQ \tilde\chi_A^2 \zeta_B^2.
\end{align}
From \eqref{psi_AB_p}, using \eqref{eq:tech v1} and \eqref{eq:z}, it follows
\[
\begin{aligned}
\left\|\wtilQ^\frac{1}{2} \tilde\chi_A^{-1}\psi_{A,B}'v_1 \right\|_{L^2} \lesssim & \frac{B}{A}\left\|\wtilQ^\frac32 1_{ \{ A \leq |x| \leq 2A \} } v_1\right\|_{L^2} + \left\|\wtilQ^{3/2} z\right\|_{L^2} \\
\lesssim & \gamma^{-\frac12} \frac{B}{A}\left\|\wtilQ^\frac32 w_1\right\|_{L^2} + \left\|\wtilQ^{3/2} z\right\|_{L^2}.
\end{aligned}
\]
Collecting these estimates, we obtain
\begin{equation}\label{eq:J4 1}
\begin{aligned}
	&\left\|\wtilQ^\frac12 \tilde \chi_A^{-1} \psi_{A,B}\px v_1\right\|_{L^2} + \left\|\wtilQ^\frac12 \tilde \chi_A^{-1} \psi_{A,B}'v_1\right\|_{L^2} \\
	&\hspace{3cm} \lesssim \gamma^{-\frac12}B\left(\left\|\wtilQ^\frac12 \px w_1\right\|_{L^2} + \left\|\wtilQ^\frac32 w_1\right\|_{L^2} \right) + \left\|\wtilQ^{3/2} z\right\|_{L^2}.
\end{aligned}
\end{equation}
For the second term in \eqref{eq:CS J3}, using $1 \lesssim \sigma_A \lesssim \zeta_A$ on $[-2\alpha(A), 2\alpha(A)]$, $U = \px - h_0$ with $h_0\lesssim 1$ and estimate \eqref{eq: tech 3},
\[
\left\|\wtilQ^{-\frac12} \tilde\chi_A (1 - \gamma\px^2)^{-1}U (\tilde\chi_B'' u_1)\right\|_{L^2} \lesssim \gamma^{-\frac12}\left\|\sigma_A \wtilQ^{-\frac12} \tilde\chi_B'' u_1\right\|_{L^2}.
\]
Now, we use that $\tilde\chi_B'' \lesssim B^{-4}\wtilQ^2$ on $[-2\alpha(B^2), 2\alpha(B^2)]$, and so
\[
\left\|\wtilQ^{-\frac12} \tilde\chi_A (1 - \gamma\px^2)^{-1}U (\tilde\chi_B'' u_1)\right\|_{L^2} \lesssim \gamma^{-\frac12}B^{-4}\left\|\sigma_A \wtilQ^{\frac32} u_1\right\|_{L^2} \lesssim \gamma^{-\frac12}B^{-4}\left\|\wtilQ^{\frac32} w_1\right\|_{L^2} .
\]
Repeating this procedure, we obtain
\[
\begin{aligned}
\left\|\wtilQ^{-\frac12} \tilde\chi_A (1 - \gamma\px^2)^{-1}U (\tilde\chi_B' \px u_1)\right\|_{L^2} \lesssim & \gamma^{-\frac12}\left\|\sigma_A \wtilQ^{-\frac12} \tilde\chi_B' \px u_1\right\|_{L^2} \\
\lesssim& \gamma^{-\frac12}B^{-2}\left\|\wtilQ^{\frac12} \px w_1\right\|_{L^2}
\end{aligned}
\]
In conclusion,
\begin{equation}\label{eq:J3 2}
\begin{aligned}
	&\left\|\wtilQ^{-\frac12} \tilde\chi_A (1 - \gamma\px^2)^{-1}U (\tilde\chi_B' \px u_1 + \tilde\chi_B'' u_1)\right\|_{L^2} \\
	&\hspace{3cm}\lesssim \gamma^{-\frac12}B^{-2}\left(\left\|\wtilQ^\frac12 \px w_1\right\|_{L^2} + \left\|\wtilQ^\frac32 w_1\right\|_{L^2} \right).
\end{aligned}
\end{equation}

Collection \eqref{eq:J4 1} and \eqref{eq:J3 2}, we obtain
\begin{equation}\label{eq:control J3}
\begin{aligned}
	\left| J_3 \right|\lesssim& \gamma^{-1}B^{-1}\left(\left\|\wtilQ^\frac12 \px w_1\right\|_{L^2} + \left\|\wtilQ^\frac32 w_1\right\|_{L^2} \right)^2 \\
	&+ \gamma^{-\frac12}B^{-2}\left(\left\|\wtilQ^\frac12 \px w_1\right\|_{L^2} + \left\|\wtilQ^\frac32 w_1\right\|_{L^2} \right)\left\|\wtilQ^{3/2} z\right\|_{L^2}.
\end{aligned}
\end{equation}

\subsubsection{Control of $J_4$.} 
Recall $J_4$ from \eqref{los Ji}.  We need now the explicit version of $N$ as in \eqref{eq:Nexplicit}. We decouple
\[
N= N_{g} + N_b,
\]
with
\begin{equation}\label{eq:Nest_g}
	N_g:= \wtilQ^2 \left( 3\wtilH(a_1^2\phi_0^2  +2a_1 \phi_0 u_1 ) + a_1^3 \phi_0^3 + 3 a_1^2 \phi_0^2 u_1 + 3 a_1\phi_0 u_1^2  \right),
\end{equation}
\begin{equation}\label{eq:Nest_b}
	N_b:=   \wtilQ^2u_1^2  \left( 3\wtilH  + u_1\right),
\end{equation}
and
\[
N^\perp = N - N_0\phi_0 =  \left( N_g   - N_0\phi_0 \right)+ N_b:= N_g^\perp + N_b. 
\]
Also, consider $J_4=J_{4,g} + J_{4,b}$, where one replaces $N_g^\perp$ and $N_b$, respectively. Consequently,
\[
\begin{aligned}
	J_{4,g}= &~{}  -\int \left(\psi_{A,B} \px v_1 + \frac12\psi_{A,B}' v_1 \right)(1 - \gamma \px^2)^{-1}U (\tilde\chi_B N_g^{\perp}).
\end{aligned}
\]
Using the Cauchy-Schwarz inequality, we have
\begin{equation}\label{eq:CS J4}
	\begin{aligned}
		\left| J_{4,g} \right|\lesssim& \left(\left\|\wtilQ^\frac12 \tilde \chi_A^{-1}\psi_{A,B}\px v_1\right\|_{L^2} + \left\|\wtilQ^\frac12 \tilde \chi_A^{-1}\psi_{A,B}'v_1\right\|_{L^2} \right)\\
		&\quad \left\|\wtilQ^{-\frac12} \tilde\chi_A (1 - \gamma\px^2)^{-1}U (\tilde\chi_B N_g^{\perp})\right\|_{L^2}.
	\end{aligned}
\end{equation}
For the first term we use \eqref{eq:J4 1} as before. It remains to bound the second term in \eqref{eq:CS J4}. Using that $N_g^\perp = N_g - N_0\phi_0$, we split it in two parts as follows
\[
\begin{aligned}
	&\left\|\wtilQ^{-\frac12} \tilde\chi_A (1 - \gamma\px^2)^{-1}U(\tilde\chi_B N_g^{\perp})\right\|_{L^2} \\
	& \hspace{1.5cm} \leq 
	\left\| \wtilQ^{-\frac12}\tilde\chi_A(1 - \gamma\px^2)^{-1}U(\tilde\chi_B N_g) \right\|_{L^2} + |N_0| \left\|\wtilQ^{-\frac12}\tilde\chi_A(1 - \gamma\px^2)^{-1}U(\tilde\chi_B \phi_0) \right\|_{L^2}.
\end{aligned}
\]
Now, we recall the estimate of $N_0$ obtained \eqref{eq:N0} and using that $|a_1|\lesssim 1$ we give a pointwise estimate for $N_g$ in \eqref{eq:Nest_g},
\begin{equation}\label{eq:estimation N}
	\begin{gathered}
		|N_g| \lesssim \wtilQ^2 \phi_0  (a_1^2 + u_1^2), \\
		|N_0| \lesssim a_1^2 + \int \wtilQ^2\phi_0 u_1^2 \lesssim a_1^2 +  \left\| \wtilQ^{1/2} u_1\right\|_{L^\infty}  \left\|\wtilQ^\frac32 w_1\right\|_{L^2}.
	\end{gathered}
\end{equation}
Thus, using $1 \lesssim \sigma_A \lesssim \zeta_A$ on $[-2\alpha(A), 2\alpha(A)]$, $U = \px - h_0$ with $h_0\lesssim 1$ and estimate \eqref{eq: tech 3},
\begin{equation*}
	\begin{aligned}
		\left\| \wtilQ^{-\frac12} \tilde\chi_A (1 - \gamma\px^2)^{-1}U( \tilde\chi_B N_g) \right\|_{L^2} \lesssim &~{} \left\|\sigma_A \wtilQ^{-\frac12} (1 - \gamma\px^2)^{-1}U(\tilde\chi_B N_g)\right\|_{L^2} \\
		\lesssim &~{} \gamma^{-\frac12} \left\| \sigma_A\wtilQ^{-\frac12}N_g\right\|_{L^2}.
	\end{aligned}
\end{equation*}
Inserting the pointwise estimate \eqref{eq:estimation N} into this, it follows from \eqref{def:sigmaA} and \eqref{tildeu} that
\begin{equation}\label{eq:J4 2}
	\begin{aligned}
		\left\| \wtilQ^{-\frac12} \tilde\chi_A (1 - \gamma\px^2)^{-1}U( \tilde\chi_B N_g) \right\|_{L^2} \lesssim& \gamma^{-\frac12}\left(a_1^2 \left\|\sigma_A\wtilQ^\frac32 \phi_0 \right\|_{L^2} + \left\|\sigma_A\wtilQ^\frac32  \phi_0 u_1^2\right\|_{L^2} \right) \\
		\lesssim& \gamma^{-\frac12}\left(a_1^2 + \left\| \wtilQ^{1/2} u_1\right\|_{L^\infty} \left\| \wtilQ^\frac32 w_1\right\|_{L^2} \right).
	\end{aligned}
\end{equation}
For the remaining term, using the exponential decay of $\phi_0$, \eqref{eq: tech 3} and \eqref{eq:estimation N} we have
\begin{equation}\label{eq:J4 3}
	\begin{aligned}
		\left\| \wtilQ^{-\frac12} \tilde\chi_A (1 - \gamma\px^2)^{-1}U\phi_0 \right\|_{L^2} \lesssim & \left\| \sigma_A \wtilQ^{-\frac12} (1 - \gamma\px^2)^{-1}U\phi_0 \right\|_{L^2} \\
		\lesssim & \gamma^{-\frac12}\left\|\sigma_A\wtilQ^{-\frac12}\phi_0\right\|_{L^2} \lesssim \gamma^{-\frac12}.
	\end{aligned}
\end{equation}

Now combining the preceding estimates \eqref{eq:J4 1}, \eqref{eq:J4 2}, \eqref{eq:J4 3} and \eqref{eq:estimation N} with \eqref{eq:CS J4} yields
\begin{equation}\label{eq:control J4g}
	\begin{aligned}
		|J_{4,g}|\lesssim &~{} \gamma^{-1}B\left(\left\|\wtilQ^\frac12 \px w_1\right\|_{L^2} + \left\|\wtilQ^\frac32 w_1\right\|_{L^2} \right)\left(a_1^2 + \left\| \wtilQ^{1/2}  u_1\right\|_{L^\infty}\left\|\wtilQ^\frac32 w_1\right\|_{L^2}\right) \\
		&~{} + \gamma^{-\frac12} \left\|\wtilQ^\frac32 z\right\|_{L^2}\left(a_1^2 + \left\| \wtilQ^{1/2} u_1\right\|_{L^\infty}\left\|\wtilQ^\frac32 w_1\right\|_{L^2}\right).
	\end{aligned}
\end{equation}
Finally, we consider \eqref{eq:Nest_b} and $J_{4,b}$:
\[
\begin{aligned}
	J_{4,b}= &~{}  \int \left(\psi_{A,B} \px v_1 + \frac12\psi_{A,B}' v_1 \right)(1 - \gamma \px^2)^{-1}U \left(\tilde\chi_B  N_b \right).
\end{aligned}
\]
Recall $w_1$,  $z$ and $v_1$ defined in \eqref{tildeu}, \eqref{eq:z} and \eqref{eq:def v1 v2}, respectively. Also, from \eqref{eq:J4 1} one has
\begin{equation}\label{partida1}
\begin{aligned}
	& \left\| \wtilQ^{1/2} \tilde \chi_A^{-1}\left( \psi_{A,B} \px v_1 + \frac12\psi_{A,B}' v_1 \right) \right\|_{L^2} \\
	&\hspace{3cm} \lesssim  \gamma^{-\frac12}B\left(\left\|\wtilQ^\frac12 \px w_1\right\|_{L^2} + \left\|\wtilQ^\frac32 w_1\right\|_{L^2} \right) + \left\|\wtilQ^{3/2} z\right\|_{L^2}.
\end{aligned}
\end{equation}
First of all, using \eqref{eq: tech 3}, that $\tilde\chi_A \lesssim \sigma_A$, and \eqref{eq:Nest_b},
\[
\begin{aligned}
	\left\| \wtilQ^{-1/2}\tilde\chi_A (1 - \gamma \px^2)^{-1} U(\tilde\chi_B N_b) \right\|_{L^2} \lesssim &~{} \left\|\sigma_A \wtilQ^{-1/2} (1 - \gamma \px^2)^{-1} U(\tilde\chi_B N_b) \right\|_{L^2} \\
	\lesssim &~{} \gamma^{-\frac12}\left(\left\|\sigma_A \tilde\chi_B \wtilQ^{3/2} u_1^2\right\|_{L^2} + \left\|\sigma_A \tilde\chi_B \wtilQ^{3/2} u_1^3 \right\|_{L^2}\right).
\end{aligned}
\]
Now, using that $1 \lesssim \alpha(B)^2\wtilQ$ on $[-2\alpha(B^2), 2\alpha(B^2)]$ and $\sigma_A \lesssim \zeta_A$, we have
\[
\begin{aligned}
	& \left\| \wtilQ^{-1/2}\tilde\chi_A (1 - \gamma \px^2)^{-1} U(\tilde\chi_B N_b) \right\|_{L^2} \\
	&\quad \lesssim \gamma^{-\frac12} \alpha(B) \left\| \wtilQ^{1/2} u_1 \right\|_{L^\infty} \left\|\sigma_A \wtilQ^{3/2} u_1\right\|_{L^2} +  \gamma^{-\frac12} \alpha(B)^2\left\| \wtilQ^{1/2} u_1 \right\|_{L^\infty}^2 \left\|\sigma_A \wtilQ^{3/2} u_1\right\|_{L^2} \\
	&\quad \lesssim \gamma^{-\frac12}  \alpha(B)\left\| \wtilQ^{1/2} u_1 \right\|_{L^\infty} \left(1 + \alpha(B)\left\| \wtilQ^{1/2} u_1 \right\|_{L^\infty} \right) \left\| \wtilQ^{3/2} w_1 \right\|_{L^2}.
\end{aligned}
\]
This last estimate, together with \eqref{partida1}, are good enough to conclude. Indeed,
\begin{equation}\label{eq:control J4b}
	\begin{aligned}
		\left| J_{4,b} \right| \lesssim &~{} \gamma^{-1}B\alpha(B)\left\| \wtilQ^{1/2} u_1 \right\|_{L^\infty} \\
		&~{}\quad \left(1 + \alpha(B)\left\| \wtilQ^{1/2} u_1 \right\|_{L^\infty} \right) \left(\left\|\wtilQ^\frac12 \px w_1\right\|_{L^2} + \left\|\wtilQ^\frac32 w_1\right\|_{L^2} \right) \left\| \wtilQ^{3/2} w_1 \right\|_{L^2} \\
		&~{} + \gamma^{-\frac12}B\left\| \wtilQ^{1/2} u_1 \right\|_{L^\infty} \left(1 + \alpha(B)\left\| \wtilQ^{1/2} u_1 \right\|_{L^\infty} \right)\left\|\wtilQ^{3/2} z\right\|_{L^2}  \left\| \wtilQ^{3/2} w_1 \right\|_{L^2}
	\end{aligned}
\end{equation}
Gathering  \eqref{eq:control J4g} and \eqref{eq:control J4b}, we obtain
\begin{equation}\label{eq:control J4}
	\begin{aligned}
		|J_{4}|\lesssim &~{} \gamma^{-1}B\left(\left\|\wtilQ^\frac12 \px w_1\right\|_{L^2} + \left\|\wtilQ^\frac32 w_1\right\|_{L^2} \right)\\
		&~{}\quad \left(a_1^2 + \alpha(B)\left(1 + \alpha(B)\left\| \wtilQ^{1/2} u_1 \right\|_{L^\infty} \right)\left\| \wtilQ^{1/2}  u_1\right\|_{L^\infty}\left\|\wtilQ^\frac32 w_1\right\|_{L^2}\right) \\
		&~{} + \gamma^{-\frac12} \left\|\wtilQ^\frac32 z\right\|_{L^2}\left(a_1^2 + B\left(1 + \alpha(B)\left\| \wtilQ^{1/2} u_1 \right\|_{L^\infty} \right)\left\| \wtilQ^{1/2} u_1\right\|_{L^\infty}\left\|\wtilQ^\frac32 w_1\right\|_{L^2}\right).
	\end{aligned}
\end{equation} 

\subsection{End of Proposition \ref{prop:virial II}}
Gathering \eqref{eq:error terms},  \eqref{eq:control J1}, \eqref{eq:control J2}, \eqref{eq:control J3} and \eqref{eq:control J4}, it follows that there exist constants $C_2, C_3>0$ such that
\begin{equation*}
	\begin{aligned}
		\frac{d}{dt}\calJ \leq &~{} -4C_2 \int \wtilQ[(\px z)^2 + \wtilQ^2 z^2] + \gamma^{-1}\frac{C_3 B}{A} \int \wtilQ[(\px w_1)^2 + \wtilQ^2 w_1^2] \\
		&~{}  + \gamma C_3 \int \wtilQ [(\px z)^2 + \wtilQ^2 z^2] + C e^{-\frac{A}{4}} \int \wtilQ[(\px w_1)^2 + \wtilQ^2 w_1^2] \\
		&~{} + \gamma^{-1}B^{-1}\left(\left\|\wtilQ^\frac12 \px w_1\right\|_{L^2} + \left\|\wtilQ^\frac32 w_1\right\|_{L^2} \right)^2 \\
		&~{} + \gamma^{-\frac12}B^{-2}\left(\left\|\wtilQ^\frac12 \px w_1\right\|_{L^2} + \left\|\wtilQ^\frac32 w_1\right\|_{L^2} \right)\left\|\wtilQ^{3/2} z\right\|_{L^2} \\[0.1cm]
		&~{} + \gamma^{-1}C_3 B\left(\left\|\wtilQ^\frac12 \px w_1\right\|_{L^2} + \left\|\wtilQ^\frac32 w_1\right\|_{L^2} \right)~{}\quad \\
		&~{}\quad \left(a_1^2 + \alpha(B)\left(1 + \alpha(B)\left\| \wtilQ^{1/2} u_1 \right\|_{L^\infty} \right)\left\| \wtilQ^{1/2}  u_1\right\|_{L^\infty}\left\|\wtilQ^\frac32 w_1\right\|_{L^2}\right) \\[0.1cm]
		&~{} + \gamma^{-\frac12} C_3\left\|\wtilQ^\frac32 z\right\|_{L^2}\left(a_1^2 + B\left(1 + \alpha(B)\left\| \wtilQ^{1/2} u_1 \right\|_{L^\infty} \right)\left\| \wtilQ^{1/2} u_1\right\|_{L^\infty}\left\|\wtilQ^\frac32 w_1\right\|_{L^2}\right).
	\end{aligned}
\end{equation*}
We fix $\gamma > 0$ such that $\gamma C_3 \leq C_2$ and also small enough to satisfy Lemma \ref{lem:commutativity} and Lemma \ref{lem:localized estimation}.

The value of $\gamma$ being now fixed, we do not mention anymore dependency of $\gamma$. Via standard inequalities and for $A$ large enough, we obtain,  for a possibly larger constant $C_3>0$,
\[
\begin{aligned}
	\frac{d}{dt}\calJ \leq &~{} -C_2 \int \wtilQ[(\px z)^2 + \wtilQ^2 z^2] + C_3\left(\frac1B + \frac{B}{A} + e^{-\frac{A}{4}}\right) \int \wtilQ[(\px w_1)^2 + \wtilQ^2 w_1^2]  \\
	&~{} + C_3 B\alpha(B)\left(a_1^2 +  B\left(1 + \alpha(B)\left\| \wtilQ^{1/2} u_1 \right\|_{L^\infty} \right)\left\|\wtilQ^{1/2} u_1\right\|_{L^\infty}\left\|\wtilQ^\frac32 w_1\right\|_{L^2}\right)^2.
\end{aligned}
\]
Since $A = \delta^{-\frac14}$ and $B = \alpha^{-1}(\delta^{-\frac18})\lesssim \delta^{-\frac18}$ (see \eqref{eqn:equivalencias}, \eqref{eq:choiceA}, \eqref{eq:choiceB}), using assumption \eqref{eq:bound} and standard inequalities, we have
\[
\begin{gathered}
	\alpha(B)\left\| \wtilQ^{1/2} u_1 \right\|_{L^\infty} \lesssim \delta^{\frac78} \lesssim 1, \quad B^{-1} + A^{-1}B + e^{-\frac{A}{4}} \lesssim \ln(\delta^{-\frac18})^{-1} \\[0.1cm]
	B\alpha(B)\left(B \left\|\wtilQ^{1/2} u_1\right\|_{L^\infty}\left\|\wtilQ^\frac32 w_1\right\|_{L^2}\right)^2 \lesssim \delta^{-\frac12} \left\|\wtilQ^{1/2} u_1\right\|_{L^\infty}^2\left\|\wtilQ^\frac32 w_1\right\|_{L^2}^2 \lesssim \delta^{\frac32}\left\|\wtilQ^\frac32 w_1\right\|_{L^2}^2.
\end{gathered}
\]
Therefore, using again \eqref{eq:bound}, for $\delta$ small enough (to absorb some constants), we obtain
\[
\begin{aligned}
	\frac{d}{dt}\calJ \leq & -C_2\int \wtilQ[(\px z)^2 + \wtilQ^2 z^2] \\
	& + C_3 \ln(\delta^{-\frac18})^{-1} \int \wtilQ[(\px w_1)^2 + \wtilQ^2 w_1^2] + C_3\delta^{\frac12}|a_1|^3.
\end{aligned}
\]
This ends the proof of \eqref{eqn:bound_dtJ}.

\section{Proof of Theorem \ref{th:Main}}\label{sec:proof theorem 1}
Before starting the proof of Theorem \ref{th:Main}, we need a coercivity result to deal with the term
\[
\int \wtilQ^7 u_1^2
\]
for $n\in \N$ that appears in the virial estimate of $\calI(t)$ (see \eqref{eq:1virial}), being a term with enough decay to be controlled by the variables $(v_1,v_2)$ and $(z_1, z_2)$. In this section, the constant $\gamma$ is fixed as in Proposition \ref{prop:virial II}.

\subsection{Coercivity}
We prove a coercivity result adapted to the orthogonality condition $\langle u_1, \phi_0\rangle = \langle u_1, L\phi_0\rangle = 0$ in \eqref{eq:perp-conditions}, where $\phi_0$ was introduced in \eqref{eq:properties-eigenvalL}. The idea is to follow the strategy used in \cite{KMM19}, where the linearized operator has an explicit unique negative single eigenvalue $\tau_0$ associated with an explicit $L^2$ eigenfunction denoted $Y_0$. Despite our system we only have the existence of such negative eigenvalue $-\mu_0^2$ associated with $\phi_0$, we still have this control given by orthogonality.

\begin{lemma}\label{lem5p1}
	Let $u$ and $v$ be measurable functions related by
	\begin{equation}\label{def:v}
		v = (1 - \gamma\px^2)^{-1}U (\tilde \chi_B u)
	\end{equation}
	and such that $\langle u, \phi_0\rangle = 0$, the following estimate holds
	\begin{equation}\label{eq:coercivity}
		\int \wtilQ^7 u^2 \lesssim \int \wtilQ^\frac92 [(\px v)^2 + v^2] + e^{-3B^2}\int \wtilQ^4 u^2,
	\end{equation}
	provided the RHS is finite.
\end{lemma}
\begin{proof}
	First, since $\wtilQ^3(1 - \tilde \chi_B) \lesssim e^{-3B^2}$ (see \eqref{tilde Q tilde H} and \eqref{eq:virial II notation}) we get
	\be\label{intQ7}
	\int \wtilQ^7 u^2 \lesssim \int \wtilQ^7 \tilde\chi_B^2 u^2 + e^{-3B^2}\int \wtilQ^4 u^2.
	\ee
	In what follows, we estimate the localized term on the right in \eqref{intQ7}. Using that $U = \phi_0\cdot\px\cdot \phi_0^{-1} $, we rewrite \eqref{def:v} as
		\[
		v - \gamma \px^2v = \phi_0 \px\left(\frac{\tilde \chi_B u}{\phi_0}\right).
		\]
		and thus, after some algebra
		\[
		\px\left(\frac{\tilde \chi_B u}{\phi_0} + \gamma \frac{\px v}{\phi_0}\right) = \frac{1}{\phi_0}\left(v - \gamma h_0 \px v \right)
		\]
		where $h_0=\phi_0'/\phi_0$ (see \eqref{def:h0}). Integrating between 0 and $x>0$, it follows
		\[
		\frac{\tilde \chi_B u}{\phi_0} + \gamma \frac{\px v}{\phi_0} = a + \int_0^x \frac{1}{\phi_0}\left(v - \gamma h_0 \px v \right)
		\]
		for some constant $a$. If we rewrite this last expression, multiplying by $\phi_0$, it follows
		\begin{equation}\label{eq:coer u}
			\tilde \chi_B u = a\phi_0 - \gamma \px v + \tilu,
		\end{equation}
	where
	\[
	\tilu = \phi_0 \int_0^x \frac{1}{\phi_0}\left(v - \gamma h_0 \px v \right).
	\]
	Let us now estimate $\tilu$. First, using the Cauchy-Schwarz inequality, a change of variables, and recalling that $\phi_0$ is even and decreasing for $x>0$, we have
	\be\label{cambio de variables}
	\begin{aligned}
		\phi_0\int_0^x \frac{|v|}{\phi_0} \lesssim &~{} \phi_0 \left(\int\wtilQ^\frac92 v^2 \right)^\frac12 \left(\int_0^x \frac{1}{\wtilQ^\frac92 \phi_0^2} \right)^\frac12 \\
		\lesssim &~{} \|\wtilQ^\frac94 v\|_{L^2} \left(\int_0^{\alpha^{-1}(x)} \frac{1}{Q^\frac{11}{2}} \right)^\frac12 \lesssim \wtilQ^{-\frac{11}{4}} \left\|\wtilQ^\frac94 v\right\|_{L^2}.
	\end{aligned}
	\ee
	Similarly, using that $|h_0|\lesssim 1$,
	\[
	\ba
	\phi_0\int_0^x \frac{|h_0\px v|}{\phi_0} \lesssim &~{}\phi_0 \left(\int \wtilQ^\frac92 (h_0 \px v)^2  \right)^\frac12 \left(\int_0^x \frac{1}{\wtilQ^\frac92 \phi_0^2} \right)^\frac12 \\
	\lesssim &~{}\wtilQ^{-\frac{11}{4}} \left\|\wtilQ^\frac94 \px v\right\|_{L^2}.
	\ea
	\]
	Collecting these estimates, we obtain the uniform bound
	\label{50}
		\[
		\wtilQ^\frac{11}{2} \tilu^2 \lesssim  \int \wtilQ^\frac92 [(\px v)^2 + v^2],
		\]
	for all $x\geq 0$. The same result holds for $x\leq 0$.
	Therefore, multiplying by $\wtilQ^{\frac{3}{2}}$, and integrating we obtain
	\be\label{bound utilde}
	\int \wtilQ^7 \tilu^2 \lesssim \left(\int \wtilQ^{\frac{3}{2}}\right)\left(\int \wtilQ^\frac92 [(\px v)^2 + v^2]\right) \lesssim \int \wtilQ^\frac92 [(\px v)^2 + v^2].
	\ee
	Using that $\langle u,\phi_0\rangle = 0$ and \eqref{eq:properties-eigenvalL}, we have
		\be\label{eq for a}
		\langle \tilde \chi_B u, \phi_0 \rangle = \langle (\tilde \chi_B  - 1)u, \phi_0 \rangle =  a - \gamma \langle\px v,\phi_0\rangle + \langle \tilu, \phi_0\rangle.
		\ee
		Thus, using the Cauchy-Schwarz inequality, the explicit exponential decay of $\phi_0$ (see \eqref{eq:exp-decay}) and the definition of $\tilde\chi_B$ in \eqref{eq:virial II notation}, we obtain
		\[
		\ba
		|\langle (\tilde \chi_B - 1)u, \phi_0\rangle| & \lesssim \|\wtilQ^{-4}\phi_0\|_{L^2}\left( \int (\tilde \chi_B - 1)^2 \phi_0 \wtilQ^4 u^2 \right)^{\frac12} \\
		& \lesssim e^{-\frac{\sqrt2}{4}\mu_0\alpha(B^2)}\left(\int \wtilQ^4 u^2 \right)^{\frac12}.
		\ea
		\]
		Replacing \eqref{bound utilde} and the above estimate in \eqref{eq for a},
		\[
		\begin{aligned}
			& a^2 \lesssim \int \wtilQ^\frac92[(\px v)^2 + v^2] + e^{-\frac{\sqrt2}{2}\mu_0\alpha(B^2)}\int \wtilQ^4 u^2.
		\end{aligned}
		\]
	Using again \eqref{eq:coer u}, we deduce
	\be\label{intQ7tilde}
	\int \wtilQ^7 \tilde\chi_B^2 u^2 \lesssim \int \wtilQ^\frac92[(\px v)^2 + v^2] + e^{-\frac{\sqrt2}{2}\mu_0\alpha(B^2)}\int \wtilQ^4 u^2.
	\ee
	We conclude \eqref{eq:coercivity} from \eqref{intQ7}, \eqref{intQ7tilde}, and recalling the definition of $\alpha$ given in \eqref{eq:alpha}.
\end{proof}

As result of the previous lemma, we have the following transfer estimate from the variable $u_1$ to the transformed and localized variable $z$ introduced in \eqref{eq:z}.
\begin{lemma}\label{lem5p2}
	Let $(u_1, u_2)$ be solution of \eqref{eq:system-perturbated} satisfying \eqref{eq:perp-conditions}, $(w_1, w_2)$ be as in \eqref{tildeu}, and $z$ as in \eqref{eq:z}. Then, for any $A$ and $B$ as in \eqref{eq:scales}, it holds
	\begin{equation}\label{eq:transfer u tx|o z w}
		\int \wtilQ^7 u_1^2 \lesssim \int \wtilQ^2[(\px z)^2 + \wtilQ^2 z^2] + e^{-3B^2} \int \wtilQ[(\px w_1)^2 + \wtilQ^2 w_1^2].
	\end{equation}
\end{lemma}
\begin{proof}
	Since $u_1$ satisfies the orthogonality condition \eqref{eq:perp-conditions}, applying \eqref{eq:coercivity}
		\[
		\int \wtilQ^7 u_1^2 \lesssim \int \wtilQ^\frac92 [(\px v_1)^2 + v_1^2] + e^{-3B^2}\int \wtilQ^4 u_1^2,
		\]
		which implies, using $\wtilQ\lesssim \zeta_A^2$ and by \eqref{tildeu},
		\[
		\int \wtilQ^7 u_1^2 \lesssim \int \wtilQ^\frac92 [(\px v_1)^2 + v_1^2] + e^{-3B^2}\int \wtilQ^3 w_1^2.
		\]
	Using that $\wtilQ\lesssim e^{-|\alpha^{-1}(x)|}$, $\wtilQ^\frac14 \lesssim \zeta_B^2$, \eqref{eq:v to z} and \eqref{eq:pxv to z}, it follows
	\[
	\begin{aligned}
		\wtilQ^\frac92 [(\px v_1)^2 + v_1^2] \lesssim &  \int\wtilQ^4\zeta_B^2 (\px v_1)^2 + \int\wtilQ^4\zeta_B^2 v_1^2 \\
		\lesssim & \int \wtilQ^4[(\px z)^2 + \wtilQ^2 z^2] + \int \wtilQ^4 z^2 \\
		& + e^{-3A}\int \wtilQ \zeta_B^2 (1-\tilde\chi_A^2)(\px v_1)^2 + e^{-A}\int \wtilQ^3 \zeta_B^2 (1-\tilde\chi_A^2) v_1^2,
	\end{aligned}
	\]
	and since $\zeta_B\lesssim \zeta_A \lesssim \sigma_A$, using \eqref{eq:tech v1} and \eqref{eq:tech px v1},
	\[
	\begin{aligned}
		\wtilQ^\frac92 [(\px v_1)^2 + v_1^2] \lesssim& \int \wtilQ^2[(\px z)^2 + \wtilQ^2 z^2] + e^{-3A} \int \sigma_A^2\wtilQ(\px v_1)^2 + e^{-A}\int \sigma_A^2 \wtilQ^3 v_1^2 \\
		\lesssim& \int \wtilQ^2[(\px z)^2 + \wtilQ^2 z^2] + e^{-A} \int \wtilQ[(\px w_1)^2 + \wtilQ^2 w_1^2].
	\end{aligned}
	\]		
	and the asserted estimate \eqref{eq:transfer u tx|o z w} follows from \eqref{eq:scales}, \eqref{eqn:equivalencias} and \eqref{eq:scales}.
\end{proof}

\subsection{Proof of Theorem \ref{th:Main}.}
Recall that the constants $\gamma >0$, $\delta_1, \delta_2 >0$ were defined and fixed in Propositions \ref{prop:virial I} and \ref{prop:virial II}.

\medskip

In this section we prove Theorem \ref{th:Main}, in particular the conditional asymptotic stability property \eqref{eq:AS}. In this case, the orthogonality conditions \eqref{eq:perp-conditions} and the dynamical equations satisfied by $(a_1,a_2)$ in \eqref{eq:a1a2} will be of key importance. It turns out that $(b_1,b_2)$ as in \eqref{eq:b1b2} are better suited variables to fully catch the exponential unstable behavior of the full system.

\begin{proposition}\label{propvirales}
	There exist $C_4>0$ and $0<\delta_3 \leq \min(\delta_1,\delta_2)$ such that for any $0<\delta\leq \delta_3$, the following holds. Fix 
		\begin{equation}\label{AB_def}
			A=\delta^{-\frac14}\quad  \hbox{and} \quad B = \alpha^{-1}(\delta^{-\frac18}).
		\end{equation}
	Assume that for all $t\geq 0$, \eqref{eq:bound} holds. Let
	\begin{equation}\label{def:H}
		\calH = \calJ + 8M C_0^{-1}\ln(\delta_3^{-\frac18})^{-1} \calI,
	\end{equation}
	with $M = \max\{C, C_3\}>0$, where $C_0, C>0$ are the constants from Proposition \ref{prop:virial I}, and $C_3$ is the constant from Proposition \ref{prop:virial II}.
	
	Then, for all $t\geq 0$,
	\begin{equation}\label{eq:dcalH}
		\frac{d}{dt}\calH \leq - C_4 \int \wtilQ\left[(\px w_1)^2 + \wtilQ^2 w_1^2\right] + |a_1|^3.
	\end{equation}
\end{proposition}
\begin{proof}
	In the context of Propositions \ref{prop:virial I} and \ref{prop:virial II}, observe that fixing $A=\delta^{-\frac14}$ and $B = \alpha^{-1}(\delta^{-\frac18}),$ for $\delta>0$ small is consistent with the requirement of scales in \eqref{eq:scales}.
	
	First, combining \eqref{eq:1virial} with \eqref{eq:transfer u tx|o z w}, for $\delta_3>0$ small enough and $0<\delta \leq \delta_3$, we obtain for some constants $C_0, C>0$ fixed, and possibly choosing a smaller $\delta_3$,
	\begin{equation*}
		\begin{aligned}
			\frac{d}{dt}\mathcal{I} \leq &~{} - \frac12 C_0\int \wtilQ [(\partial_x w_1)^2 + \wtilQ^2 w_1^2] + C\int \wtilQ[(\px z)^2 + \wtilQ^2 z^2] \\
			&~{} + Ce^{-3B^2} \int \wtilQ[(\px w_1)^2 + \wtilQ^2 w_1^2] + C\delta_3 |a_1|^3 \\
			\leq &~{} - \frac14 C_0 \int \wtilQ [(\partial_x w_1)^2 + \wtilQ^2 w_1^2] + C\int \wtilQ[(\px z)^2 + \wtilQ^2 z^2] + |a_1|^3.
		\end{aligned}
	\end{equation*}
	Secondly, for $\frac{d}{dt}\mathcal{J}$, using \eqref{eqn:bound_dtJ} and $0<\delta\leq \delta_3$, we get for $C_2, C_3 >0$ fixed,
	\begin{equation}\label{emeC}
	\begin{aligned}
		\frac{d}{dt}\calJ \leq &~{}  -C_2\int \wtilQ[(\px z)^2 + \wtilQ^2 z^2] \\
	& ~{} + C_3\ln(\delta_3^{-\frac18})^{-1} \int \wtilQ[(\px w_1)^2 + \wtilQ^2 w_1^2] + C_3\delta^\frac12|a_1|^3.
	\end{aligned}
	\end{equation}
	Therefore, defining $\calH$ as in \eqref{def:H} and by combining the above estimates, it follows that
	\begin{equation*}
		\begin{aligned}
			\frac{d}{dt}\calH \leq &~{} \left(-C_2 + 8M^2 C_0^{-1}\ln(\delta_3^{-\frac18})^{-1}\right) \int \wtilQ[(\px z)^2 + \wtilQ^2 z^2] \\
			&~{} - M \ln(\delta_3^{-\frac18})^{-1} \int \wtilQ[(\px w_1)^2 + \wtilQ^2 w_1^2] \\
			&~{} + M\left(\delta^{\frac12} + 8C_0^{-1}\ln(\delta_3^{-\frac18})^{-1} \right) |a_1|^3.
		\end{aligned}
	\end{equation*}
	Thus, possible choosing a smaller $\delta_3$ (in particular, $0<\delta_3^{\frac18} \leq e^{-\frac{16C^2}{C_0 C_2}}$), we obtain
	\[
	\begin{aligned}
	\frac{d}{dt}\calH \leq &~{} -\frac{C_2}{2}\int \wtilQ[(\px z)^2 + \wtilQ^2 z^2] \\
	&~{} - C\ln(\delta_3^{-\frac18})^{-1} \int \wtilQ[(\px w_1)^2 + \wtilQ^2 w_1^2] + |a_1|^3.
	\end{aligned}
	\]
	We have that \eqref{eq:dcalH} follows directly from the above estimate where $C_4 = M \ln(\delta_3^{-\frac18})^{-1}>0$.
\end{proof}
We define now
\begin{equation}\label{eq:calB}
	\calB = b_+^2 - b_-^2,
\end{equation}
where $b_+$, $b_-$ are given in \eqref{eq:b1b2}.
\begin{lemma}\label{lem:bound internal modes}
	There exist $C_5>0$ and $0<\delta_4\leq \delta_3$ such that for any $0<\delta\leq \delta_4$, the following holds. Fix $A = \delta^{-\frac14}$. Assume that for all $t\geq 0$, \eqref{eq:bound} holds. Then, for all $t\geq0$,
	\begin{equation}\label{eq:b+ b-}
		|\dot b_+ - \mu_0 b_+| + |\dot b_- + \mu_0 b_-| \leq C_5\left( b_+^2 + b_-^2 + \int \wtilQ^3 w_1^2 \right),
	\end{equation}
	and
	\begin{equation}\label{eq:db+ db-}
		\left| \frac{d}{dt} (b_+^2) - 2\mu_0 b_+^2 \right| + \left|\frac{d}{dt} (b_-^2) + 2\mu_0 b_-^2 \right| \leq C_5 \left(b_+^2 + b_-^2 + \int \wtilQ^3 w_1^2\right)^\frac32.
	\end{equation}
	In particular, for $\calB$ in \eqref{eq:calB}:
	\begin{equation}\label{eq:dcalB}
		\frac{d}{dt}\calB \geq \mu_0(b_+^2 + b_-^2) - C_5\int \wtilQ^3 w_1^2 = \frac{\mu_0}{2} (a_1^2 + a_2^2) - C_5\int \wtilQ^3 w_1^2.
	\end{equation}
\end{lemma}
\begin{proof}
	From \eqref{eq:estimation N} and \eqref{eq:b1b2}, it holds
	\[
	|N_0|\lesssim a_1^2 + \int \wtilQ^3 w_1^2 \lesssim b_+^2 + b_-^2 + \int\wtilQ^3 w_1^2.
	\]
	From \eqref{eq:a1a2} we conclude the estimates \eqref{eq:b+ b-} and \eqref{eq:db+ db-}. Finally, estimate \eqref{eq:dcalB} is a consequence of \eqref{eq:db+ db-} taking $\delta_4>0$ small enough.
\end{proof}

Combining \eqref{eq:dcalH} and \eqref{eq:dcalB}, it holds
\[
\frac{d}{dt}\left(\calB - 2C_5 C_4^{-1}\calH\right) \geq \frac{\mu_0}{2}(a_1^2 + a_2^2) + C_5\int \wtilQ[(\px w_1)^2 + \wtilQ^2 w_1^2] -2C_5|a_1|^3,
\]
and thus, for possibly smaller $\delta>0$,
\begin{equation}\label{eq:d calB-calH}
	\frac{d}{dt}\left(\calB - 2C_5C_4^{-1}\calH\right) \geq \frac{\mu_0}{4}(a_1^2 + a_2^2) + C_5\int \wtilQ[(\px w_1)^2 + \wtilQ^2 w_1^2].
\end{equation}
By the choice of $A=\delta^{-\frac14}$, the bound $|\varphi_A|\lesssim A$, \eqref{eq:zeta} and \eqref{eq:bound}, we have for all $t\geq 0$,
\[
\left| \calI \right| \leq \left| \int \left( \varphi_A \partial_x u_1 + \frac12 \varphi_A' u_1 \right) u_2 \right|  \lesssim A \left( \|\partial_x u_1\|_{L^2}+  \| \wtilQ u_1\|_{L^2} \right)\|u_2\|_{L^2} \lesssim \delta,
\]
Similarly, using that $U = \px - h_0$, $\psi_{A,B}' = \wtilQ\tilde\chi_A^2 \zeta_B^2 + (\tilde\chi_A^2)' \varphi_B$, \eqref{eq:estimates}, \eqref{eq:basic inequalities} and  \eqref{eq:conmutativity sech}, it holds
\[
\begin{aligned}
	\left| \calJ \right| =&~{} \left|  \int \left(\psi_{A,B} (x) \px v_1(t,x) + \frac12\psi_{A,B}' (x) v_1 (t,x) \right)v_2 (t,x)dx \right|  \\
	\lesssim &~{} B\left( \|\partial_x v_1\|_{L^2}+  \| \wtilQ v_1\|_{L^2} \right) \|v_2\|_{L^2}\lesssim \delta.
\end{aligned}
\]
Then, we have
\[
|\calH|\lesssim \delta.
\]
Estimate $|\calB|\lesssim \delta^2$ is also clear from \eqref{eq:bound}.

Therefore, integrating estimate \eqref{eq:d calB-calH} on $[0,t]$ and passing to the limit as $t\to+\infty$, it follows that
\begin{equation*}
	\int_0^\infty \left\{a_1^2 + a_2^2 + \int \wtilQ[(\px w_1)^2 + \wtilQ^2 w_1^2] \right\}dt \lesssim \delta.
\end{equation*}
Since 
\[
\int \wtilQ^2[(\px u_1)^2 + \wtilQ^\frac32 u_1^2] \lesssim \int \wtilQ[(\px w_1)^2 + \wtilQ^2 w_1^2],
\]
this implies
\begin{equation}\label{eq:int a1 a2 u}
	\int_0^\infty \left\{a_1^2 + a_2^2 + \int \wtilQ^2 \left[ (\px u_1)^2 + \wtilQ^\frac32 u_1^2 \right] \right\}dt \lesssim \delta.
\end{equation}
Making use of the above equation, we will complete the proof of Theorem \ref{th:Main}. Let
\[
\calK = \int u_1 u_2 \wtilQ^2  \quad \text{ and }\quad  \calG = \frac12 \int [(\px u_1)^2 + \wtilQ^\frac32 u_1^2 + u_2^2]\wtilQ^2.
\]
Using \eqref{eq:system-perturbated}, we have
\[
\begin{aligned}
	\frac{d}{dt} \calK = &~{} \int  [\dot u_1 u_2 + u_1 \dot u_2]\wtilQ^2 = \int \left[u_2^2 - u_1\left(Lu_1 + N^\perp \right) \right]\wtilQ^2\\
	=&~{} \int [u_2^2 - (\px u_1)^2 - 2\wtilQ^2(1 - \wtilQ)u_1^2]\wtilQ^2 + \frac12 \int (\wtilQ^2)''u_1^2 - \int N^\perp \wtilQ^2 u_1.
\end{aligned}
\]
From \eqref{eq:estimation N}, the exponential decay of $\phi_0$ and the bound \eqref{eq:bound} we can check that it holds
\[
\int N^\perp\wtilQ^2 u_1 \lesssim a_1^2 + \int \wtilQ^{\frac72} u_1^2.
\]
In particular, collecting the above estimates and using that $(\wtilQ^2)''\lesssim \wtilQ^\frac72$, it follows that there exists some $C>0$ such that
\[
\int \wtilQ^2 u_2^2 \leq \frac{d}{dt} \calK + Ca_1^2 + C \int \wtilQ^2 [(\px u_1)^2 + \wtilQ^\frac32 u_1^2].
\]
From this, the bound $|\calK|\lesssim \delta^2$ and \eqref{eq:int a1 a2 u}, we deduce
\begin{equation}\label{eq:bound calG}
	\int_0^\infty [a_1^2 + a_2^2 + \calG]dt \lesssim \delta.
\end{equation}

Analogously, we compute
\[
\begin{aligned}
	\frac{d}{dt}\calG = &~{} \int [(\px \dot u_1)(\px u_1) + \wtilQ^\frac32 \dot u_1 u_1 + \dot u_2 u_2]\wtilQ^2 \\
	= &~  \int \left[(\px u_2)(\px u_1) + \wtilQ^\frac32 u_2 u_1 - \left(Lu_1 + N^\perp \right)u_2 \right]\wtilQ^2 \\
	= &~{} -2\int \wtilQ\wtilQ' u_2 \px u_1 + \int (2\wtilQ^\frac32 - 2\wtilQ^\frac12 + 1)\wtilQ^\frac72 u_1 u_2 - \int \wtilQ^2 N^\perp u_2,
\end{aligned}
\]
and so, using \eqref{eq:estimation N} as before, we obtain
\begin{equation}\label{eq:control calG}
	\left|\frac{d}{dt} \calG \right| \lesssim a_1^2 + \calG.
\end{equation}
By \eqref{eq:bound calG}, there exists an increasing sequence $t_n\to +\infty$ such that
\[
\lim_{n\to\infty} \left[ a_1^2(t_n) + a_2^2(t_n) + \calG(t_n) \right] = 0.
\]
For $t\geq 0$, integrating \eqref{eq:control calG} on $[t, t_n]$, and passing to the limit as $n\to \infty$, we obtain
\[
\calG(t) \lesssim \int_t^\infty [a_1^2 + \calG]dt.
\]
Using \eqref{eq:bound calG}, we deduce that $\lim_{t\to \infty}\calG(t)= 0$.

Finally, by \eqref{eq:system-perturbated}, \eqref{eq:estimation N} and the exponential decay of $\phi_0$, we get
\[
\left| \frac{d}{dt}(a_1^2) \right| + \left|\frac{d}{dt} (a_2^2) \right| \lesssim a_1^2 + a_2^2 + \int \wtilQ^\frac72 u_1^2.\]
Similarly as before, by integration on $[t, t_n]$ and taking $n\to \infty$,
\[
a_1^2(t) + a_2^2(t) \lesssim \int_t^\infty [a_1^2 + a_2^2 + \calG]dt,
\]
which proves $\lim_{t\to\infty} ( |a_1(t)| + |a_2(t)| )= 0$. By the decomposition of solution the \eqref{eq:decomposition}, this clearly implies \eqref{eq:AS}. The proof of Theorem \ref{th:Main} is complete.

\section{Existence of a stable manifold}\label{sec:proof_manifold}

\subsection{Properties of $L$ and $\widetilde L$} 

Now we provide different characterizations of the operators $L$ and $\widetilde L$ appearing in \eqref{eq:L} and \eqref{mL}, respectively. Notice that $\widetilde L = L - 2\widetilde Q^2 \wtilH^2$.  We start with some basic facts.

\begin{lemma}\label{lema_tilde L}
	Consider $\widetilde L$ appearing in \eqref{mL}. Then the following are satisfied:
	\begin{enumerate}
		\item[(i)] $\widetilde L: L^2 (\mathbb R) \longrightarrow L^2(\mathbb R)$ is a self-adjoint operator with dense domain $H^2(\mathbb R)$. 
		\item[(ii)] The odd function $\wtilH \in L^\infty(\R)$ solves $\wtilL  \wtilH =0$ and has only one zero.
		\item[(iii)] $\widetilde L$ has a unique negative eigenvalue.
		\item[(iv)] $\widehat H$ defined as
		\[
		\begin{aligned}
			\widehat H(x):= &~{}\left( \sinh (\alpha^{-1}(x)) + 3\alpha^{-1}(x) \right) \wtilH (x) -4 =: \beta(x)  \wtilH (x) -4\\
			\beta(x)= &~{} 3x +2\alpha^{-1}(x),
		\end{aligned}
		\]
		is a second, linearly independent solution of $\wtilL u=0$, with Wronskian $W[\wtilH,\widehat H]=\wtilH  \widehat H' -\widehat H \wtilH' =3$, and $\lim_\infty \widehat H =+\infty$ in a linear fashion.
	\end{enumerate}
\end{lemma}

\begin{proof}
	Items $(i)$ and $(ii)$ are direct. By standard Sturm-Liouville theory, $\wtilL $ has a unique negative eigenvalue (see \cite[p. 6]{BizKah16}). This proves $(iii)$. Item $(iv)$ can be checked directly. The proof of Lemma \ref{lema_tilde L} is complete.
\end{proof}

\begin{lemma}\label{lema_virial_final}
	Under $\langle \phi_0, u_1\rangle=0$, one has  
	\[
	\langle \widetilde L u_1, u_1 \rangle \geq 0.
	\]
\end{lemma}

\begin{proof}
	Since $\phi_0$ is even and exponentially decreasing, one has $\langle \phi_0, \wtilH \rangle$ well-defined and equals zero. Let $\phi_1\in L^\infty(\mathbb R)$ be the unique even solution of $\wtilL \phi_1 =\phi_0$ (notice that $\phi_1$ is unique thanks to its even character). It is not difficult to compute a formula for $\phi_1$. Indeed, the most general $\phi_1$ is given by
	\[
	\begin{aligned}
		\phi_1 = &~{} \alpha_{00} \widetilde H+ \beta_{00}\widehat H + \frac13\left(  \widehat H\int^{\infty}_x \phi_0 \wtilH +  \wtilH \int_{0}^x \phi_0 \widehat H \right) . 
	\end{aligned}
	\]
	with $\alpha_{00}, \beta_{00}$ free parameters. The condition $\phi_1$ even forces $\alpha_{00}=0$, and the condition $\phi_1\in L^\infty$ ensures $\beta_{00}=0$. Consequently, $\phi_1$ is unique and given by
	\[
	\phi_1 =  \frac13\left(  \widehat H\int^{\infty}_x \phi_0 \wtilH +  \wtilH \int_{0}^x \phi_0 \widehat H \right). 
	\]
	See Fig. \ref{fig:phi1} for a graph of this function. Additionally,
	\[
	\phi_{1,x} =  \frac13 \left( \widehat H' \int^{\infty}_x \phi_0 \wtilH +  \frac13 \wtilQ^2  \int_{0}^x \phi_0 \widehat H \right) \in L^2 (\mathbb R).
	\]
	One can easily check that $\lim_\infty \phi_{1,x} \widehat H =\lim_\infty \phi_{1} \widehat H_x =0$. Since $\phi_1$ exists, $\widehat H \in S'(\mathbb R)$ and $\widetilde L \phi_1 \in S(\mathbb R)$, naturally the dual pairing $\langle \widetilde L \phi_1 , \widehat H \rangle$ is well-defined and equals  $ \langle \phi_0, \widehat H \rangle$. Consequently, $ \langle \phi_0, \widehat H \rangle =\langle \widetilde L \phi_1 , \widehat H \rangle  =\langle  \phi_1 , \widetilde L \widehat H \rangle =0$. Since both $\phi_1$ and $\widehat H$ are even, we have
	
	\begin{claim}\label{Cl6p3}
		$ \int_0^\infty \phi_0 \widehat H = 0$.
	\end{claim}
	
	As a corollary of this fact, one easily sees that  $\phi_1 \in L^2(\mathbb R)$. Fig. \ref{fig:phi1} shows that $\phi_1$ is probably negative, but this will not be used for the proof. Another consequence of the previous claim is the following: consider the function $g$ defined as
	\[
	[0,\infty) \ni x\longmapsto g(x):= \int_0^x \phi_0 \widehat H .
	\]
	This function is zero at the origin, and because of $\widehat H(0)=-4$, $\phi_0>0$ and $\widehat H$ strictly increasing, at least for $x>0$ small one has  $g(x)<0$. Additionally, $g$ has a unique critical point (where $\widehat H =0$), and converges to a value less or equal than zero as $x\to +\infty$. Therefore, $g(x)<0$ for all $x> 0$. Integrating by parts,
	\[
	\begin{aligned}
		\langle \phi_0,\phi_1 \rangle =&~{} \frac23 \int_0^\infty \phi_0 \left(  \widehat H\int^{\infty}_x \phi_0 \wtilH +\wtilH \int_{0}^x \phi_0 \widehat H\right) \\
		=&~ \frac43 \int_0^\infty \phi_0 \widetilde H \int_0^x \phi_0 \widehat H= \frac43 \int_0^\infty \phi_0 \widetilde H g <0.
	\end{aligned}
	\]
	We conclude that $\langle \phi_1,\phi_0\rangle <0$.  The last inequality implies by classical arguments by Weinstein \cite[Lemma E.1]{Weinstein} that $\langle \wtilL u_1, u_1 \rangle \geq 0$. 
\end{proof}

\begin{figure}[htbp]
	\centering
	\includegraphics[width=0.85\linewidth]{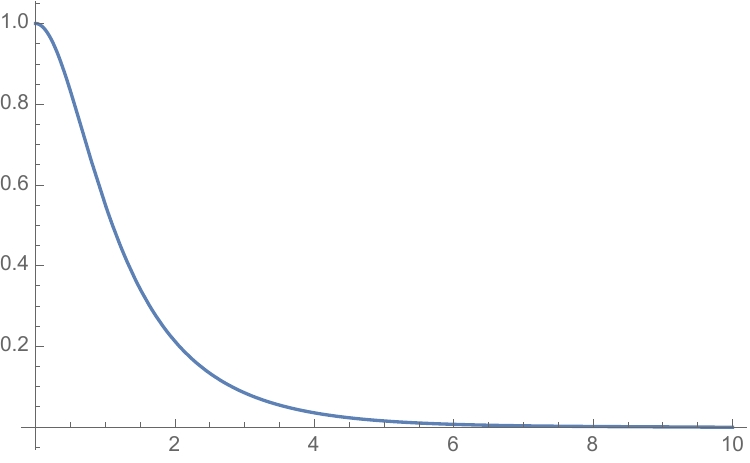}
	\includegraphics[width=0.85\linewidth]{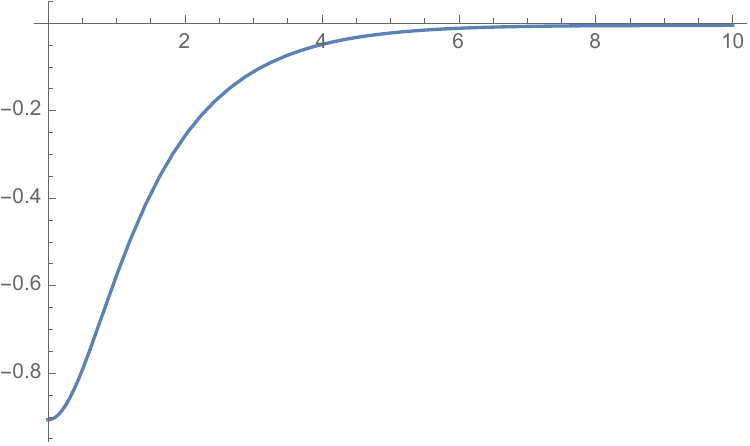}
	\caption{Left: Graph of $\phi_0$ (not rescaled to have unit norm), with associated eigenvalue $\sim -0.658$ and $\mu_0\sim 0.811$ (see Lemma \ref{valor mu0}). Right: Graph of $\phi_1$ solution to $\wtilL\phi_1=\phi_0$, $\phi_1$ even, obtained with $\phi_1(0)=-0.907$.}
	\label{fig:phi1}
\end{figure}

\begin{lemma}\label{lema_final}
	There exists a constant $c_0>0$ such that for any $u\in H_0(\R)$ satisfying $\angles{\phi_0}{u}=\angles{\wtilQ^2\wtilH^3}{u}=0$, one has
	\[
	\angles{\wtilL u}{u} \geq c_0 \|u\|_{H_0}^2.
	\]
\end{lemma}
\begin{proof}
	The proof relies in a similar proof by Weinstein \cite[Prop. 2.9]{Weinstein}. Let define
	\begin{equation}\label{tau}
		\tau = \inf \left\{ \angles{\wtilL u}{u}  : \ \|u\|_{H_0} = 1, \ \angles{\phi_0}{u}=\angles{\wtilQ^2\wtilH^3}{u}=0 \right\}.
	\end{equation}
	We will prove that $\tau > 0$. From Lemma \ref{lema_virial_final} it is sufficient to prove that $\tau = 0 $ leads to a contradiction.
	
	We first prove that $\tau=0$ implies the minimum is attained in the admissible class.
	Given $\set{u_n}$ a minimizing sequence of \eqref{tau} in $H_0(\R)$. Using Lemma \ref{claim:Linfty} and Lemma \ref{lema_virial_final}, for any $\eta>0$ we can choose $u_n$ such that
	\[
	0 < \int (\px u_n)^2 + \wtilQ^3 u_n^2 \leq \frac53 \int\wtilQ^3 u_n^2 + \eta.
	\]
	Since $\{u_n\}$ is uniformly bounded in $H_0(\R)$, we can assume, up to a sequence, that it weakly converges to a function $u_\infty\in H_0(\R)$ as $n\to +\infty$. By the weak convergence and the exponential decay of $\phi_0$ we have that $u_\infty$ satisfies the orthogonal conditions in \eqref{tau}. In addition, the functions $\wtilQ u_n$ are uniformly bounded in $H^1(\R)$, thus we can also assume that $
	\wtilQ u_n \to \wtilQ u_\infty$ as $n\to +\infty$ in $\calC_\text{loc}^0(\R)$. Combining this with the estimates given in Lemma \ref{lemma:estimation} and Lemma \ref{claim:Linfty}, we obtain
	\begin{equation}\label{strg_conv}
		\int \wtilQ^3 (u_n - u_\infty)^2 \to 0, \qquad \hbox{as}\quad  n\to +\infty.
	\end{equation}
	Since $\eta>0$ is arbitrary, this implies $u_\infty\not\equiv 0$.
	
	By Fatou's lemma $\|u_\infty\|_{H_0}\leq 1$. Let us suppose $\|u_\infty\|_{H_0}<1$ and define $v_\infty = u_\infty / \|u_\infty\|_{H_0}$ which is admissible, namely it belongs to the set defining $\tau$ in \eqref{tau}. By the weak convergence of $\px u_n$ and \eqref{strg_conv} we have
	\[
	\angles{\wtilL u_\infty}{u_\infty}\leq \liminf_{n\to\infty}\angles{\wtilL u_n}{u_n} = 0.
	\]
	Hence, $\angles{\wtilL v_\infty}{v_\infty}\leq 0$ and by Lemma \ref{lema_virial_final} the equality is attained. Thus we can take $u_\infty$ satisfying the orthogonality conditions and such that $\|u_\infty\|_{H_0}=1$.
	
	Since the minimum is attained at an admissible function $u_\infty\not\equiv0$, there exist $(u_\infty, \alpha, \beta, \gamma)$ among the critical values of the Lagrange multiplier problem
	\[
	\begin{aligned}
		\wtilL u = \alpha(-\px^2 u + \wtilQ^2 u) + \beta\phi_0 + \gamma\wtilQ^2\wtilH^3  \\
	\end{aligned}
	\]
	such that
	\[
	\|u\|_{H_0}=1, \quad
	\angles{\phi_0}{u}=\angles{\wtilQ^2\wtilH^3}{u}=0.
	\]
	This implies $\alpha = \angles{\wtilL u}{u}$, so $\alpha = \tau = 0$ is a critical value. Therefore, we need to conclude that
	\begin{equation}\label{multiplier}
		\wtilL u_\infty = \beta\phi_0 + \gamma \wtilQ^2\wtilH^3
	\end{equation}
	has no nontrivial solutions $(u_\infty, \beta,\gamma)$ satisfying the constraints. Testing \eqref{multiplier} against $\wtilH$ and integrating by parts, we find that $\gamma = 0$. Therefore, from the proof of Lemma \ref{lema_virial_final} we have that $\wtilL u_\infty = \beta \phi_0$ implies
	\[
	u_\infty = \frac{\beta}{3}\left(  \widehat H\int^{\infty}_x \phi_0 \wtilH +\wtilH \int_{0}^x \phi_0 \widehat H\right),
	\]
	and so $\angles{\phi_0}{u_\infty}\neq0$. This violate the constrains unless $\beta= 0$. Thus  $u_\infty\equiv0$, a contradiction. This conclude the proof of Lemma \ref{lema_final}. 
\end{proof}

\subsection{Improved coercivity estimate}

Additionally, due to the lack of a spectral gap for $L$, we will need a weighted version of a coercivity to control the non-linear term, having the following lemma.
\begin{lemma}\label{lem:coercivity}
	Let $L$ be the operator introduced in \eqref{eq:L}, with essential spectrum $[0,\infty)$ (see Lemma \ref{lemma:Lproperties}). One has that there exists $C_5>0$ such that
	\begin{equation}\label{eq:coer_final}
		\langle Lu, u\rangle \geq C_5 \left( \int (\partial_x  u)^2 + \int \wtilQ^2  u^2 \right),
	\end{equation}
	for all $u\in H^1(\R)$, provided \eqref{eq:perp-conditions} is satisfied.
\end{lemma}

In the following we give a proof of Lemma \ref{lem:coercivity}.
\begin{proof}[Proof of Lemma \ref{lem:coercivity}]
	Recall that, under the orthogonality condition $\langle u, \phi_0 \rangle =0$, one has $\langle Lu, u\rangle \geq 0.$ Now we prove that there is a lower bound given by a suitable  $L^2$ weighted term. Let $\varepsilon_0>0$ be sufficiently small, indeed, $\varepsilon_0= \frac1{1000}$ is good enough. Consider the decomposition
	\[
	L= \varepsilon_0 \left(-\partial_x^2 +\wtilQ^2 \right) + L_{\varepsilon_0},
	\]
	with  
	\[
	L_{\varepsilon_0} := -(1-\varepsilon_0)\partial_x^2 +  2\wtilQ^{2}\left(1 - \frac12\varepsilon_0 - \wtilQ \right).
	\]
	Let us prove that under $\langle u, \phi_0 \rangle =0$, one has $\langle L_{\varepsilon_0} u, u\rangle \geq 0$.
	
	\medskip
	
	The idea of proof is standard. Carefully following Sections \ref{A:LINEAL-SPECTRAL-THEORY} and \ref{B:POSITIVITY-POTENTIAL}, if $\varepsilon_0$ is sufficiently small, one has
	\begin{itemize}
		\item $L_{\varepsilon_0}$ has a negative eigenvalue and essential spectrum $[0,\infty)$ (Lemma \ref{lemma:Lproperties});
		\item $L_{\varepsilon_0}$ has no positive eigenvalues (Lemma \ref{lema_posi}); 
		\item  $L_{\varepsilon_0}$ has a unique negative eigenvalue $-\mu_{\varepsilon_0}^2$ (Corollary \ref{cor:uniqueness}), an associated exponentially decreasing unique $L^2$ normalized eigenfunction $\phi_{\varepsilon_0}$ (Lemma \ref{lem_decay_expo} and Corollary \ref{cor:parity});
		\item Lemma \ref{valor mu0} is also satisfied: by explicit computations one has for the chosen $\varepsilon_0$, and $f$ as in \eqref{test_function},
		\[
		\begin{aligned}
			\langle L_{\varepsilon_0} f,f \rangle =&~{} (1-\varepsilon_0)\int f'^2 +2 \int f^2 \widetilde Q^2\left(1- \frac12\varepsilon_0- \widetilde Q\right)\\
			=&~{} (1-\varepsilon_0)  \int f'^2(x)dx +2 \int f^2(\alpha(y))Q\left(1-  Q \right)(y)dy \\
			&~{} -\varepsilon_0 \int f^2(\alpha(y))Q(y)dy\\
			\sim &~{} -0.6564;
		\end{aligned}
		\]
		so $\mu_0^2\geq 0.656$ and $0.809 \leq \mu_0 $. Additionally, 
		\[
		L\geq  L_{\varepsilon_0} \geq (1-\varepsilon_0) \left(-\partial_x^2 -0.845Q_p^{7/2} \right);
		\]
		Consequently,
		\[
		0.808 \leq \mu_{\varepsilon_0} \leq 0.882.
		\]
		\item If $\langle u_1, \phi_{\varepsilon_0} \rangle =0$, one has $\langle L_{\varepsilon_0} u_1, u_1\rangle \geq 0$. 
	\end{itemize}
	Finally, using a standard argument by Weinstein \cite{Weinstein}, the proof concludes if one shows that $\langle L_{\varepsilon_0}^{-1} \phi_0, \phi_0\rangle <0$.
	
	Using that $L\phi_0 = -\mu_0^2 \phi_0$ and a direct computation from the definitions yielding $L_{\varepsilon_0} = (1 - \varepsilon_0)L + \varepsilon_0 \wtilQ^2(1-2\wtilQ)$, we have
	\[
	L_{\varepsilon_0} \phi_0 = - (1 - \varepsilon_0)\mu_0^2\phi_0 + \varepsilon_0\wtilQ^2(1-2\wtilQ)\phi_0.
	\]
	Applying the operator $L_{\varepsilon_0}^{-1}$ in both sides we obtain
	\[
	L_{\varepsilon_0}^{-1}\phi_0 = -\frac{\phi_0 - \varepsilon_0 L_{\varepsilon_0}^{-1}(\wtilQ^2(1-2\wtilQ) \phi_0)}{\mu_0^2(1 - \varepsilon_0)},
	\]
	and so
	\[
	\langle L_{\varepsilon_0}^{-1} \phi_0, \phi_0\rangle =
	- \frac{1 - \varepsilon_0\langle L_{\varepsilon_0}^{-1}(\wtilQ^2(1-2\wtilQ)\phi_0), \phi_0 \rangle }{\mu_0^2(1 - \varepsilon_0)}.
	\]
	Denoting $u_\varepsilon = L_{\varepsilon}^{-1}(\wtilQ^2(1-2\wtilQ)\phi_0) \in L^\infty(\mathbb R)$, this function must be solution of
	\[
	-(1-\varepsilon)\px^2 u + 2\wtilQ^2\left(1- \frac{\varepsilon}{2} - \wtilQ\right)u = \wtilQ^2(1-2\wtilQ)\phi_0,
	\]
	such that $\partial_x u_\varepsilon \in L^2$ and $\partial_x^2 u_\varepsilon$ decay as $1/x^2$. Testing against $\partial_x^{-2} \phi_0 := \int_{-\infty}^x\int_{-\infty}^y \phi_0(s)dsdy$ (which grows linearly), and using that $\wtilQ$ decays as $1/x$ and $ \phi_0$ exponentially,
	\[
	-(1-\varepsilon) \left\langle \px^2 u_\varepsilon, \partial_x^{-2} \phi_0 \right\rangle + 2 \left\langle \wtilQ^2\left(1- \frac{\varepsilon}{2} - \wtilQ\right)u_\varepsilon, \partial_x^{-2} \phi_0\right\rangle = \left\langle \wtilQ^2(1-2\wtilQ)\phi_0, \partial_x^{-2} \phi_0 \right\rangle,
	\]
	and therefore $\left|  \left\langle u_\varepsilon, \phi_0 \right\rangle \right| \leq C$ independent of $\varepsilon>0$ small.
	Hence, there exist $\varepsilon_1>0$ sufficiently small such that for all $0<\varepsilon_0 \leq \varepsilon_1$,
	\[
	\langle L_{\varepsilon_0}^{-1}\phi_0, \phi_0\rangle \leq - \frac{1 - {\varepsilon_0} C_{\varepsilon_1}}{\mu_0^2 ( 1- \varepsilon_0)} < 0.
	\]
	Defining $C_5 = \min\{\varepsilon_0, \varepsilon_1\}$ we obtain \eqref{eq:coer_final}.
\end{proof}

\subsection{Construction of a stable manifold} The end of the proof is standard and follows \cite{KMM19}, with a main difference given by the control of the resonance. By Lemma \ref{lem:bound internal modes} and a standard contradiction argument, we construct initial data leading to global solution close to the ground state $\wtilH$.

{\bf Step 1.} 
Let $u_1$ be as in \eqref{eq:decomposition}-\eqref{eq:perp-conditions}. Consider \eqref{mL0} with $\bar w_1 = a_1 \phi_0 + u_1$ and $\bar w_2 = a_2\phi_0 + u_2$. One gets
replacing and using orthogonality \eqref{eq:perp-conditions} that
\begin{equation}\label{expansion_buena}
	\begin{aligned}
		2\left\{ E(\phi_1, \phi_2) - E(\wtilH, 0)\right\} =&~ \mu_0^2(a_2^2 - a_1^2) +  \|u_2\|_{L^2}^2 + \langle L u_1, u_1\rangle + 2\int \wtilQ^2\wtilH(a_1\phi_0 + u_1)^3 \\
		& + \frac12\int \wtilQ^2(a_1\phi_0 + u_1)^4 \\
		=&~
		\mu_0^2(a_2^2 - a_1^2) + \|u_2\|_{L^2}^2 + \langle L u_1, u_1\rangle + \frac12\int \wtilQ^2(u_1^2 + 4\wtilH u_1)u_1^2
		\\
		& + 2a_1\int\wtilQ^2\phi_0 u_1^3 + 3a_1\int \wtilQ^2(2\wtilH + a_1\phi_0)\phi_0 u_1^2 \\
		& + 2a_1^2\int\wtilQ^2(3\wtilH + a_1\phi_0)\phi_0^2 u_1 + \frac12 a_1^3 \int \wtilQ^2(4\wtilH + a_1\phi_0)\phi_0^3\\
		=&: I_1+I_2+I_3.
	\end{aligned}
\end{equation}
In the term $I_1$, we complete the square root of the fourth term,
\[
\begin{aligned}
& \frac12\int \wtilQ^2 (u_1^2 + 4\wtilH u_1 + (2\wtilH)^2 - (2\wtilH)^2)u_1^2  = \frac12\int \wtilQ^2 (u_1 + 2\wtilH)^2 u_1^2 - 2\int \wtilQ^2\wtilH^2 u_1^2.
\end{aligned}
\]
Additionally,
\[
V - 2\wtilQ^2\wtilH^2 = 2\wtilQ^2(1 - \wtilQ - \wtilH^2) = 2\wtilQ^2\left( \frac23 \wtilQ - \wtilQ \right) = -\frac23 \wtilQ^3,
\]
obtaining that 
\begin{equation}\label{mL}
\ba
	I_1= &~{} \mu_0^2(a_2^2 - a_1^2) + \|u_2\|_{L^2}^2 + \int \underbrace{\left(-\px^2 u_1 - \frac23 \wtilQ^3 u_1 \right)}_{\wtilL u_1}u_1 +  \frac12\int \wtilQ^2(u_1 + 2\wtilH)^2 u_1^2.
\ea
\end{equation}
Now we perform the necessary estimates for $I_2$ and $I_3$ in \eqref{expansion_buena}. We have
\be\label{okokok}
\ba
\left| 2a_1\int\wtilQ^2\phi_0 u_1^3 \right| \lesssim&~ |a_1| \| \wtilQ^{\frac12} u_1 \|_{L^\infty}  \int  \phi_0\wtilQ ^{\frac32} u_1^2 \\
\lesssim&~ |a_1| \| \wtilQ^{\frac12} u_1 \|_{L^\infty} \| \wtilQ u_1\|_{L^2}^2\\
\lesssim&~  |a_1| \| u_1\|_{H_0}^3,
\ea
\ee
\[
\left| 3a_1\int \wtilQ^2(2\wtilH + a_1\phi_0)\phi_0 u_1^2 \right| \lesssim   |a_1| \| \wtilQ u_1\|_{L^2}^2,
\]	
and
\[
\left| 2a_1^2\int\wtilQ^2(3\wtilH + a_1\phi_0)\phi_0^2 u_1\right| \lesssim C_\epsilon a_1^4 + \epsilon \| \wtilQ u_1\|_{L^2}^2.
\]	
so that
\begin{equation}\label{I_2_cota}
	|I_2| \le C |a_1|^3 + \epsilon \|  u_1\|_{H_0}^2.
\end{equation}
Finally,
\begin{equation}\label{I_3_cota}
	|I_3| =\left| \frac12 a_1^3 \int \wtilQ^2(4\wtilH + a_1\phi_0)\phi_0^3 \right| \lesssim  |a_1|^3.
\end{equation}
Completing the square as in \eqref{mL}, \eqref{I_2_cota} and \eqref{I_3_cota} lead to
\begin{equation}\label{la_clave}
\begin{aligned}
	& -4\mu_0^2 b_+b_-  + \|u_2\|_{L^2}^2 + \langle \widetilde L u_1,u_1\rangle + \int \wtilQ^2(u_1 + 2\wtilH)^2u_1^2  \\
	& \hspace{3.2cm} \leq  C \{ E(\phi_1, \phi_2) - E(\wtilH, 0)\}  + C |a_1|^3 + \epsilon \|u_1\|_{H_0}^2.
\end{aligned}
\end{equation}
We remark that $\widetilde L \widetilde H =0$.
Let us further decompose $u_1$ now as
\[
u_1 = a(t) \widetilde H + \tilde u_1, \quad  \langle \phi_0, \tilde u_1\rangle=\langle \widetilde Q^2 \widetilde H^3, \tilde u_1\rangle=0.
\] 
Clearly $a(t)= \frac{\langle u_1,\wtilQ^2 \wtilH^3\rangle}{ \langle \wtilH, \wtilQ^2 \wtilH^3 \rangle}$. Note that $a(t)$ is well-defined, $ \langle \widetilde L u_1, u_1\rangle = \langle \widetilde L \tilde u_1, \tilde u_1\rangle$, and we have
\be\label{expansionamiento}
\ba
& a^2  \| \widetilde H \|_{H_0}^2 + \|\tilde u_1\|_{H_0}^2 +O\left(|a| \|\tilde u_1\|_{H_0} \right) \\
&\hspace{3cm} =a^2  \| \widetilde H \|_{H_0}^2 + \|\tilde u_1\|_{H_0}^2   + \frac83a \int \wtilQ^3 \widetilde H \tilde u_1 \\
&\hspace{3cm}  =a^2  \| \widetilde H \|_{H_0}^2 + \|\tilde u_1\|_{H_0}^2   + 2a \int   \tilde u_1 \left(- \widetilde H''   + \wtilQ^2\widetilde H \right)\\
&\hspace{3cm}  =  \int (a \widetilde H' )^2 + \int \wtilQ^2 (a \widetilde H)^2 +  \int ( \partial_x \tilde u_1)^2 + \int \wtilQ^2 \tilde u_1^2\\
&\hspace{3cm}  \quad + 2a \int  \widetilde H'  \partial_x \tilde u_1 + 2a \int \wtilQ^2\tilde u_1\widetilde H\\
&\hspace{3cm}  =  \int (a \widetilde H' + \partial_x \tilde u_1)^2 + \int \wtilQ^2 (a \widetilde H + \tilde u_1)^2 =  \|u_1\|_{H_0}^2.
\ea
\ee
Let $\delta_0>0$ be defined by
\[
\delta_0^2 = b_+^2(0) + b_-^2(0) + \|u_2(0)\|_{L^2}^2 + \| u_1(0)\|_{H_0}^2 + \left\|\wtilQ (u_1(0) + 2\wtilH)u_1(0)\right\|_{L^2}^2.
\]
{From \eqref{expansion_buena} and conservation of energy applied at $t=0$, one gets $| E(\phi_1, \phi_2) - E(\wtilH, 0)| \lesssim \delta_0^2$. Thus, from \eqref{la_clave} at some $t>0$ gives
	\begin{equation}\label{por favor}
	\begin{aligned}
	& \|u_2\|_2^2 + \langle \widetilde L \tilde u_1, \tilde u_1\rangle + \int \wtilQ^2(u_1 + 2\wtilH)^2 u_1^2 \lesssim \delta_0^2 + b_+^2 + b_-^2 + |a_1|^3 + \epsilon \norm{u_1}_{H_0}^2.
	\end{aligned}
	\end{equation}
	For the non-linear term, since $\angles{\wtilQ^2\wtilH^3}{\tilu_1}=0$,
	\[
	\begin{aligned}
		\int \wtilQ^2(u_1 + 2\wtilH)^2 u_1^2 =&~{} \int \wtilQ^2 \left( a(2+a)\wtilH^2 + 2(1+a)\wtilH\tilde u_1 + \tilde u_1^2 \right)^2 \\
		\geq &~{} a^2(2+a)^2 \int \wtilQ^2 \wtilH^4 \\
		&~{} + 4a(1+a)(2+a)\int \wtilQ^2\wtilH^3 \tilde u_1 + 2a(2+a)\int \wtilQ^2 \wtilH^2 \tilde u_1^2 \\
		\geq&~{} a^2(2+a)^2 \int \wtilQ^2 \wtilH^4 - C|a|\| \tilde u_1\|_{H_0}^2.
	\end{aligned}
	\]
	Replacing in \eqref{por favor} and using \eqref{expansionamiento}, we obtain
	\begin{equation*}
	\ba
	& \|u_2\|_2^2 + \langle \widetilde L \tilde u_1, \tilde u_1\rangle + a^2(2+a)^2  \lesssim \delta_0^2 + b_+^2 + b_-^2 + |a_1|^3 + \epsilon a^2 + \epsilon \norm{\tilu_1}_{H_0}^2 + |a\|\tilde u_1\|_{H_0}^2.
	\ea
	\end{equation*}
	We apply Lemma \ref{lema_final} now, where we get for $\epsilon$ and $\delta_0$ small,
	\begin{equation}\label{eq:energy estimate}
	\begin{aligned}
	&	a^2 + \|u_2\|_{L^2}^2 + \| \tilde u_1\|_{H_0}^2 \lesssim |b_+|^2 + |b_-|^2 + \delta_0^2 +  O\left( |b_+|^3 , |b_-|^3 , |a|^3 , \|\tilu_1\|_{H_0}^3 \right).
	\end{aligned}
	\end{equation}
	
	{\bf Step 2.} Let $\bfvareps = (\varepsilon_1, \varepsilon_2)\in \calA_0$ (see \eqref{eq:A0}). Then the condition $\langle \bfvareps, \bfZ_+ \rangle = 0$ rewrites
	\[
	\langle \varepsilon_1, \phi_0 \rangle + \langle \varepsilon_2, \mu_0^{-1}\phi_0 \rangle = 0.
	\]
	Notice that the LHS above is perfectly well-defined thanks to the decay properties of $\phi_0$, see \eqref{eq:properties-eigenvalL}. Define $b_-(0)$ and $(u_1(0), u_2(0))$ such that
	\be\label{def_b0m}
	b_-(0) = \langle \varepsilon_1, \phi_0\rangle = -\langle \varepsilon_2, \mu_0^{-1} \phi_0\rangle,
	\ee
	and
	\[
	\varepsilon_1 = b_-(0)\phi_0 + a(0) \widetilde H + \tilde u_1(0), \qquad \varepsilon_2 = -b_-(0)\mu_0\phi_0 + u_2(0).
	\]
	Then, it holds
	\[
	\langle \tilde u_1(0), \phi_0\rangle =\langle \tilde u_1(0), \widetilde Q^2 \widetilde H^3\rangle = \langle u_2(0), \phi_0\rangle = 0.
	\]
	Recall that ${\bf \widetilde H}=(\widetilde H,0)$. From \eqref{eq:A0} and \eqref{eq:calM}, we observe that the initial data in the statement of Theorem \ref{th:Manifold} decomposes as
	\[
	(\phi, \pt\phi)(0) = (1+a(0)){\bf \widetilde H} + (\tilde u_1, u_2)(0)  + b_-(0)\bfY_- + h(\bfvareps)\bfY_+.
	\]
	Now, we prove that there exist a function 
	\[
	h(\bfvareps) := b_+(0)
	\]
	such that the corresponding solution $(\phi, \pt\phi)$ is global and satisfies \eqref{eq:cont initial data}. Explicitly, we show that at least considering $h(\bfvareps)=b_+(0)$, the statement is satisfied.
	
	Let us consider $\delta_0 > 0$ small and $K>1$ large to be chosen later.  Recall that
	\[
	\|u_1\|_{H_0}^2=  \| \partial_x u_1\|_{L^2}^2 + \| \wtilQ u_1 \|_{L^2}^2. 
	\]
	In line with the approach outlined in \cite{KMM19}, we introduce the following bootstrap estimates
	\begin{gather}\label{eq:bootstrap 1}
		| a | \leq K^2 \delta_0, \quad  \| \tilde u_1\|_{H_0} \leq K^2 \delta_0 \quad \text{ and } \quad \|u_2\|_{L^2}\leq K^2 \delta_0, \\\label{eq:bootstrap 2}
		|b_-|\leq K\delta_0, \\\label{eq:bootstrap 3}
		|b_+|\leq K^5 \delta_0^2.
	\end{gather}
	Given any $(\tilde u_1(0), u_2(0))$, $b_+(0)$, $b_-(0)$ and $a(0)$ such that
	\begin{equation}\label{eq:u10 u20 b-0}
	\begin{gathered}
		\quad |a(0)|\leq \delta_0,\quad \| \tilde  u_1(0)\|_{H_0}\leq \delta_0, \quad \|u_2(0)\|_{L^2}\leq \delta_0, \quad |b_-(0)|\leq \delta_0, \\
		\| \wtilQ (u_1(0) + 2\wtilH)u_1(0)\|_{L^2} \leq \delta_0
	\end{gathered}
	\end{equation}
	and $b_+(0)$ satisfying
	\begin{equation}\label{eq:b+0}
		|b_+(0)|\leq K^5\delta_0,
	\end{equation}
	let
	\[
	T = \sup\left\{t\geq 0 \text{ such that } \eqref{eq:bootstrap 1}, \eqref{eq:bootstrap 2}, \eqref{eq:bootstrap 3} \text{ hold on } [0,t] \right\}.
	\]
	Considering that $K>1$, it follows that $T$ is well defined in $[0, +\infty]$. We will prove that there exists at least a value of $b_+(0)$ as in \eqref{eq:b+0}, $b_+(0)\in \left[-K^5\delta_0^2, K^5\delta_0^2\right]$ such that $T=\infty$. We proceed by contradiction, assuming that any $b_+(0)\in  \left[-K^5\delta_0^2, K^5\delta_0^2 \right]$ leads to $T<\infty$.
	By \eqref{eq:bootstrap 1}, we have
	\begin{equation}\label{eq:bootstrap 4}
		a^2+ \|\tilde u_1\|_{H_0}^2  + \|u_2\|_{L^2}^2 \leq 3K^4\delta_0^2.
	\end{equation}
	First, we strictly improve the estimate \eqref{eq:bootstrap 4}.  From the conservation of energy and the coercivity of $\widetilde L$, estimate \eqref{eq:energy estimate} holds (notice that this estimate is independent of $K$). Furthermore, from \eqref{eq:bootstrap 2}-\eqref{eq:bootstrap 3}, it holds
	\[
	a^2 + \| \tilde u_1\|_{H_0}^2  + \|u_2\|_{L^2}^2 \leq C_6 \left( K^2\delta_0^2 + K^{10} \delta_0^4 + \delta_0^2 \right),
	\]
	for some constant $C_6>0$. Thus, using first the largeness of $K$, and after fixing $K$, the smallness of $\delta_0$, it holds
	\begin{equation}\label{eq:constrains 1}
		C_6 \leq \frac14 K^2, \quad K^4\delta_0 \leq 1,
	\end{equation}
	and we obtain $a^2 + \| \tilde u_1\|_{H_0}^2 + \|u_2\|_{L^2}^2 \leq \frac34 K^4\delta_0^2$, which strictly improves the inequality \eqref{eq:bootstrap 4}.
	
	\medskip
	
	Second, we use \eqref{eq:db+ db-} to control $b_-$. By \eqref{eq:bootstrap 1}-\eqref{eq:bootstrap 2}-\eqref{eq:bootstrap 3}, we have
	\[
	\left|\frac{d}{dt} (e^{2\mu_0 t}b_-^2)\right| \leq C_7(K^{15} \delta_0^6 + K^6\delta_0^3)e^{2\mu_0 t},
	\]
	for some constant $C_7>0$. Therefore, by integration on $[0,t]$ and using \eqref{eq:u10 u20 b-0}, we obtain
	\[
	b_-^2 \leq \frac{C_7}{2\mu_0}(K^{15} \delta_0^6 + K^6\delta_0^3) + \delta_0^2.
	\]
	Under the constraints
	\begin{equation}\label{eq:constrains 2}
		\frac{C_7}{2\mu_0}K^{15}\delta_0^4 \leq \frac14 K^2, \quad \frac{C_7}{2\mu_0}K^{6}\delta_0 \leq \frac14 K^2, \quad 1 \leq \frac14 K^4,
	\end{equation}
	it holds $b_-^2 \leq \frac34 K^2\delta_0^2$ which strictly improves \eqref{eq:bootstrap 2}.
	
	By the improved estimates (under the constraints \eqref{eq:constrains 1}-\eqref{eq:constrains 2}) and a continuity argument, we observe that if $T<+\infty$, then $b_+(T) = K^5 \delta_0^2$.
	
	Next, we analyze the growth of $b_+$. If $t\in [0, T]$ is such that $|b_+(t)| = K^5\delta_0^2$, then it follows from \eqref{eq:b+ b-} that
	\[
	\begin{aligned}
		\frac{d}{dt}(b_+^2) \geq&~{} 2\mu_0 b_+^2 - 2C_4 |b_+|\left(b_+^2 + b_-^2 + \int \wtilQ^3 w_1^2 \right) \\
		\geq&~{} 2\mu_0 b_+^2 - 2C_4|b_+|(b_+^2 + K^2\delta_0^2 + K^4\delta_0^2) \\
		\geq&~{} 2\mu_0 K^{10}\delta_0^4 - C_8K^5\delta_0(K^{10}\delta_0^4 + K^4\delta_0^2),
	\end{aligned}
	\]
	for some constant $C_8>0$. Under the constraints
	\begin{equation}\label{eq:constrains 3}
		C_8K^{15}\delta_0^2 \leq \frac12 \mu_0 K^{10}, \quad C_8K^9 \leq \frac12 \mu_0 K^{10},
	\end{equation}
	the following inequality holds
	\[
	\frac{d}{dt}(b_+^2) \geq \mu_0 K^{10}\delta_0^4 > 0.
	\]
	By standard arguments, the above condition implies that $T$ is the first time for which $|b_+(t)|=K^5\delta_0^2$. Furthermore, $T$ depends continuously on the variable $b_+(0)$. Now, the image of the continuous map defined by
	\[
	b_+(0)\in \left[ -K^5\delta_0^2, K^5\delta_0^2 \right]\longmapsto b_+(T)\in \left\{-K^5\delta_0^2, K^5\delta_0^2 \right\},
	\]
	is exactly $\left\{-K^5\delta_0^2, K^5\delta_0^2 \right\}$, which is a contradiction.
	
	As a consequence, provided the constraints in \eqref{eq:constrains 1}, \eqref{eq:constrains 2}, \eqref{eq:constrains 3} are fulfilled, there exists at least one value of $b_+(0)\in (-K^5\delta_0^2, K^5\delta_0^2)$ such that $T=\infty$. Finally, to satisfy the conditions \eqref{eq:constrains 1}, \eqref{eq:constrains 2}, \eqref{eq:constrains 3} we fix a large enough $K>0$, depending only on the constants $C_6, C_7$ and $C_8$, and then choose $\delta_0>0$ small enough.
	
	\subsection{Uniqueness and Lipschitz regularity}
	The following proposition implies both the uniqueness of the choice of $h(\bfvareps)=b_+(0)$, for a given $\bfvareps\in\calA_0$, and the Lipschitz regularity of the graph $\calM$ defined from the resulting map $\bfvareps\in \calA_0\longmapsto h(\bfvareps)$ (see \eqref{eq:calM}). This is sufficient to complete the proof of Theorem \ref{th:Manifold}.
	
	\begin{proposition}
		There exist $C, \delta>0$ such if $(\phi, \pt\phi)$ and $(\tilde\phi, \pt\tilde\phi)$ are two solutions of \eqref{eq:varEL} satisfying for all $t\geq0$,
		\begin{equation}\label{eq:smallness}
			\|(\phi, \pt\phi)(t) - (\wtilH, 0)\|_{H_0 \times L^2}<\delta, \quad \|(\tilde\phi, \pt\tilde\phi)(t) - (\wtilH, 0)\|_{H_0 \times L^2}<\delta.
		\end{equation}
		Then, decomposing
		\begin{equation*}
			(\phi, \pt\phi) = (\wtilH, 0) + \bfvareps + b_+(0)\bfY_+, \quad (\tilde\phi, \pt\tilde\phi) = (\wtilH, 0) + \tilde\bfvareps + \tilde b_+(0)\bfY_+
		\end{equation*}
		with $\langle \bfvareps, \bfZ_+\rangle = \langle \tilde\bfvareps, \bfZ_+\rangle = 0$, it holds
		\begin{equation}\label{eq:lipschitz}
			|b_+(0) - \tilde b_+(0)|\leq C\delta^{\frac12}\|\bfvareps - \tilde\bfvareps\|_{H_0\times L^2}.
		\end{equation}
	\end{proposition}
	\begin{proof}
		We decompose the two solutions $(\phi, \pt\phi)$ and $(\tilde\phi, \pt\tilde\phi)$ satisfying \eqref{eq:smallness} as in Subsection \ref{subsec:decomposition}. In particular, from \eqref{eq:bound}, there exists $C_0>0$ such that for all $t>0$,
		\begin{equation}\label{eq:bound tilde}
			\begin{aligned}
				& \|\px u_1(t)\|_{L^2} + \|\px \tilde u_1(t)\|_{L^2}  + \| \wtilQ u_1\|_{L^2} + \|\wtilQ  \tilde u_1\|_{L^2} \\
				& \quad + \|u_2(t)\|_{L^2} + \|\tilde u_2(t)\|_{L^2} + |b_{\pm}(t)| + |\tilde b_{\pm}(t)| \leq C_0 \delta.
			\end{aligned}
		\end{equation}
		We denote
		\[
		\begin{gathered}
			\checka_1 = a_1 - \tila_1, \quad \checka_2 = a_2 - \tila_2, \quad \checkb_+ = b_+ - \tilb_+, \quad \checkb_- = b_- - \tilb_-, \\
			\checku_1 = u_1 - \tilu_1, \quad \checku_2 = u_2 - \tilu_2, \quad \checkN = N - \tilN,\\
			\checkN^{\perp} = N^\perp - \tilN^\perp, \quad \checkN_0^\perp = N_0 - \tilN_0.
		\end{gathered}
		\]
		Then, from \eqref{eq:a1a2} and \eqref{eq:system-perturbated}, the equations of $(\checku_1, \checku_2, \checkb_+, \checkb_-)$ write
		\begin{equation}\label{eq:checkb}
			\begin{cases}
				\dot{\checkb}_+(t) = \mu_0 \checkb_+(t) - \dfrac{\checkN_0}{2\mu_0}\\[0.3cm]
				\dot{\checkb}_-(t) = -\mu_0 \checkb_-(t) + \dfrac{\checkN_0}{2\mu_0},
			\end{cases}
			\qquad \text{ and } \qquad
			\begin{cases}
				\dot{\checku}_1 = \checku_2 \\[0.1cm]
				\dot \checku_2 = - L\checku_1 - \checkN^{\perp}.
			\end{cases}
		\end{equation}
		We claim that
		\begin{equation}\label{eq:est checkN}
			|\checkN_0| + \|\checkN^\perp\|_{L^2} \leq C\delta (|\checkb_+| + |\checkb_-| + \|\wtilQ u_1\|_{L^2}).
		\end{equation}
		Indeed, recalling the definition of $N$ \eqref{eq:N}, we obtain
		\[
		|\checkN| \lesssim \wtilQ^2(|\checka_1|\phi_0 + |\checku_1|)(|a_1|\phi_0 + |\tila_1|\phi_0 + |u_1| + |\tilu_1|).
		\]
		Using the H\"older inequality and again \eqref{eq:bound tilde}, we find $\|\checkN\|_{L^2}\leq \delta(|\checka_1| + \|\wtilQ u_1\|_{L^2})$ and estimate \eqref{eq:est checkN} follows.
		
		Let define
		\[
		\beta_+ = \checkb_+^2, \quad \beta_- = \checkb_-^2, \quad \beta_c = \langle L\checku_1, \checku_1\rangle + \langle \checku_2, \checku_2\rangle.
		\]
		Computing the variation of these terms using \eqref{eq:checkb}, we get
		\[
		\ba
		& \dot \beta_c = -2\langle \checkN^\perp, \checku_2\rangle, \quad \dot\beta_+ - 2\mu_0\beta_+ = -\frac{1}{\mu_0}\checkb_+\checkN_0, \\
		&  \dot\beta_- + 2\mu_0\beta_- = \frac1\mu_0 \checkb_-\checkN_0.
		\ea
		\]
		By \eqref{eq:est checkN} and the coercivity property \eqref{eq:coer_final}, we have
		\begin{equation}\label{eq:bound betas}
			|\dot\beta_c| + | \dot\beta_+ - 2\mu_0\beta_+| + |\dot\beta_- + 2\mu_0\beta_-| \leq K\delta(\beta_c + \beta_+ + \beta_-),
		\end{equation}
		for some $K>0$. In order to obtain a contradiction, assume that the following holds
		\begin{equation}\label{eq:contradiction}
			0 \leq K\delta (\beta_c(0) + \beta_+(0) + \beta_-(0)) < \frac{\mu_0}{10}\beta_+(0).
		\end{equation}
		We consider the following bootstrap estimate
		\begin{equation}\label{eq:bootstrap 5}
			K\delta(\beta_c + \beta_+ + \beta_-) \leq \mu_0\beta_+.
		\end{equation}
		Define
		\[
		T = \sup\{t>0 \text{ such that \eqref{eq:bootstrap 5} holds} \} > 0.
		\]
		We work on the interval $[0, T]$. Note that from \eqref{eq:bound betas} and \eqref{eq:bootstrap 5}, it holds
		\begin{equation}\label{eq:exp growth}
			\mu_0 \beta_+ \leq 2\mu_0\beta_+ - K\delta(\beta_c + \beta_+ + \beta_-) \leq \dot\beta_+.
		\end{equation}
		Then, $\beta_+$ is positive and increasing on $[0, T]$.
		
		Next, by \eqref{eq:bound betas} and \eqref{eq:bootstrap 5},
		\[
		\dot\beta_c \leq \mu_0\beta_+ \leq \dot\beta_+,
		\]
		and thus, integrating and using that $\beta_+(0)>0$, we obtain
		\[
		\beta_c(t) \leq \beta_c(0) + \beta_+(t) - \beta_+(0)\leq \beta_c(0) + \beta_+(t).
		\]
		Furthermore, by \eqref{eq:contradiction} and the growth of $b_+$, for $\delta$ small enough, we get
		\[
		\ba
		K\delta\beta_c(t) \leq &~{} K\delta(\beta_c(0) + \beta_+(t)) \leq \frac{\mu_0}{10}\beta_+(0) + K\delta\beta_+(t)\leq \frac{\mu_0}{5}\beta_+(t).
		\ea
		\]
		For $\beta_-$, by \eqref{eq:bound betas} and \eqref{eq:bootstrap 5},
		\[
		\dot\beta_- \leq -2\mu_0 \beta_- + \mu_0\beta_+,
		\]
		by integration and the growth of $b_+$, we have
		\[
		\beta_-(t) \leq e^{-2\mu_0 t}\beta_-(0) + \mu_0\beta_+(t)e^{-2\mu_0 t}\int_0^t e^{2\mu_0 s}ds \leq \beta_-(0) + \frac12\beta_+(t).
		\]
		Therefore, using \eqref{eq:contradiction}, for $\delta$ small enough, we get
		\[
		K\delta\beta_-(t) \leq K\delta(\beta_-(0) + \beta_+(t)) \leq \frac{\mu_0}{10}\beta_+(0) + K\delta\beta_+(t)\leq \frac{\mu_0}{5}\beta_+(t).
		\]
		Finally, it is clear that for $\delta$ small, it holds $K\delta\beta_+\leq \frac{\mu_0}{5}\beta_+$.
		
		Considering the previous estimates, we have proved that, for all $t\in[0,T]$,
		\[
		K\delta(\beta_c(t) + \beta_+(t) + \beta_-(t)) \leq \frac35 \mu_0 \beta_+(t).
		\]
		By a continuity argument, this means that $T=+\infty$. However, by the exponential growth of $b_+$ given by \eqref{eq:exp growth}, and since $\beta_+(0)>0$, we obtain a contradiction with the global bound \eqref{eq:bound tilde} on $|b_+|$.
		
		Since estimate \eqref{eq:contradiction} is contradicted, and since it holds
		\[
		\bfvareps = \bfu(0) + b_-(0)\bfY_-, \quad \tilde\bfvareps = \tilde\bfu(0) + \tilb_-(0)\bfY_-,
		\]
		with $\langle \bfu(0), \bfY_-\rangle = \langle \tilde\bfu(0), \bfY_-\rangle = 0$, we have proved \eqref{eq:lipschitz}.
	\end{proof}
	
	\section{Spectral Theory for $L$}
	\label{A:LINEAL-SPECTRAL-THEORY}
	
	In this section, we describe the spectral properties of the operator $L$ introduced in equation \eqref{eq:L}. Being a variable coefficients operator with no explicit eigenfunctions, the understanding here becomes more subtle, and some interesting new features appear in the spectral analysis.
	
	\medskip
	
	Notice that $L$ correspond to a Schr\"{o}dinger operator with potential $V(x) = 2\wtilQ^2(x)(1-\wtilQ(x))$, where we have defined the function
	\[
	\wtilQ(x) = Q(\alpha^{-1}(x)) \quad \text{with }\  \alpha(x) = \frac13(\sinh x + x). 
	\]
	Unlike standard operators \cite{PosTel33}, $L$ has a complicated structure with slow decay, essentially just enough to run suitable estimates.
	
	\begin{remark}
		A direct analysis shows that the null space of $P_0 = -\px^2$ is spanned by functions of the type ${1, x}$ as $x\to \infty$.
		Note that this set is linearly independent and there are no $L^2(\R)$ integrable functions in the semi-infinite line $[0, +\infty)$. Therefore, the analysis of $V$ becomes essential to understand the set of possible solutions in $L^2(\R)$ for the operator $L$.
	\end{remark}
	
	\begin{lemma}\label{lemma:Lproperties}
		The linear operator $L$ defined by
		\begin{equation}\label{def_L_V}
			L \phi := -\partial_x^2\phi + V(x)\phi, \quad \textup{with} \quad V(x) = 2\wtilQ^2(x)(1 - \wtilQ(x)),
		\end{equation}
		with dense domain $\mathcal{D}(L) = H^2(\R)$, satisfies the following properties.
		\begin{enumerate}
			\item \label{l:L1} The essential spectrum of $L$ is $[0, +\infty)$.
			\item \label{l:L2} $\sigma_{disc}(L)\cap \mathbb{R}_-$ is not empty.
			\item \label{l:L3} The operator $L$ has a first simple eigenvalue $\lambda_0$, with associated eigenfunction $\phi_0$ that satisfies
			\begin{equation}\label{eq:eigen}
				L\phi_0 = \lambda_0 \phi_0, \quad \phi_0 \in H^2(\mathbb R).
			\end{equation}
		\end{enumerate}
	\end{lemma}
	\begin{proof}
		Proof of \eqref{l:L1}.  Clearly $L$ is self-adjoint on $H^2(\R)$, so the whole spectrum of $L$ is contained on the real axis. Even more, since $\alpha(x)$ is strictly monotone, positive and $\alpha^{-1}(x) \to \pm \infty$ as $x\to \pm \infty$, we can see from Lemma \ref{lemma:estimation} that the associated potential $V(x)$ goes to 0 when $x  \to \pm \infty$. This imply by standard arguments (see \cite{DunSch63}, Chapter XIII, section 6) that the essential spectrum of $L$ is $[0,+\infty)$.
		
		Proof of \eqref{l:L2}. First note that by choosing $\phi = \wtilQ$ we obtain
		\begin{align*}
			L\wtilQ &= -\partial_x^2 \wtilQ + 2\wtilQ^3(1-\wtilQ)= \px(\wtilQ^2 \wtilH) + 2\wtilQ^3(1-\wtilQ) \\
			&= -2\wtilQ^3 \wtilH^2 + \frac13\wtilQ^4 + 2\wtilQ^3(1-\wtilQ)= -\frac13 \wtilQ^4,
		\end{align*}
		and then
		\[
		\langle L \wtilQ, \wtilQ \rangle = -\frac13 \int \wtilQ^5(x)dx = -\frac13 \int Q^4(y)dy < 0.
		\]
		This conclude that $\sigma_{disc}(L)\cap \R_- \neq \varnothing$.
		
		\medskip
		
		Proof of \eqref{l:L3}. First, since $L$ is bounded from below we consider the operator $L_c = L + c$ for a large enough constant $c>0$ such that the associated potential is strictly positive. Since for any $f\in \mathcal{C}_0^1(\R)$ the problem
		\begin{equation*}
			\begin{cases}
				-L_c v(y) = f(y), \quad y\in\R \\[0.1cm]
				\hspace{1.2cm} v \in H^2(\R),
			\end{cases}
		\end{equation*}
		has a unique solution satisfying
		$\|v\|_{H^{2}} \lesssim \|f\|_{H^{1}}$, it follows that $L_c^{-1}:\mathcal{C}^1(\R) \to \mathcal{C}^1(\R)$ is linear compact. From the strong maximum principle theorem if $f\geq 0$ then $v=L_c^{-1}f>0$ in $\R$. This implies that $L_c^{-1}$ is a strongly positive operator over the set of nonnegative functions. Now it follows from the Krein-Rutman theorem (see \cite{De85} \cite{KR48}) that the radius of the operator $r(L_c^{-1})$ is a positive simple eigenvalue, and the associated eigenfunction $f$ is nonnegative. Thus $\phi_0=L_c^{-1}f$ satisfies
		\[ -L \phi_0(x) = \lambda_0 \phi_0(x), \quad x\in\R\]
		with $\phi_0>0$ in $\R$, and $\lambda_0 = r(L_c^{-1}) - c$ a simple eigenvalue.
	\end{proof}
	Eigenvalues embedded in the continuous spectrum of $L$ depend directly on the decay and oscillation of the potential $V$. As emphasized in \cite[Chapter XIII, Section 13]{ReedSimon78}, the existence of embedded eigenvalues in the continuous spectrum of $L$ depends on detailed assumptions over the decay, symmetry and oscillation of the potential $V$.
	\begin{lemma}\label{lema_posi}
		The operator $L$ has no strictly positive eigenvalues. 
	\end{lemma}
	\begin{proof}
		By Lemma \ref{lemma:estimation} we have a polynomial decrease of $V\sim |x|^{-2}$, and even more
		\[
		\begin{aligned}
			\int_0^{\infty} |V(x)|dx = &~{}  2\int_0^{\infty} \wtilQ^2(x)|1 - \wtilQ(x)|dx \\
			=&~{} 2\int_0^{\infty} Q(s)|1 - Q(s)|ds \\
			\leq&~{} \int_0^{\infty} Q(s)ds < +\infty.
		\end{aligned}
		\]
		This, and the fact that $V$ is a symmetric function on $\mathbb R$, allows us apply a particular case of the Kato-Argmon-Simon Theorem (see \cite[Theorem XIII.56]{ReedSimon78}), where we conclude that $L$ has no strictly positive eigenvalues.
	\end{proof}

	\begin{lemma}\label{valor mu0} One has the following bounds for the first negative eigenvalue $\lambda_0=-\mu_0^2$ in terms of $\mu_0$:
		\[
		0.808 \leq \mu_0 \leq 0.883.
		\]
	\end{lemma}
	
	\begin{proof}
		Recall that 
		\[
		\lambda_0 = \inf_{\|f\|_{L^2} =1} \left( L f,f\right).
		\]
		We introduce now the following test function: 
		\begin{equation}\label{test_function}
			f(x):=  c_0 e^{-\frac12x^2} \left( a_4 x^4+ a_2 x^2+ a_0 \right),
		\end{equation}
		with
		\[
		a_4:= -0.0574167, \quad a_2:=0.115416 , \quad a_0:=- 0.761391.
		\]
		Here, $c_0$ is an explicit normalizing constant, obtained from
		\[
		\ba
		1= &~{} \int f^2 \\
		= &~{}c_0^2 \int  e^{-x^2}\left( a_4^2 x^8 +2a_4a_2 x^6 +( a_2^2+2a_4a_0) x^4 +2a_2a_0 x^2+ a_0^2 \right).
		\ea
		\]
		and the fact that from Wolfram Mathematica,
		\[
		\int e^{-x^2} =\sqrt{\pi }, \quad\int x^2 e^{-x^2} =\frac{\sqrt{\pi }}{2}, \quad\int x^4e^{-x^2} =\frac{3 \sqrt{\pi }}{4},
		\]
		and
		\[
		\int x^6e^{-x^2} =\frac{15 \sqrt{\pi }}{8}, \quad\int x^8 e^{-x^2} =\frac{105 \sqrt{\pi }}{16}. 
		\]
		One can easily see from the previous exact integrals that 
		\[
		c_0 \sim 1.0000005590505727.
		\]
		 Then, since $\alpha(y)=\frac13 (y+ \sinh y)$ is bijection,
		\[
		\begin{aligned}
			\left( L f,f\right) = &~{}  \int f'^2 +2 \int f^2 \widetilde Q^2(1-\widetilde Q) \\
			= &~{} \int f'^2(x)dx +2 \int f^2(\alpha(y)) Q(1-Q)(y)dy
			\sim -0.652,
		\end{aligned} 
		\]
		and therefore $\mu_0^2 \geq 0.652$ and  $0.808\leq \mu_0$. In the other sense, if 
		\[
		Q_p= \left(\frac{p+1}{2\cosh^2\left( \frac{p-1}2 x\right)} \right)^{1/(p-1)}, \quad p=9/2,
		\]
		we have $L\geq L_p:=  -\partial_x^2 -0.845Q_p^{7/2}$. This is a consequence of the fact that 
		\be\label{checkeo8}
		2\wtilQ^2(x)(1 - \wtilQ(x)) \geq - 0.845Q_p^{7/2}(x) = -\frac{2.32375}{\cosh^2\left( \frac74 x\right)} .
		\ee
		By parity, this is an inequality that need to be checked only in the region in $[0,\infty)$ where $1 - \wtilQ(x) \leq 0$, which is the small compact region $[0, x_0]$, with {\color{blue} $x_0\sim 1.01634$}. This is easily checked to high accuracy by graphing both functions, see Fig. \ref{fig:pordebajo}. Indeed, we have that \eqref{checkeo8} is equivalent to check the inequality made of explicit functions
		\be\label{pordebajo}
		2Q^2(y)(1 - Q(y)) \geq  -\frac{2.32375}{\cosh^2\left( \frac74 \alpha(y)\right)}.
		\ee
		Notice that $L_p$ is a classical operator with explicit first eigenfunction $Q_p^m$, $m=\frac1{40} (-35 + \sqrt{4943}) \sim 0.88$ and first eigenvalue $-m^2\sim -0.7791$, from which $ \mu_0 \leq 0.883$.
	\end{proof}
	
	\begin{figure}[htbp]
		\centering
		\includegraphics[width=0.82\linewidth]{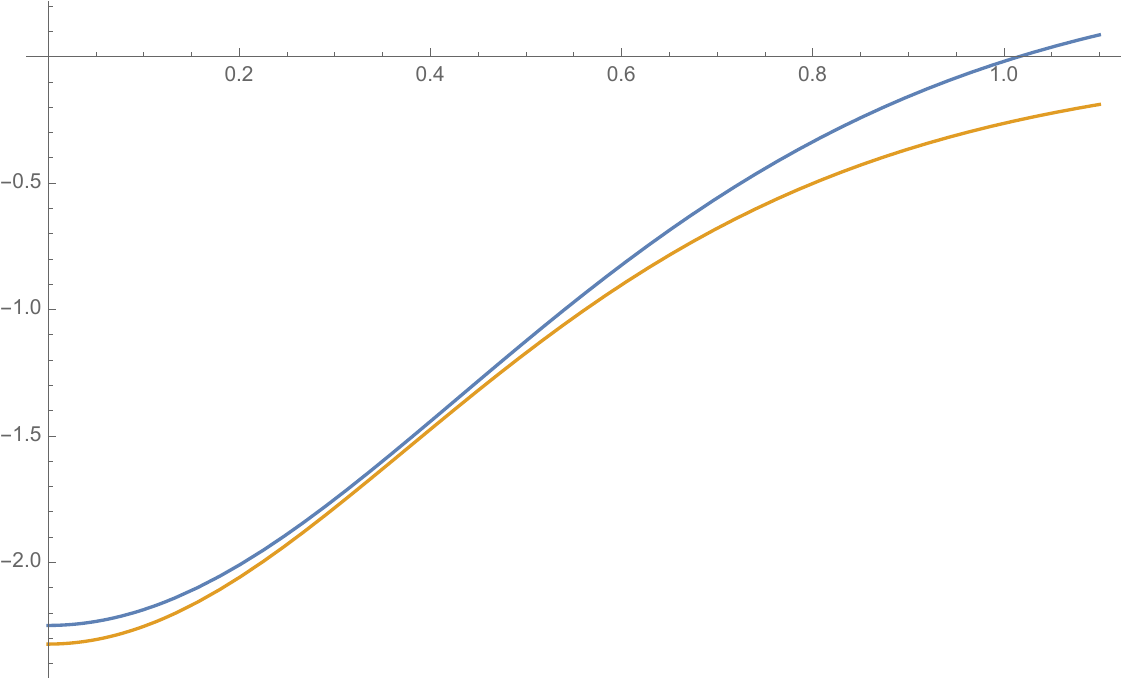}
		\includegraphics[width=0.82\linewidth]{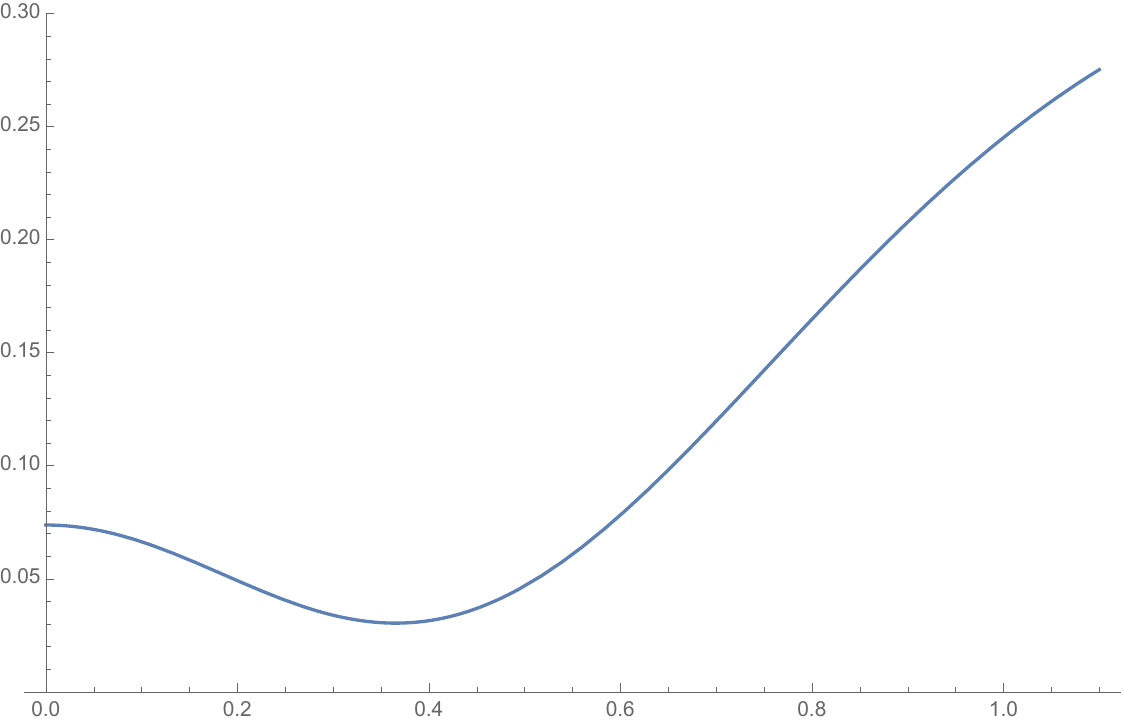}
		\caption{Above: Comparison between the potentials $2\wtilQ^2(x)(1 - \wtilQ(x))$ (blue line) and $- 0.845Q_{9/2}^{7/2}(x)$ (yellow line) in the region $[0,1.1]$. Below: Plot of the difference $2\wtilQ^2(x)(1 - \wtilQ(x))+ 0.845Q_{9/2}^{7/2}(x)$ in the considered region.}
		\label{fig:pordebajo}
	\end{figure}
	
	\begin{lemma}\label{lem_decay_expo}
		For the operator $L$, the associated eigenfunction $\phi_0$ of the first simple eigenvalue $-\mu_0^2$ satisfies, along with its derivatives, an exponential decay given by
		\begin{equation}\label{eq:exp-decay}
			|\phi_0(x)|, ~ |\px \phi_0(x)|, ~ |\px^2 \phi_0(x)| \lesssim e^{-\frac{\sqrt2}{2}\mu_0 |x|}
		\end{equation}
	\end{lemma}
	\begin{proof}
		This result follows from a standard argument of ODE (see e.g. \cite{BerLio83}) adapted for the particular variable coefficient problem analyzed in this article. For the sake of completeness, we show it here.
		
		By Lemma \ref{lemma:Lproperties} $\phi_0$ is a normalized even solution of class $H^1(\R)$ associated with the principal eigenvalue $\lambda_0 = -\mu_0^2$ satisfying the equation
		\[ \px^2 \phi_0 = q(x) \phi_0 \]
		where $q(x) = \mu_0^2 + V(x)$. In the following we restrict our analysis in the semi-infinite line $[0, +\infty)$ due to the parity of $\phi_0$. Since $V\geq 0$ for $x \geq x_r$, with $x_r = \alpha(2\arcosh(\sqrt{3/2}))$, one has the bound by below
		\[ 
		q(x) \geq \mu_0^2,
		\]
		for any $x \geq x_r$.
		
		We define $v = \phi_0^2\geq 0$, which verifies
		\[
		\frac12 \px^2 v(x) = (\px \phi_0)^2(x) + q(x)\phi_0^2(x) \geq \mu_0^2 v(x),
		\]
		for any $x\geq x_r$.
		
		Now let define the auxiliary function $z = e^{-\sqrt{2}\mu_0 x}(\px v + \sqrt2\mu_0 v)$ to compare the decreasing rate of $\phi_0$ with respect to an exponential.  We have
		\[ \px z = e^{-\sqrt2\mu_0 x} (\px^2 v - 2\mu_0^2 v) \geq 0,\]
		hence $z$ is non-decreasing on $[x_r, +\infty)$.
		
		Next, we prove that $z\leq 0$ for $x\geq x_r$ by contradiction: If there exists a $x_0>x_r$ such that $z(x_0)>0$, then
		\[
		z(x) \geq z(x_0) > 0 , 
		\]
		for all $x\geq x_0$. This implies that
		\[
		\px v + \sqrt2\mu_0 v \geq z(x_0) e^{\sqrt{2}\mu_0 x},
		\]
		then $\px v + \sqrt2 \mu_0 v$ is not integrable on $(x_0, + \infty)$. But $\phi_0 \px \phi_0$ and $\phi_0^2$ are integrable on $(x_0, + \infty)$, so that $\px v$ and $v$ are integrable. This is a contradiction, hence we conclude that $z(x)\leq 0$ for $x>x_r$.
		
		In particular, we have the inequality
		\[ 
		\px(e^{\sqrt2 \mu_0 x}v) = e^{2\sqrt2\mu_0 x} z \leq 0 \quad \text{for } x\geq x_r, 
		\]
		This implies that $v(x)\lesssim e^{-\sqrt2\mu_0 x}$. Replacing the definition of $v$, we obtain the decay estimate for the first eigenfunction given by
		\[
		|\phi_0(x)| \lesssim e^{-\frac{\sqrt2}{2}\mu_0x }.
		\]
		To obtain the exponential decay of $\px \phi_0$, we use the trivial bound
		\[ 
		\mu_0^2 \leq q(x) \leq \mu_0^2 + 1,
		\]
		for all $x> x_r$. Hence, integrating over $(x_1,x_2)$
		\[ \mu_0^2 \int_{x_1}^{x_2} \phi_0 \leq \px \phi_0(x_2) - \px \phi_0(x_1) \leq (\mu_0^2 + 1)\int_{x_1}^{x_2} \phi_0, \]
		and from the exponential decay of $\phi_0$, letting $x_1, x_2 \to +\infty$ proves that $\px \phi_0$ has a limit at infinity. From the exponential decay of $\phi_0$, this limit must be zero. Therefore
		\[ |\px\phi_0(x)| \leq (\mu_0^2 + 1) \int_x^{\infty} |\phi_0| \lesssim  e^{-\frac{\sqrt2}{2}\mu_0 x}. \]
		Finally, the exponential decay for $\px^2\phi_0$ follows directly from the decay of $\phi_0$.
	\end{proof}
	
	\begin{corollary}\label{cor:dphi}
		If $\phi_0:\R\to \R$ is a positive function, then $\phi_0'(x)$ is non-positive for all $x\geq 0$, and has a unique root at 0. 
	\end{corollary}
	\begin{proof}
		First, we denote as $x_0>0$ the point where $V(x_0) = - \mu_0^2$.
		
		If $0 < x < x_0$, then integrating equation \eqref{eq:properties-eigenvalL} between $0$ and $x$, and by Corollary \ref{cor:parity} we have
		\begin{align*}
			\phi_0'(x) &= \int_0^x (\mu_0^2 + V(y))\phi_0(y)dy < 0.
		\end{align*}
		If $x>x_0$, we integrate \eqref{eq:properties-eigenvalL} and by the decay estimate over $\phi_0'$ we obtain that
		\[ 
		\phi_0'(x) = -\int_x^\infty (\mu_0^2 + V(y)) \phi_0(y)dy < 0, 
		\]
		since $\phi_0$ and $\mu_0^2 + V(y)$ are positive for $y\geq x_0$.
	\end{proof}

	\medskip

	\section{Positivity and repulsivity of the potential}
	\label{B:POSITIVITY-POTENTIAL}
	Now, we focus on proving some results related to the transformed problem for the Schr\"odinger equation for $L_0$, see subsection \ref{subsec:virial II} for details. In particular, the objective of this section is to prove the repulsivity of the potential $V_0$ (in the sense that $xV_0' \leq 0$ for any $x$), and its strict repulsivity in a particular subregion of space. Recall that this is one of the most relevant facts needed to apply a virial argument to describe the stability of the kink \cite[Theorem XIII.60]{ReedSimon78}.  This result becomes subtle due to the lack of an explicit form for the eigenvalue, in contrast to other recent works. See also the cubic-quintic NLS case by Martel \cite{Martel1,Martel2} and the works \cite{Mau23,MM23} for problems in some sense similar to ours. Hence, we must establish some results with an auxiliary function that determines the transformed problem.

	\subsection{Key properties and positivity} 
	We start out with a fundamental lemma. For this, let $\phi_0$ be the positive, even and exponentially decaying eigenfunction satisfying \eqref{eq:eigen}, and  define $h_0:\R_{+} \to \R$ as 
	\begin{equation}\label{def:h0}
		h_0(x) = \frac{\phi_0'(x)}{\phi_0(x)}.
	\end{equation}
	Finally, recall $L$ and $V$ from \eqref{def_L_V}.
	
	\begin{lemma}\label{lem:h0 properties}
		Let $h_0$ be as in \eqref{def:h0}. Then one has the following:
		\begin{enumerate}
			\item \label{h0:1} The function $h_0$ is well defined over $\R_+$. It is non-positive and one can write the principal eigenfunction $\phi_0$ of the operator $L$ as follows
			\begin{equation}\label{eq:h0 phi0}
				\phi_0(x) = \phi_0(0)\exp\left(\int_0^{x} h_0(y)dy \right).
			\end{equation}		
			\item \label{h0:2} The function $h_0$ is the unique solution of the initial value problem
			\begin{equation}\label{eq:h0 ivp}
				\begin{cases}
					\ h_0'(x) + h_0^2(x) = \mu_0^2 + V(x), \quad \text{for } x\geq 0, \\
					\ \hspace{1.5cm} h_0(0) = 0.
				\end{cases}
			\end{equation}	
			\item \label{h0:3} We have the integral formulation
			\begin{equation}\label{eq:h0'}
				h_0'(x) = -\frac{1}{\phi_0^2(x)} \int_x^{\infty}V'(y)\phi_0^2(y)dy
			\end{equation}
			for all $x\geq 0$.
		\end{enumerate}
	\end{lemma}
	\begin{proof}
		Proof of \eqref{h0:1}. By \eqref{eq:eigen}, the first eigenvalue $-\mu_0^2$ associated with $L$ obey the equation
		\begin{equation}\label{auxiliar0}
			\phi_0''(x) = (\mu_0^2 + V(x))\phi_0.
		\end{equation}
		From Lemma \ref{lemma:Lproperties}, $\phi_0$ is the unique positive and even eigenfunction, and it has no roots. From Corollary \ref{cor:dphi} we have that $\phi_0'(x)$ is negative for $x>0$. This proves that $h_0$ is well defined over $\R_+$, and even more, by direct integration we have that the identity
		\begin{equation*}
			\phi_0(x) = \phi_0(0)\exp\left(\int_0^{x} h_0(y)dy \right),
		\end{equation*}
		is well defined over all $x\in [0, + \infty)$. The extension to any $x\in\mathbb R$ is direct.
		
		Proof of \eqref{h0:2}. This is a direct fact from the parity of $h_0$ and the eigenvalue equation \eqref{eq:eigen} that obeys $\phi_0$.
		
		Proof of \eqref{h0:3}. From \eqref{auxiliar0} and the decay estimates \eqref{eq:exp-decay} we have
		\begin{align*}
			(\phi_0'(x))^2 &= -\int_x^{\infty} (\mu_0^2 + V(y))(\phi_0^2)'(y)dy \\
			&= (\mu_0^2 + V(x))\phi_0^2(x) + \int_x^{\infty} V'(y) \phi_0^2(y) dy.
		\end{align*}
		Dividing by $\phi_0^2$ and by definition of $h_0$, we obtain
		\begin{equation*}
			h_0^2(x) = \mu_0^2 + V(x) + \frac{1}{\phi_0^2(x)} \int_x^\infty V'(y)\phi_0^2(y)dy.
		\end{equation*}
		Replacing in \eqref{eq:h0 ivp} we have \eqref{eq:h0'}.
		
		\medskip
		
	\end{proof}
	\begin{remark}
		The function $h_0$ is primordial to understand the Darboux transformation applied in Subsection \ref{subsec:transformedproblem}, since we can write the operators $L_0, U, U^*$ as follows
		\[
		L_0 = -\px^2 + 2(h_0^2 - \mu_0^2) - V,
		\]
		\[ U = \px - h_0, \qquad U^* = -\px - h_0. \]
	\end{remark}
	
	\begin{remark}
		Lemma \ref{lem:h0 properties} also suggests a growing dependence of the sign of $h_0'$ with respect to the potential $V'$. This fact and the convexity of $h_0$ will allow us to obtain useful bounds to control the derivative of the transformed potential $V_0'$.
	\end{remark}
	
	\begin{lemma}\label{def: roots}
		There exist only a unique positive root $x_0$ of $V(x)$, a unique positive root $x_1$ of $V'(x)$, and two positive roots $\{x_{2,1} , x_{2,1}\}$ of $V''(x)$. Moreover, $0 < x_{2,1} < x_0 < x_1 < x_{2,2}$ (see also Figure \ref{fig:better bound}).
	\end{lemma}
	
	\begin{remark}\label{rem:V}
		Explicitly, one has
		\begin{equation*}
			\begin{cases}
				\ V(x) \leq 0 \quad \text{ for } 0\leq x\leq x_0, \\
				\ V(x) \geq 0 \quad \text{ for } x\geq x_0.
			\end{cases} \qquad
			\begin{cases}
				\ V'(x) \geq 0 \quad \text{ for } 0\leq x\leq x_1, \\
				\ V'(x) \leq 0 \quad \text{ for } x\geq x_1.
			\end{cases} \qquad
		\end{equation*}
		
		\begin{equation*}
			\begin{cases}
				\ V''(x) \geq 0 \quad \text{ for } 0\leq x\leq x_{2,1}, \\
				\ V''(x) \leq 0 \quad \text{ for } x_{2,1}\leq x\leq x_{2,2}, \\
				\ V''(x) \geq 0 \quad \text{ for } x\geq x_{2,2}.\\
			\end{cases}
		\end{equation*}
	Below we provide a rigorous proof of the ordering $0 < x_{2,1} < x_0 < x_1 < x_{2,2}$ and $x_{2,2}>1$.
	\end{remark}
	
	\begin{proof}[Proof of Lemma \ref{def: roots}]
		Since $Q(x)$ is positive, even, decreasing for $x>0$, and has range $(0,\frac32)$, we easily see that for $V(x)=2\wtilQ^2(x)(1 - \wtilQ(x))$, its root $x_0>0$ is unique. From \eqref{tilde Q tilde H} and \eqref{eq:dalpha}, $V'$ satisfies 
		\begin{equation}\label{Vp}
			\begin{aligned}
				V' (x)= &~{} 4 \wtilQ (x)\wtilQ'(x) -6 \wtilQ^2(x) \wtilQ' (x)\\
				=&~{}  2\wtilQ^2(x) Q' (\alpha^{-1}(x))(2-3\wtilQ(x)).
			\end{aligned}
		\end{equation}
		By the same arguments as before, $x_1>0$ is unique. Moreover, $V'>0$ in $(0,x_1)$ and negative in $(x_1,\infty)$. Notice that $V(x_0)=2\wtilQ^2(x_0)(1 - \wtilQ(x_0))=0$, and since $x_0>0$,
		\[
		\ba
		V'(x_0)= &~{} 2\wtilQ^2(x_0) Q' (\alpha^{-1}(x_0))(2-3\wtilQ(x_0)) \\
		= &~{}-2\wtilQ^3(x_0) Q' (\alpha^{-1}(x_0))>0.
		\ea
		\]
		Therefore, by uniqueness $x_0<x_1$. Since also $V'(0)=0$, one has $0<x_{2,1}<x_1$, where $x_{2,1}>0$ is a root of $V''$. Finally,
		\[
		\begin{aligned}
		V'' (x)=&~{} 8 \wtilQ^2 (x)Q'^2(\alpha^{-1}(x))  +4 \wtilQ^3(x) Q'' (\alpha^{-1}(x)) \\
		&~{} -18 \wtilQ^3 (x) Q'^2(\alpha^{-1}(x))   -6 \wtilQ^4 (x)Q'' (\alpha^{-1}(x)).
		\end{aligned}
		\]
		Since $Q''=Q-Q^2$ and $Q'^2 =Q^2-\frac23Q^3$, we obtain
		\begin{equation}\label{Vpp}
			V'' (x)= 2\wtilQ^4(x) \left( 6 -\frac{50}3 \wtilQ(x) +9\wtilQ^2(x)\right).
		\end{equation}
		Notice that $\wtilQ \in (0,\frac32)$ in $x>0$. The equation $9m^2-\frac{50}3m +6=0$ has two positive real roots: $m_\pm =\frac1{27} (25 \pm \sqrt{139})$, $m_-\sim 0.49$ and $m_+\sim 1.36$, both below $\frac32$. Since $\alpha^{-1}$ is a bijection this implies that $V''$ has only two positive roots, $x_{2,1}$ and $x_{2,2}$. Let us check that $x_{2,1}<x_0$ and $x_{2,2}>x_1$. Indeed,
		\[
		\ba
		& V''(0)=2\left(\frac32 \right)^4 \left( 6 -\frac{50}3 \left(\frac32 \right) +9\left(\frac32 \right)^2\right)\sim 12.65,\\
		& V'' (x_0) = -\frac{10}3<0, 
		\ea
		\]
		therefore $x_{2,1}$ first root of $V''$ must satisfy $x_{2,1}<x_0$. Finally, since $\wtilQ(x_1) = \frac23$ and $V'(x_1)=0$ as unique root, we have
		\[
		V''(x_1)=2\left(\frac23 \right)^4 \left( 6 -\frac{50}3 \left(\frac23 \right) +9\left(\frac23 \right)^2\right)\sim -0.44,
		\]
		implying that $x_{2,2}>x_1$. Finally, the fact that $x_{2,2}>1$ is a consequence of the fact that $x_0>1$, proved in Lemma \ref{valor mu0}, and the ordering proved above. The proof is complete. 
	\end{proof}
	
	Recall that $h_0(x)< 0$ if $x>0$ (Lemma \ref{lem:h0 properties}).
	
	\begin{lemma}\label{lem:bounds h0}
		If we define
		\begin{equation}\label{tmu0}
			\widetilde \mu_0:= \sqrt{\mu_0^2 + \max_{y>0}V(y)},\quad  \max_{y>0} V(y) = \frac{8}{27},
		\end{equation}
		the following upper and lower bounds for $h_0$ are satisfied:
		\begin{enumerate}
			\item \label{l:b1} For all $x\geq0$,
			\begin{equation}\label{eq:h0 bound 0<x}
				-\tilde{\mu}_0 \leq h_0(x).
			\end{equation}
			\item \label{l:b2}
			For all $x\geq x_0$,
			\begin{equation}\label{eq:h0 bound x1<x}
				h_0(x) \leq -\mu_0.
			\end{equation}
			In addition, we have the limit
			\begin{equation}\label{lim h0}
				\lim_{x\to +\infty} h_0(x) = -\mu_0.
			\end{equation}
		\end{enumerate}
	\end{lemma}
	\begin{proof}
		Proof of \eqref{l:b1}. By Lemma \ref{def: roots} we know that $V'(x)$ has a unique positive root $x_1$. Then, by \eqref{eq:h0'} and Remark \ref{rem:V} we conclude that $h_0'$ is positive for large $x$ and it has at most one positive root. Now, from Lemma \ref{valor mu0}, \eqref{eq:h0 ivp}, $Q(0)=\widetilde Q(0)=\frac32$ and \eqref{def_L_V}, $h_0'$ satisfies
		\[ 
		h_0'(0) = \mu_0^2 + V(0) = \mu_0^2 - \frac94 \sim -1.59.
		\]
		Also, by Remark \ref{rem:V}, and \eqref{eq:h0'} we obtain $h_0'(x_1) > 0$. Therefore there exists a unique positive root of $h_0'$, that we denote $\bar x$, with $0<\bar x < x_1$. Moreover, $h_0' < 0$ in $(0, \bar x)$ and positive in $(\bar x, \infty)$. Due to the sign of $h_0$, $\bar x$ must correspond to the global minimum for $h_0$ in the positive line. With this result, $h_0\leq 0$ and using \eqref{eq:h0 ivp} and \eqref{tmu0},
		\[
		\begin{aligned}
			h_0^2(x) \leq  h_0^2(\bar x) = \mu_0^2 + V(\bar x) \leq  \mu_0^2 + \max_{y>0} V(y) = \tilde \mu_0^2.
		\end{aligned}
		\]
		This concludes \eqref{eq:h0 bound 0<x}.	
		
		Proof of \eqref{l:b2}. First, from Remark \ref{rem:V}, if $x \geq x_1$ then $V(x) >0$, $V'(x)\leq 0, \phi_0'(x) <0$, and by \eqref{eq:h0 ivp} and \eqref{eq:h0'} we have
		\[
		\begin{aligned}
			\mu_0^2 - h_0^2(x) = &~{}  -\frac{1}{\phi_0^2(x)} \int_x^{\infty} V'(y)\phi_0^2(y)dy - V(x)\\
			\leq &~{} - \int_x^{\infty} V'(y) dy - V(x)  = 0.
		\end{aligned}
		\]
		Since $h_0(x)\leq 0$, we conclude that $h_0(x) \leq -\mu_0$.
		
		Similarly, from Remark \ref{rem:V}, if $x_0 \leq x \leq x_1$ we have that $V(x), V'(x) \geq 0$, $\phi_0'(x)<0$. Then
		\begin{align*}
			\mu_0^2 - h_0^2(x) = &~{} -\frac{1}{\phi_0^2(x)} \int_x^{x_1} V'(y)\phi_0^2(y)dy + \frac{1}{\phi_0^2(x)} \int_{x_1}^{\infty} |V'(y)|\phi_0^2(y)dy - V(x) \\
			\leq &~{}   -\frac{\phi_0^2(x_1)}{\phi_0^2(x)} \int_x^{x_1} V'(y)dy  - \frac{\phi_0^2(x_1)}{\phi_0^2(x)} \int_{x_1}^{\infty} V'(y)dy - V(x) \\
			=&~{} -\left( 1- \frac{\phi_0^2(x_1)}{\phi_0^2(x)}\right)V(x) \leq 0.
		\end{align*}
		We conclude that $h_0(x)\leq -\mu_0$ for all $x \geq x_0$.
		
		If we consider $x\geq x_1$ we have $V'(x)\leq0$, and using \eqref{eq:h0 ivp} and \eqref{eq:h0'} and by triangle inequality we have
		\be\label{Vpneg}
		\ba
		\left| h_0^2(x) - \mu_0^2\right| = &~{}  |V(x)-h_0'(x)| \\
		\leq &~{} \frac{1}{\phi_0^2(x)} \int_x^{\infty} |V'(y)|\phi_0^2(y)dy + |V(x)| \leq 2|V(x)|.
		\ea
		\ee
		Taking $x$ to infinity in this last inequality, we obtain \eqref{lim h0}.
	\end{proof}
	
	We will need a refined version of the previous result. The next lemma will be used to obtain better bounds for $h_0$ in the interval $(0, x_0)$.
	
	\begin{lemma}
		For all $x\geq0$, one has
		\begin{equation}\label{eq:h2<h0<h1}
			(\mu_0^2 - \widetilde\mu_0^2)x - 2\wtilQ(x)\wtilH(x) \leq h_0(x) \leq \mu_0^2 x - 2\wtilQ(x)\wtilH(x),
		\end{equation}
		where $\widetilde\mu_0$ is defined in \eqref{tmu0}, and $\wtilH$ is the modified version by $\alpha^{-1}$ of the kink $H$ satisfying \eqref{eq:static-EL}. Even more,
		\begin{equation}\label{eq:better bound}
			\mu_0^2 x - R(x) \leq h_0 \qquad \text{for all } x>0,
		\end{equation}
		where we define the auxiliary function
		\[
		R(x):= 2\ln\left(\frac32\right) - 2\ln(\wtilQ) + 2\wtilQ\wtilH + \frac{\mu_0^2 - \tilde\mu_0^2}{2}x^2.
		\]
	\end{lemma}
	\begin{proof}
		First, we consider the initial value problem:
		\begin{equation}\label{eq:tilh ivp}
			\begin{cases}
				\ h_1' = \mu_0^2 + V \\
				\ h_1(0) = 0.
			\end{cases}
		\end{equation}
		Using \eqref{eq:dalpha}, and a change of variables, we have
		\[
		\begin{aligned}
			\int_0^x V(y)dy = &~{}  2\int_0^x \widetilde Q^2(y)(1-\widetilde Q(y))dy =  2\int_0^{\alpha^{-1}(x)}  Q(s)(1- Q(s))ds \\
			=&~{} 2\int_0^{\alpha^{-1}(x)} Q''(s)ds = 2Q' (\alpha^{-1}(x)) =- 2\wtilQ(x)\wtilH(x).
		\end{aligned}
		\]
		Then, the explicit solution for \eqref{eq:tilh ivp} problem is given by
		\[ 
		h_1(x) = \int_0^x (\mu_0^2 + V(y))dy = \mu_0^2 x + \int_0^x V(y)dy = \mu_0^2 x - 2\wtilQ(x)\wtilH(x). 
		\]
		Notice that $h_1(0) = h_0(0)=0$, and from \eqref{eq:tilh ivp} one has $h_0'(x)\leq h_1'(x)$ for all $x\geq0$. Thus, the inequality
		\[
		h_0(x) \leq \mu_0^2 x - 2\wtilQ(x)\wtilH(x),
		\]
		holds for all $x\geq 0$. This proves the upper bound in \eqref{eq:h2<h0<h1}.
		
		Second, we consider the initial value problem:
		\be\label{h_2}
		\begin{cases}
			\ h_2' = \mu_0^2 - \widetilde\mu_0^2 + V \\
			\ h_2(0) = 0.
		\end{cases}
		\ee
		The explicit solution for this problem is given by
		\[
		\begin{aligned}
			h_2(x) = &~{} \int_0^x (\mu_0^2 - \widetilde\mu_0^2 + V(y))dy
			=(\mu_0^2 - \widetilde\mu_0^2) x + \int_0^x V(y)dy \\
			=&~{} (\mu_0^2 - \widetilde\mu_0^2) x - 2\wtilQ(x)\wtilH(x). 
		\end{aligned}
		\]
		Using \eqref{eq:h0 bound 0<x}, one has
		\[
		h_2'(x) \leq \mu_0^2 + V(x) - h_0^2(x) = h_0'(x),
		\]
		for all $x\geq 0$. Since $h_2(0) = h_0(0)=0$, this implies that $h_2(x)\leq h_0(x) \leq0$. Hence,
		\[
		(\mu_0^2 - \widetilde\mu_0^2) x - 2\wtilQ(x)\wtilH(x)\leq h_0(x)
		\]
		for all $x \geq 0$, obtaining the lower bound in \eqref{eq:h2<h0<h1}. We notice that we can improve this bound analogously. Let $h_2$ be given in \eqref{h_2}. If we consider the initial value problem
		\be\label{h_3}
		\begin{cases}
			\ h_3' = \mu_0^2 - h_2^2 + V \\
			\ h_3(0) = 0,
		\end{cases}
		\ee
		the explicit solution $h_3$ is given by
		\[
		h_3(x) = \mu_0^2 x -  2\ln\left(\frac32\right) + 2\ln(\wtilQ) - 2\wtilQ\wtilH - \frac{\mu_0^2 - \tilde\mu_0^2}{2}x^2.
		\]
		Notice that $h_0^2(x) \leq h_2^2(x)$. Since from \eqref{h_3} and \eqref{eq:h0 ivp} we have $h_3'(x)\leq h_0'(x)$ for all $x>0$, and $h_3(0) = h_0(0) = 0$, we conclude that $h_3\leq h_0$, and this proves \eqref{eq:better bound}.
	\end{proof}
	
	Now, we are in condition to obtain estimates for $h_0$ in the interval $(0, x_0)$ in the next lemma, useful for the proof of repulsivity in the transformed potential.
	
	\begin{lemma}\label{lem8p8}
	 One has the following properties:
		\begin{enumerate}
			\item \label{l:b3} For $0\leq x\leq x_{2,1}$ we have
			\begin{equation}\label{eq:h0 bound 0<x<x0}
				-4k_0\wtilH(x) \leq h_0(x)
			\end{equation}
		where $k_0:= \frac{9 - 4\mu_0^2}{3}$ is a positive constant.
			\item \label{l:b4} For all $x$ such that $x_{2,1} \leq x \leq x_0$,
			\begin{equation}\label{eq:h0 bound x0<x<x1}
				(\mu_0^2 - \widetilde\mu_0^2) (x-x_0) -\widetilde\mu_0 \leq h_0(x) \leq - \frac{\mu_0}{x_0}x.
			\end{equation}
		\end{enumerate}	
	\end{lemma}
	\begin{proof}	
		Proof of \eqref{l:b3}. Thanks to Lemma \ref{valor mu0} the constant $k_0$ is strictly positive. We define the auxiliary function
		\begin{equation}\label{def g}
			g(x) = h_0(x) +  4k_0\wtilH(x).
		\end{equation}
		By direct calculation, we obtain $g(0) = g'(0) = 0$, and by the mean value theorem,
		\begin{equation}\label{eq:MVT g}
			g(x) = g'(\xi)x,
		\end{equation}
		for some $\xi \in (0, x)$. Thus, to prove the positivity of $g$ for $0\leq x\leq x_{2,1}$, it is enough to study the sign of $g'$.
		Deriving $g$, and using \eqref{eq:h0 ivp}, \eqref{tilde Q tilde H}, \eqref{eq:dalpha}, one has that proving $g'\geq0$ is equivalent to prove
		\begin{equation}\label{eq:equiv bound}
			h_0^2 \leq \mu_0^2 + V + \frac43 k_0\wtilQ^2.
		\end{equation}
		for $0\leq x\leq x_{2,1}$. Using \eqref{eq:better bound} we have that
		\begin{equation}\label{eq:1}
			h_0^2 \leq \mu_0^4x^2 - 2\mu_0^2 xR(x) + R^2(x).
		\end{equation}
		The RHS of this last equation is explicit except for $\mu_0$, so comparing both RHSs of \eqref{eq:equiv bound} and \eqref{eq:1}, it is sufficient to prove the following,
		\[
		\mu_0^4x^2 - 2\mu_0^2 xR(x) + R^2(x) \leq \mu_0^2 + V + \frac43 k_0\wtilQ^2,
		\]
		equivalent to prove
		\begin{equation}\label{eq:2}
		\mu_0^4 x^2 + \left(\frac49\wtilQ^2 - 2xR(x) - 1\right)\mu_0^2 \leq V + \wtilQ^2 - R^2
		\end{equation}
		for all $0\leq x\leq x_{2,1}$. Now, applying Lemma \ref{valor mu0}, one has
		\begin{equation}\label{eq:3}
		\mu_0^4 x^2 + \left(\frac49\wtilQ^2 - 2xR(x) - 1\right)\mu_0^2 \leq G(\alpha^{-1}(x)),
		\end{equation}
		where
		\[
		G(s):= (0.883)^4  \alpha(s)^2 + \left(\frac49 Q^2 - 2\alpha(s)R(\alpha(s)) - 1\right)(0.808)^2
		\]
		is given by explicit functions. Combining these last inequalities, we obtain
		\begin{equation}\label{eq:4}
			G(s) \leq V(\alpha(s)) + Q^2(s) - R^2(\alpha(s)),
		\end{equation}
		for $0\leq s \leq \alpha^{-1}(x_{2,1})$ (see Figure \ref{fig:better bound}). Indeed, \eqref{eq:4} is equivalent to the estimate
		\[
		\ba
		& (0.883)^4  \alpha(s)^2 + (0.808)^2 \left(\frac49Q(s)^2 - 2\alpha(s)R(\alpha(s)) - 1\right)\\
		&\qquad \leq  3Q^2(s) -2Q^3(s)  -R^2(\alpha(s)),
		\ea
		\]
		where we have used that 
		\[
		V(\alpha(s)) = 2Q^2(s)(1 -Q(s)),
		\]
		and
		\be\label{Ralphas}
		R(\alpha(s)):= 2\ln\left(\frac32\right) - 2\ln(Q) (s)+ 2 Q (s) H(s) + \frac{\mu_0^2 - \tilde\mu_0^2}{2}\alpha(s)^2.
		\ee
		
		\begin{figure}[ht]
			\centering
			\includegraphics[width=0.89\textwidth]{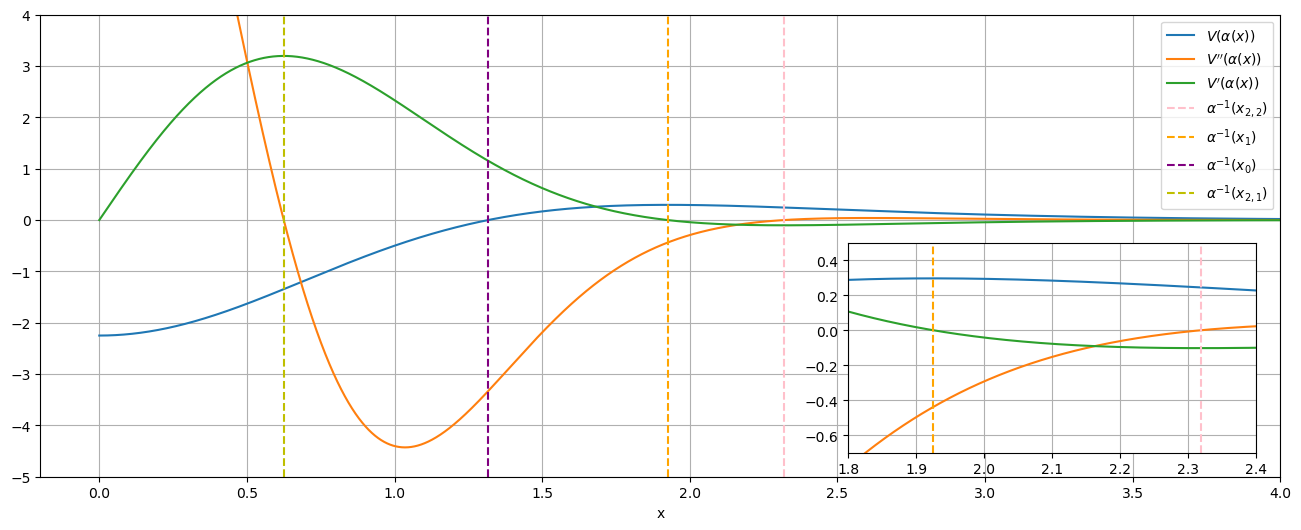}
			\includegraphics[width=0.89\textwidth]{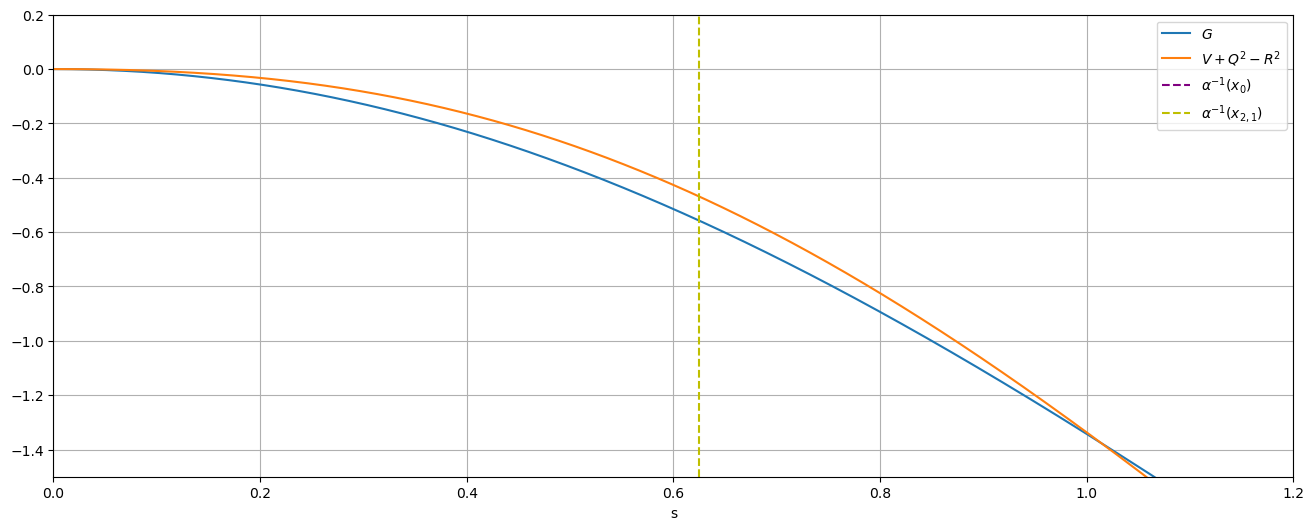}
			\caption{Above: Numerical computation of $V(\alpha(x)), V'(\alpha(x)), V''(\alpha(x))$ where their roots are explicitly plotted in dashed vertical lines. In particular we observe that $0 < x_{2,1} < x_0 < x_1 < x_{2,2}$. Below: Numerical computation of auxiliary functions $G(s)$ and $ V(\alpha(s)) + Q^2(s) - R^2(\alpha(s)$. In particular we observe that $G \leq V + Q^2 + R$ for $x\in(0, x_{2,1})$.}
			\label{fig:better bound}
		\end{figure}
		
		Replacing \eqref{eq:4} into \eqref{eq:3} we obtain \eqref{eq:2}, and we conclude via \eqref{eq:1} that $g'(x)\geq0$. This proves that $g$ is a positive function for  $0\leq x \leq x_{2,1}$. Hence, by \eqref{def g} and \eqref{eq:MVT g} we conclude \eqref{eq:h0 bound 0<x<x0}.
		
		\medskip
		
		Proof of \eqref{l:b4}. We claim that $h_0$ is a convex function for $x\in(0, x_0)$. First, from the proof of Lemma \ref{lem:bounds h0} we know that $h_0'$ has a unique root denoted by $\bar x$, with $h_0'(x)< 0$ in $(0, \bar x)$ and negative sign in $(\barx, \infty)$. Now using that $V(x_0)=0$, \eqref{eq:h0 bound x1<x}, and \eqref{eq:h0 ivp}, we have
		\[
		h_0'(x_0) = \mu_0^2 - h_0^2(x_0) \leq \mu_0^2 - \mu_0^2 = 0.
		\]
		This implies that $h_0'$ is negative in $(0, x_0)$.
		
		In addition, if $x\in (x_{2,1}, x_0)$ we know from \eqref{eq:h0 bound 0<x} that $-\widetilde \mu_0 \leq h_0$. Hence, replacing in \eqref{eq:h0 ivp}, we obtain
		\begin{equation}\label{eq:bound h0'}
			\mu_0^2 - \widetilde\mu_0^2 + V(x) \leq \mu_0^2 - h_0^2 + V = h_0'(x) \leq 0. 
		\end{equation}
		Taking derivative in \eqref{eq:h0 ivp}, using that $h_0', h_0 \leq 0$, the lower bounds from \eqref{eq:h0 bound 0<x} \eqref{eq:better bound} and \eqref{eq:bound h0'},
		\begin{align*}
			h_0'' =&~{} V' -2h_0h_0' \geq V' - 2(\mu_0^2 x - R(x)) h_0' \\
			\geq &~{} V' - 2(\mu_0^2 x - R(x)) (\mu_0^2 - \widetilde\mu_0^2 + V) \\
			\geq &~{} V' - 2(0.808^2 x - R(x)) \left(-\frac{8}{27} + V\right) \\
			=: &~{} j_1(\alpha^{-1}(x))\wtilH
		\end{align*}
		where $j_1$ is obtained employing Lemma \ref{valor mu0}. Computing this function, we have that 
		\be\label{checkeo7}
		\hbox{$j_1(s)>0$ for all $s\in (\alpha^{-1}(x_{2,1}), \alpha^{-1}(x_0))$}
		\ee
		 (see Fig. \ref{fig:positivity bounds} upper panel). Hence, by bijectivity of $\alpha:\R\to\R$, we conclude 
		\[
		h_0''(x) \geq j_1(\alpha^{-1}(x)) > 0,
		\]
		for all $x\in (x_{2,1}, x_0)$. This proves the convexity of $h_0(x)$ over $(x_{2,1}, x_0)$. Using \eqref{eq:h0 ivp}, \eqref{eq:h0 bound 0<x}, if $x_{2,1} \leq x \leq x_0$, by definition of convexity,
		\[
		\begin{aligned}
			h_0(x) \geq &~{} h_0'(x_0)(x - x_0) + h_0(x_0) \\
			= &~{} (\mu_0^2 - h_0^2(x_0) )(x - x_0) + h_0(x_0) \geq (\mu_0^2 - \widetilde\mu_0^2) (x - x_0) -\widetilde\mu_0.
		\end{aligned}
		\]
		This proves the lower bound in \eqref{eq:h0 bound x0<x<x1}.
		
		If now $0\leq x\leq x_{2,1}$, first we define the positive constants
		\[
		\begin{aligned}
		\kbar_0:=\frac{1}{3}(9 - 4(0.883)^2), \quad \bark_0:=\frac{1}{3}(9 - 4(0.808)^2).	
		\end{aligned}
		\]
		Notice that thanks to Lemma \ref{valor mu0} the inequality $\kbar_0<k_0<\bark_0$ holds. Using that $h_0', h_0 \leq 0$, $V'\geq0$, \eqref{eq:h0 bound 0<x<x0},\eqref{eq:h0 ivp}, and replacing $V, V'$, we have the following set of inequalities
		\[
		\begin{aligned}
			h_0'' = &~{} V' - 2h_0h_0'  \\		
			\geq &~{} V' + 2k_0\left(\mu_0^2 + V - k_0^2 \wtilH^2 \right)\wtilH \\
			= &~{} \left(-2k_0(k_0^2 - \mu_0^2) + \frac43 k_0^3\wtilQ + 4k_0 \wtilQ^2 - 4(1 + k_0)\wtilQ^3 + 6\wtilQ^4 \right)\wtilH \\
			\geq &~{} \left(-2\bark_0(\bark_0^2 - \mu_0^2) + \frac43 \kbar_0^3\wtilQ + 4\kbar_0\wtilQ^2 - 4(1 + \bark_0)\wtilQ^3 + 6\wtilQ^4 \right)\wtilH \\
			:= &~{} j_2(\alpha^{-1}(x))\wtilH.
		\end{aligned}
		\]
		Considering the variable $s = \alpha^{-1}(x)$, we obtain that $j_2(s)$ simplifies to 
		\be\label{j2final}
		\begin{aligned}
		j_2(s) = &~{} -2\bark_0(\bark_0^2 - \mu_0^2) + \frac43 \kbar_0^3 Q(s) \\
		&~{} + 4\kbar_0 Q^2 (s)- 4(1 + \bark_0) Q^3(s) + 6 Q^4(s).
		\end{aligned}
		\ee
		The expression \eqref{j2final} is bounded employing Lemma \ref{valor mu0}. Computing (see Fig. \ref{fig:positivity bounds} upper panel), we have that 
		\be\label{checkeo7}
		\hbox{$j_2(s)>0$ for all $s\in (0, \alpha^{-1}(x_{2,1}))$}.
		\ee
		Hence, by bijectivity of $\alpha$, we conclude $h_0''(x)>0$ for all $x\in (0, x_{2,1})$.
		
		This proves the convexity of $h_0$ over $(0, x_0)$, and it is enough to conclude \eqref{eq:h0 bound x0<x<x1}. Indeed, using convexity between $(0, h_0(0))$ and $(x_0,
		h_0(x_0))$, and \eqref{eq:h0 bound x1<x}, we have 
		\[
		h_0(x) \leq \frac{h_0(x_0)}{x_0}x \leq  - \frac{\mu_0}{x_0}x.
		\]
		This proves the upper bound in \eqref{eq:h0 bound x0<x<x1}, where $x_{2,1}<x<x_0$ and we conclude the proof of Lemma \ref{lem8p8}.
	\end{proof}

	\subsection{Positivity.}\label{posi}
	Now, employing the estimates over $h_0$ in the previous subsection and the integral form of $h_0'$, we are in position to deal with the sign of $V_0$.
	
	\begin{lemma}\label{lem:positivity}
		The potential $V_0$ in \eqref{eq:L0} is non-negative over the real line. In particular $L_0$ only has nonnegative spectrum.
	\end{lemma}
	\begin{proof}
		To prove the positivity of $V_0$, first we will obtain a convenient formulation of the potential in terms of an integral. By definition of $V_0$ and \eqref{eq:h0'} we have
		\[
		V_0(x) = V(x) + \frac{2}{\phi_0^2(x)}\int_x^{\infty}V'(y)\phi_0^2(y)dy.
		\]
		Integrating by parts to eliminate the potential $V$ on the right hand side, and using \eqref{eq:h0 phi0}, we obtain
		\begin{align*}
			V_0(x) =& \frac{1}{\phi_0^2(x)}\int_x^{\infty}V'(y)\phi_0^2(y)dy - \frac{1}{\phi_0^2(x)}\int_x^{\infty} V(y)(\phi_0^2(y))'dy \\
			=&\frac{1}{\phi_0^2(x)}\int_x^{\infty}[V'(y) - 2h_0(y)V(y)]\phi_0^2(y)dy.
		\end{align*}
		Thus, we have the integral formulation of $V_0$,
		\begin{align*}
			V_0(x) &=\frac{1}{\phi_0^2(x)}\int_x^{\infty}K(y)\phi_0^2(y)dy,
		\end{align*}
		where we have defined $K(y):= V'(y) - 2h_0(y)V(y).$ We will prove the positivity of $K(y)$ for all $y\geq0$.
		
		For $y\geq x_0$ this is straightforward, since we know that $V(y), V'(y)\geq 0$ and $h_0(y)<0$, then $K(y)$ must be non-negative.
		
		For $x_{2,1}\leq y\leq x_0$, we know that $V(y), h_0(y)\leq 0$. Using the bound \eqref{eq:better bound} for $h_0(y)$, using Lemma \ref{valor mu0}, and replacing directly $V', V$, we have
		\[
		\begin{aligned}
			K(y) =&~{} V' - 2h_0V \geq V' - 2(\mu_0^2 x - R)V \geq V' - 2(0.808^2 x - R)V \\
			= &~{} 2\wtilQ^2[2(0.808^2 x - R) - 2((0.808^2 x - R) + \wtilH)\wtilQ + 3\wtilH\wtilQ^2 ] \\
			=: &~{} 2\wtilQ^2 k_1(\alpha^{-1}(y)).
		\end{aligned}
		\]
		We recall that the function $k_1$ is explicitly known employing Lemma \ref{valor mu0}. Indeed, we have
		\be\label{def_k1}
		\ba
		k_1(s):= &~{} 2(0.808^2 \alpha(s) - R(\alpha(s))) - 2((0.808^2  \alpha(s)  - R(\alpha(s))) \\
		&~{}+ H(s))Q(s) + 3H(s)Q(s)^2,
		\ea
		\ee
		with $R(\alpha(s))$ given in \eqref{Ralphas}. Computing this, we have that $k_1(s)$ from \eqref{def_k1} satisfies
		\be\label{checkeo6}
		\hbox{$k_1(s)>0$ for all $s\in (\alpha^{-1}(x_{2,1}), \alpha^{-1}(x_0))$}
		\ee
		(see Fig. \ref{fig:positivity bounds} lower panel). Hence, by bijectivity of $\alpha$, we conclude $K(y)>0$ for all $y\in (x_{2,1}, x_0)$.
		
		For $0\leq y \leq x_{2,1}$ we just consider the bound \eqref{eq:h0 bound 0<x<x0} for $h_0$ instead of \eqref{eq:h0 bound 0<x}. Then we proceed analogously:
		\be\label{cota_K}
		\begin{aligned}
			K(y) \geq &~{} V' - \frac83\left(\mu_0^2 - \frac94 \right)\wtilH V \\
			= &~{} 2\wtilQ^2\wtilH\left[-\frac23(4\mu_0^2 - 9) + 8\left(\frac13\mu_0^2 - 1\right)\wtilQ + 3\wtilQ^2 \right] \\[0.1cm]
			\geq &~{} 4\wtilQ^2\wtilH\left[-\frac13(4(0.883)^2 - 9) + 4\left(\frac13 (0.808)^2 - 1\right)\wtilQ + \frac32\wtilQ^2 \right] \\
			=: &~{} 4\wtilQ^2\wtilH k_2(\alpha^{-1}(y)),
		\end{aligned}
		\ee
		where $k_2$ is explicitly known employing Lemma \ref{valor mu0}. Indeed, one has the simple formula
		\be\label{def_k2}
		k_2(s):= -\frac13(4(0.883)^2 - 9) + 4\left(\frac13 (0.808)^2 - 1\right)Q(s) + \frac32 Q^2(s).
		\ee
		Computing, we have that 
		\be\label{checkeo5}
		\hbox{$k_2(s)>0$ for all $s\in (0, \alpha^{-1}(x_{2,1}))$}
		\ee
		(see Fig. \ref{fig:positivity bounds} lower panel). Hence, by bijectivity of $\alpha$, we conclude $K(y)>0$ for all $y\in (0, x_{2,1})$.
	\end{proof}
	
	One of the most crucial properties about $L$ for our analysis of the stability of the kink is that it possesses only one negative eigenvalue.
	
	\begin{corollary}\label{cor:uniqueness}
		The operator $L$ has a unique negative eigenvalue $-\mu_0^2<0$ of multiplicity one.
	\end{corollary}
	
	\begin{remark}
		Corollary \ref{cor:uniqueness} shows the unstable character of the kink solution $H$, under which the asymptotic stability could only hold if one already has orbital stability. 
	\end{remark}
	
	\begin{proof}
		This is just a consequence of removing the first eigenvalue once we obtain the transformed super-symmetric partner operator $L_0$. We recall the following decomposition
		\begin{equation*}
			L = (-\px - h_0)(\px + h_0) - \mu_0^2 = U^*U - \mu_0^2,		
		\end{equation*}
		and changing the order of the operators $U$ and $U^*$, we define
		\begin{equation}\label{eq:L0 decomposition}
			L_0 =  (\px + h_0)(-\px - h_0) - \mu_0^2 = UU^* - \mu_0^2,
		\end{equation}
		obtaining the super-symmetric relation
		\begin{equation}\label{UL=L0U}
			UL=L_0U
		\end{equation}
		which is, by construction, isospectral to $L$ except for $\lambda = -\mu_0$. This is, we claim 
		\[
		\sigma_p(L_0) = \sigma_p(L)\setminus\{-\mu_0^2\}.
		\]
		Let $\lambda \neq -\mu_0^2$ be an eigenvalue of $L$, with the corresponding eigenfunction $\phi$. Then, by equation \eqref{UL=L0U} we get $L_0(U\phi) = \lambda U\phi$. Since by Lemma \ref{lemma:Lproperties} $\lambda_0 = -\mu_0^2$ is a simple eigenvalue, we have that $U\phi\not\equiv 0$. This proves that $\sigma_p(L)\setminus\{-\mu_0^2\}\subseteq \sigma_p(L_0)$. For the reversed inclusion, we only need to prove that $-\mu_0^2 \notin \sigma_p(L_0)$, since for the rest we could repeat the same procedure as above, but relative to the eigenvalues of $L_0$. By contradiction, we assume that there exists some $\varphi\in L^2(\R)$ such that $L_0\varphi = -\mu_0^2\varphi$. Then, by \eqref{eq:L0 decomposition}, we obtain $UU^*\varphi = 0$, and using that $\text{ran}(U^*)\perp \ker(U)$ we have that $U^*\varphi=0$, which implies that $\varphi = \phi_0^{-1}$, which is a contradiction since $g\in L^2(\R)$.
		
		By Lemma \ref{lem:positivity} we conclude that $L_0$ has no negative eigenvalues, and from the above we conclude that $-\mu_0^2$ is the unique negative eigenvalue associated with the operator $L$. 
	\end{proof}
	
	\begin{corollary}\label{cor:parity}
		Given $\phi_0$ eigenfunction associated with the unique negative eigenvalue $-\mu_0^2$, then $\phi_0$ is an even function and $\px \phi_0$ is odd.
	\end{corollary}
	\begin{proof}
		The parity follows from the fact that $L$ is invariant over the reflection $x \to -x$, so the eigenfunctions are even or odd, and since $\phi_0$ is positive in the real line we conclude it is even. Since $\lambda_0$ is the unique negative eigenvalue of multiplicity one, $\phi_0$ is unique, even, and $\px\phi_0$ is odd.
	\end{proof}
	
	\begin{figure}[h!]
		\centering
		\includegraphics[width=0.95\textwidth]{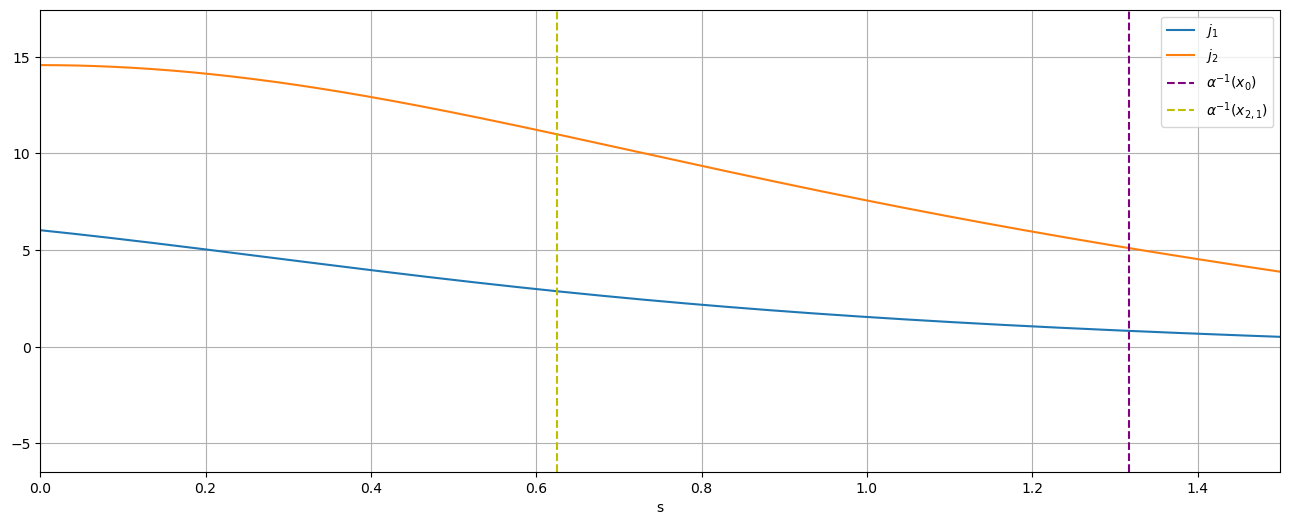}
		\includegraphics[width=0.95\textwidth]{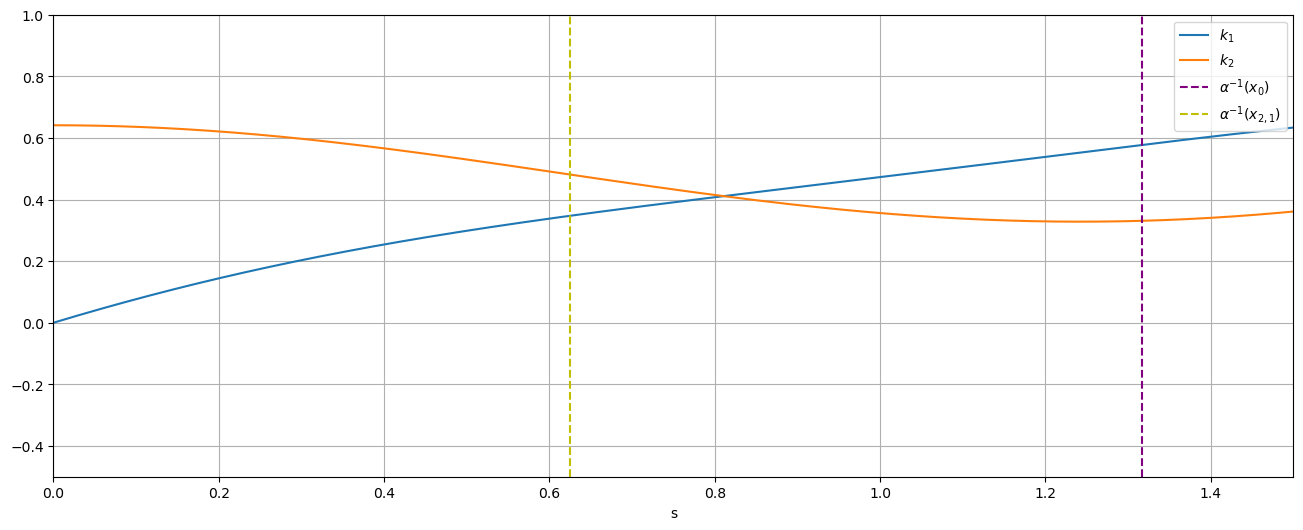}
		\caption{Above: Numerical computation of $j_1(s)$, lower bound for $h_0''$ for $s$ in $(x_{2,1}, x_0)$, and  $j_2(s)$, lower bound for $s$ in $(0, x_{2,1}$). Below: Numerical computation of $k_1(s)$, lower bound for $K(\alpha^{-1}(s))$ with $s$ in $(x_{2,1}, x_0)$, and  $k_2(s)$, lower bound for $K(\alpha^{-1}(s))$ with $s$ in $(0, x_{2,1}$).}
		\label{fig:positivity bounds}
	\end{figure}

	\subsection{Repulsivity}
	\begin{lemma}\label{lem:repulsivity}
		The derivative of the transformed potential $V_0'(x)$ is odd and negative for any $x\neq0$. In particular, $L_0$ has a repulsive potential.
	\end{lemma}
	
	The rest of this section is devoted to prove Lemma \ref{lem:repulsivity}.
	
	\subsubsection{An integral formula.} By \eqref{eq:h0 phi0} we have that $(\phi_0^2)' = 2h_0\phi_0^2$. Using this, the definition of $V_0$ in \eqref{eq:L0} and $h_0$, \eqref{eq:h0'}, and integration by parts, we get
	\begin{align*}
		V_0'(x) &= 4h_0(x)h_0'(x) - V'(x) \\[0.1cm]
		&=  -\frac{2h_0(x)}{\phi_0^2(x)}\int_x^\infty V'(y)\phi_0^2(y)dy - \frac{h_0(x)}{\phi_0^2(x)} \int_x^{\infty} \frac{V'(y)}{h_0(y)}(\phi_0^2(y))'dy - V'(x) \\[0.1cm]
		&= -\frac{2h_0(x)}{\phi_0^2(x)}\int_x^\infty V'(y)\phi_0^2(y)dy  +  \frac{h_0(x)}{\phi_0^2(x)}\int_x^\infty \left(\frac{V'(y)}{h_0(y)}\right)'\phi_0^2(y) dy \\
		&\quad - \left.\frac{h_0(x)V'(y)\phi_0^2(y)}{\phi_0^2(x)h_0(y)}\right|_{x}^{\infty} - V'(x)\\[0.1cm]
		&= \frac{h_0(x)}{\phi_0^2(x)}\int_x^\infty \left(\frac{V'(y)}{h_0(y)} - 2V(y)\right)' \phi_0^2(y) dy \\
		&= \frac{h_0(x)}{\phi_0^2(x)}\int_x^\infty \left(\frac{V''(y)}{h_0(y)} - \frac{V'(y)h_0'(y)}{h_0^2(y)} - 2V'(y)  \right)\phi_0^2(y) dy \\[0.1cm]
		&= \frac{h_0(x)}{\phi_0^2(x)}\int_x^\infty \left(V''(y)h_0(y) - V'(y)h_0'(y) - 2V'(y)h_0^2(y)  \right)\left(\frac{\phi_0}{h_0}\right)^2(y) dy,
	\end{align*}
	Thus, we have the equivalent formulation
	\begin{equation}\label{eq:integral formulation}
		V_0'(x) = \frac{h_0(x)}{\phi_0^2(x)}\int_x^\infty I(y) \left(\frac{\phi_0}{h_0}\right)^2(y) dy,
	\end{equation}
	where, using equation \eqref{eq:h0 ivp}, we have
	\begin{equation}\label{def:Iy}
		I(y) = V''(y)h_0(y) - V'(y)(h_0^2(y) + \mu_0^2 + V(y)).  
	\end{equation}
	Due to the dependence of this expression on the sign of the potential and its derivatives, we will divide the proof depending on the roots $\{x_0, x_1, x_{2,1},x_{2,2}\}$ (see Lemma \ref{def: roots}).
	
	\medskip
	
	To prove that $V_0'$ is non positive, we restrict our analysis to the interval $(0, \infty)$ by parity. We will prove the positivity of $I(y)$ for all $y\geq 0$ by separate cases.
	
	\subsubsection{Positivity for $x_1\leq y < \infty$.}
	Firstly, we consider the case $y\geq x_{2,2}$.  Then Remark \ref{rem:V} ensures that $V(y), V''(y)\geq0$, $V'(y)\leq0$. We apply in \eqref{def:Iy} the bounds \eqref{eq:h0 bound 0<x} and \eqref{eq:h0 bound x1<x} for $h_0$, and Lemma \ref{valor mu0}:
	\[
	\begin{aligned}
		I(y)  = &~{} - V''(y)|h_0(y)| + |V'(y)|(h_0^2(y) + \mu_0^2 + V(y)) \\
		\geq &~{}  - \widetilde\mu_0 V''(y) + |V'(y)|\left( 2\mu_0^2 + V(y) \right) \\
		\geq &~{} - 1.038 V''(y) + (2\cdot 0.808^2 + V(y))|V'(y)|.
	\end{aligned}
	\]
	Replacing directly $V, V'$, $V''$ and considering the variable $s = \alpha^{-1}(y)$, we obtain
	\[
	\begin{aligned}
		I(\alpha(s))  \geq &~{} - 2.075 Q^4\left( 6 -\frac{50}3 Q +9Q^2\right) \\
		&~{}  + 4(2 - 3Q)(0.652 + Q^2 - Q^3)Q^3 H \\
		= &~{} 2Q^3\left[ 2.611 - 6(1.038 + 0.652H)Q + \left(\frac{50}{3}1.038 + 4H \right)Q^2 \right.\\
		&~{}\qquad \left. - (9.342 + 20H)Q^3 + 6Q^4H \right] \\
		=: &~{} 2Q^3 i_1(s).
	\end{aligned}
	\]
	By the exponential decay of $Q$, we obtain explicitly via computation that 
	\be\label{checkeo4}
	\hbox{$i_1(s)>0$ for all $s\geq \alpha^{-1}(x_{2,2})$}
	\ee
	(see Fig. \ref{fig:bound 0 x1} down panel). Hence, we conclude $I(y)>0$ for all $y\geq x_{2,2}$ by the bijection of $\alpha:\R\to\R$.
	
	If now $x_1 \leq y \leq x_{2,2}$, then $V(y)\geq0$, $V'(y), V''(y)\leq0$, applying \eqref{eq:h0 bound 0<x}, \eqref{eq:h0 bound x1<x}, and Lemma \ref{valor mu0}, replacing $V, V'$ and $V''$,
	\[
	\begin{aligned}
		I(y) = &~{} |V''(y)h_0| + |V'(y)| \left(  h_0^2(y) + \mu_0^2 + V(y) \right) \\
		\geq &~{} \mu_0 |V''(y)| + |V'(y)| \left( 2\mu_0^2 + V(y) \right) \\
		\geq &~{}  0.808 |V''(y)| + |V'(y)| \left(2\cdot 0.808^2 + V(y) \right).
	\end{aligned}
	\]
	Again, replacing $V, V'$, $V''$ and considering the variable $s = \alpha^{-1}(y)$, we obtain
	\[
	\begin{aligned}
		I(\alpha(s)) = &~{} -2\mu_0  Q^4\left( 6 -\frac{50}3 Q +9Q^2\right) + 4(2 - 3Q)(0.808^2 + Q^2 - Q^3)Q^3H \\
		= &~{} 2Q^3H \bigg[ 4\cdot 0.808^2 H - 6\cdot0.808(1 + 0.808H)Q \\[0.1cm]
		&~{} \hspace{1.3cm} + \left(\frac{50}{3}\cdot 0.808 + 4H \right)Q^2 - (10H + 9\cdot 0.808)Q^3 + 6HQ^4 \bigg] \\
		=: &~{} 2Q^3H i_2(s),
	\end{aligned}
	\]
	where $\hat k(s)$ is explicitly known employing Lemma \ref{valor mu0}. Computing this function, we have that 
	\be\label{checkeo3}
	\hbox{$i_2(s)>0$ for all $s\in (\alpha^{-1}(x_1), \alpha^{-1}(x_{2,2}))$}
	\ee
	(see Fig. \ref{fig:bound 0 x1}). Hence, by bijectivity of $\alpha$, we conclude $I(y)>0$ for all $y\in (x_1, x_{2,2})$.
	
	\subsubsection{Positivity for $x_0\leq y < x_1$.} In this case $V(y), V'(y)\geq0$, and $V''(y)\leq 0$. This, combined with inequalities \eqref{eq:h0 bound x1<x}, \eqref{eq:h0 bound 0<x}, and Lemma \ref{valor mu0}, gives us that $I$ satisfies the following inequality for all $y\in[x_0, x_1]$:
	\begin{align*}
		I(y) = &~{} |V''(y)h_0(y)| -  V'(y) \left( h_0^2(y) + \mu_0^2 + V(y) \right) \\
		\geq &~{} \mu_0 |V''(y)| - V'(y) \left( \widetilde\mu_0^2 + \mu_0^2 + V(y) \right) \\
		\geq &~{}  0.808|V''(y)| - V'(y) \left( 1.959 + V(y) \right) .
	\end{align*}
	Replacing $V, V'$, $V''$ and considering the variable $s = \alpha^{-1}(y)$, we obtain
	\[
	\begin{aligned}
		I(\alpha(s)) \geq &~{} -2\cdot 0.808  Q^4\left( 6 -\frac{50}3 Q +9Q^2\right) \\
		&~{} + 2(2 - 3Q)(1.959 + 2Q^2 - 2Q^3)Q^3H \\
		= &~{} 2Q^3\bigg[ 3.842 H - 3(1.959 + 1.616 H)Q + \left(\frac{50}{3}\cdot 0.808 + 4H \right)Q^2 \\
		&~{} \qquad  - (7.272 + 10H)Q^3 + 6HQ^4 \bigg] \\
		=: &~{} 2Q^3 i_3(s),
	\end{aligned}
	\]
	where $i_3(s)$ is explicitly known thanks to Lemma \ref{valor mu0}. A simple graph reveals that 
	\be\label{checkeo2}
	\hbox{$i_3(s)>0$ for all $s\in (\alpha^{-1}(x_0), \alpha^{-1}(x_{1}))$}
	\ee
	(see Fig. \ref{fig:bound 0 x1} above). Hence, by bijectivity of $\alpha$, we conclude $I(y)>0$ for all $y\in (x_0, x_1)$.
	
	\begin{figure}[ht]
		\centering
		\includegraphics[width=0.95\textwidth]{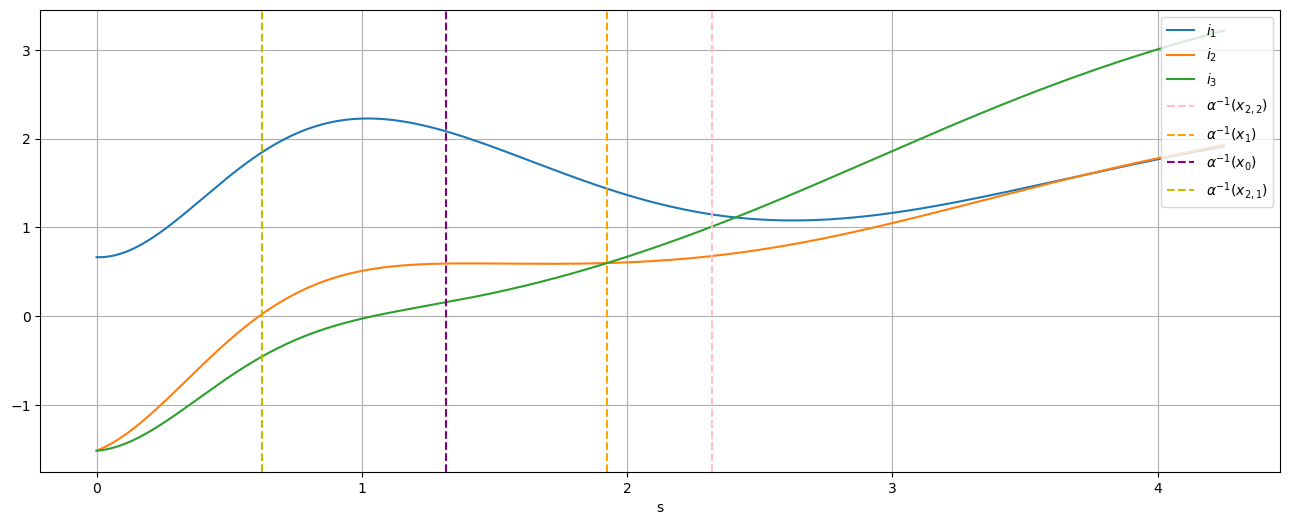}
		\includegraphics[width=0.95\textwidth]{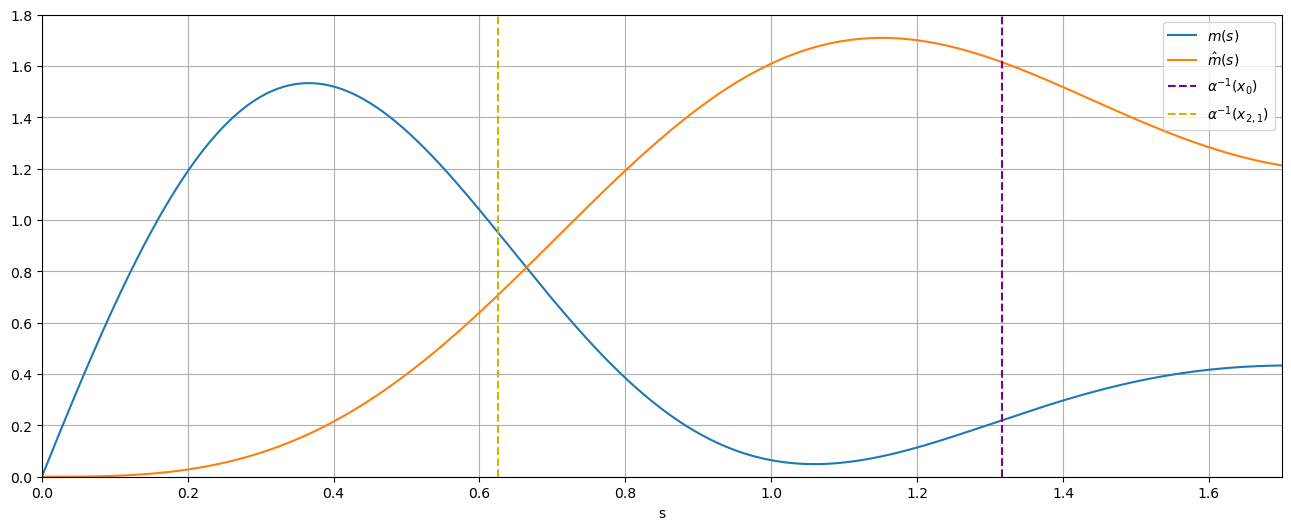}
		\caption{Above: Numerical computation of the bounds for $I(\alpha(x))$ in the intervals $(\alpha^{-1}(x_0), \alpha^{-1}(x_1))$, $(\alpha^{-1}(x_1), \alpha^{-1}(x_{2,2}))$, and $(\alpha^{-1}(x_{2,2}), \infty)$. Below: Numerical computation of the bounds for $I(\alpha(x))$ in terms if $m(s)$ and $\hat m(s)$ in the intervals $(0, \alpha^{-1}(x_{2,1}))$ and $(\alpha^{-1}(x_{2,1}), \alpha^{-1}(x_0))$.}
		\label{fig:bound 0 x1}
	\end{figure}
	
	\subsubsection{Positivity for $x_{2,1}\leq y < x_0$.} If $y$ is a positive real number such that $x_{2,1} \leq y < x_0$, then $V(y), V''(y)\leq 0$, $V'(y)\geq0$. We separate the study in two cases.
	
	\medskip
	
	{\bf Case 1.} If $h_0^2(y) + \mu_0^2 + V(y) \leq 0$, directly by the sign of the expression in \eqref{def:Iy}
	\begin{align*}
		I(y) = |V''(y)h_0(y)| + |V'(y)(h_0^2(y) + \mu_0^2 + V(y))| \geq 0.
	\end{align*}

	\medskip	
	
	{\bf Case 2.} On the other hand, if $h_0^2(y) + \mu_0^2 + V(y) \geq 0$, by \eqref{eq:h0 bound x0<x<x1} and Lemma \ref{valor mu0} we know
	\[
	\begin{aligned}
	h_0^2(y) + \mu_0^2 + V(y) \geq & \left(\frac{8}{27} (x-x_0) + \widetilde\mu_0 \right)^2 + \mu_0^2 + V(y) \\
	\geq & \left(\frac{8}{27} (x-x_0) + 0.974 \right)^2 + 0.652 + V(y) .
	\end{aligned}
	\]
	Hence, using \eqref{eq:h0 bound x0<x<x1} and the above estimate to bound by below \eqref{def:Iy},
	\begin{align*}
		I(y) \geq -\frac{\mu_0}{x_0} y V''(y) - V'(y)\left( \left(	\frac8{27} (y-x_0) + 0.974 \right)^2 + 0.652 + V(y)\right).
	\end{align*}
	Replacing $V, V'$, $V''$ and considering the variable $s = \alpha^{-1}(y)$, we obtain
	\[
	\begin{aligned}
		 I(\alpha(s)) \geq&~ - 2\frac{0.808}{x_0}\alpha(s)  Q^4\left( 6 -\frac{50}3 Q +9Q^2\right) \\
		&~{} + 2(2-3Q)Q^3H\left( \left(	\frac8{27} (\alpha(s)-x_0) + 0.974 \right)^2 + 0.652 + 2Q^2(1-Q)\right) \\
		=:&~ m(s),
	\end{aligned}
	\]
	where $m(s)$ is explicitly known employing Lemma \ref{valor mu0}. The only problem is probably an estimate for $x_0$. Everything being explicit, one easily checks that 
	\be\label{checkeo1}
	\hbox{$m(s)>0$ for all $s\in (\alpha^{-1}(x_{2,1}), \alpha^{-1}(x_0))$}
	\ee
	(see Fig. \ref{fig:bound 0 x1} panel below). Hence, since $\alpha$ is bijective, we conclude $I(y)>0$ for all $y\in (x_{2,1}, x_0)$.

	\subsubsection{Positivity for $0\leq y < x_{2,1}$.} Finally, for this case $V(y)\leq 0$, $V'(y), V''(y) \geq 0$, and using \eqref{eq:h0 bound 0<x<x0} we obtain
		\[
		\ba
		& h_0^2(y) + \mu_0^2 + V(y) \\
		&\qquad \leq \left(\mu_0^2 - \frac94 \right)^2y^2 + \mu_0^2 + V(y) \leq \left(0.652 - \frac94 \right)^2 y^2 + 0.78 + V(y) \leq 0,
		\ea
		\]
	where the last inequality was obtained using the bounds for $\mu_0$ of Lemma \ref{valor mu0}. Hence, this combined with inequalities \eqref{eq:h0 bound x1<x}, \eqref{eq:h0 bound 0<x} gives us that $I$ satisfies for all $y\in(0, x_{2,1})$:
	\begin{align*}
		I(y) = &~{} V''(y)|h_0(y)| +  V'(y) \left| h_0^2(y) + \mu_0^2 + V(y) \right|.
	\end{align*}
	Bounding by below, we have
	\[
	I(y) \geq \left(0.652 - \frac94\right) y V''(y) - V'(y) \left( \left(0.652 - \frac94 \right)^2y^2 + 0.652 + V(y)\right)
	\]
	Replacing $V, V'$, $V''$ and considering the variable $s = \alpha^{-1}(y)$, we obtain 
	\begin{align*}
		I(\alpha(s)) \geq &~{} -3.196  \alpha(s)Q^4\left( 6 -\frac{50}3 Q +9Q^2\right) \\
		&~{}  + 2Q^3 H(2 - 3Q)\left( 2.553604 \alpha^2(s) + 0.652 + 2Q^2(1-Q)\right) \\ 
		=: &~{} \hat m(s).
	\end{align*}
	where $\hat m(s)$ is explicitly known since $Q$ and $H$ are explicit, and $\alpha(s)=\frac13(s+\sinh s)$. Computing this function, we have that 
	\be\label{checkeo0}
	\hbox{$\hat m(s)>0$ for all $s\in (0, \alpha^{-1}(x_{2,1}))$}
	\ee
	(see Fig. \ref{fig:bound 0 x1} down panel). Hence, by bijectivity of $\alpha$, we conclude $I(y)>0$ for all $y\in (0, x_{2,1})$. This proves that $I(y)\geq 0$ for all $y\geq0$.

	\subsubsection{Proof of Lemma \ref{lem:repulsivity}.}
	Since $h_0(x) \leq 0$ for all $x\geq 0$, we conclude by \eqref{eq:integral formulation}
	\[
	\ba
	& V_0'(x)\\
	& = \frac{h_0(x)}{\phi_0^2(x)}\int_x^\infty \left(V''(y)h_0(y) - V'(y)h_0'(y) - 2V'(y)h_0^2(y)  \right)\left(\frac{\phi_0}{h_0}\right)^2(y) dy \leq 0,
	\ea
	\]
	for all $x\geq0$.

	\subsection{Decay of the derivative of the potential}
	
	In order to prove the positivity of the transformed problem, we need an upper bound for $V_0'$. We state the following lemma.
	\begin{lemma}\label{lem:decay of V0}
		For $|x|\gg 1$ we have that $V_0$ is strictly negative, and decay as $V'(x)$. Even more, the following bound
		\begin{equation}\label{eq:bound V0'}
		\left( \frac{2\tilde\mu_0}{\mu_0} +1\right) V'(x) \leq V_0'(x) \leq \frac12 V'(x),
		\end{equation}
		is satisfied for all $x\geq x_{2,2}$.
	\end{lemma}
	\begin{proof}
		Due to the parity we restrict our analysis to the positive axis, and we can assume that $x\geq x_{2,2}$.
		
		First, we prove the lower bound using that from Lemma \ref{rem:V} $|V'(x)|$ decrease for $x\geq x_{2,2}$, and in addition employing equations \eqref{eq:h0 phi0}, \eqref{eq:h0'}, \eqref{eq:h0 bound x1<x}, we have that
		\begin{align*}
			|V_0'(x)| &\leq \left| \frac{4h_0(x)}{\phi_0^2(x)}\int_x^\infty V'(y)\phi_0^2(y)dy \right| + |V'(x)| \\[0.1cm]
			&= \left| \frac{4h_0(x)}{\phi_0^2(x)}\int_x^\infty \frac{V'(y)}{2h_0(y)}(\phi_0^2(y))'dy \right| + |V'(x)| \\
			&\leq \left| \frac{2\mu_0^{-1}h_0(x)V'(x)}{\phi_0^2(x)}\int_x^\infty (\phi_0^2(y))'dy \right| + |V'(x)| \\
			&\leq \left( \frac{2\tilde\mu_0}{\mu_0} +1\right) |V'(x)| ,
		\end{align*}
		for all $x\geq x_{2,2}$.
		
		Second, analogously to the proof of Lemma \ref{lem:repulsivity} we use the integral formula for $h_0$ and apply specific bounds. Using the definition of $V_0$, Lemma \ref{lem:h0 properties}, equation \eqref{eq:h0 ivp}, and integration by parts,
		\begin{align*}
			V_0'(x) &= 4h_0(x)h_0'(x) - \frac32V'(x) + \frac12V'(x) \\[0.1cm]
			&= h_0(x)h_0'(x) + 3h_0(x)h_0'(x) - \frac32V'(x) + \frac12V'(x) \\[0.1cm]
			&=  -\frac{h_0(x)}{\phi_0^2(x)}\int_x^\infty V'(y)\phi_0^2(y)dy \\
			&\quad  - \frac32\frac{h_0(x)}{\phi_0^2(x)} \int_x^{\infty} \frac{V'(y)}{h_0(y)}(\phi_0^2(y))'dy - \frac32V'(x) + \frac12V'(x)\\[0.1cm]
			&= -\frac{h_0(x)}{\phi_0^2(x)}\int_x^\infty V'(y)\phi_0^2(y)dy  +  \frac32\frac{h_0(x)}{\phi_0^2(x)}\int_x^\infty \left(\frac{V'(y)}{h_0(y)}\right)'\phi_0^2(y) dy \\
			&\quad - \left.\frac32\frac{h_0(x)V'(y)\phi_0^2(y)}{\phi_0^2(x)h(y)}\right|_{x}^{\infty} - \frac32V'(x) + \frac12V'(x)\\[0.1cm]
			&= \frac{h_0(x)}{\phi_0^2(x)}\int_x^\infty \left(\frac32\frac{V'(y)}{h_0(y)} - V(y)\right)' \phi_0^2(y) dy + \frac12V'(x)\\[0.1cm]
			&= \frac12\frac{h_0(x)}{\phi_0^2(x)}\int_x^\infty \left(3\frac{V''(y)}{h_0(y)} - 3\frac{V'(y)h_0'(y)}{h_0^2(y)} - 2V'(y)  \right)\phi_0^2(y) dy + \frac12V'(x)\\[0.1cm]
			&= \frac12\frac{h_0(x)}{\phi_0^2(x)}\int_x^\infty \left(3V''(y)h_0(y) - 3V'(y)h_0'(y) - 2V'(y)h_0^2(y)  \right)\left(\frac{\phi_0}{h_0}\right)^2(y) dy \\
			&\quad + \frac12V'(x).
		\end{align*}
		Thus, we define the integral form for $V_0'$ given by
		\begin{equation}\label{eq:integral form decay}
			V_0'(x) = \frac12\frac{h_0(x)}{\phi_0^2(x)}\int_x^\infty J(y) \left(\frac{\phi_0}{h_0}\right)^2(y) dy + \frac12V'(x)
		\end{equation}
		where we have denoted $J(y)$ as the term in parenthesis in the penultimate equation. Using equation \eqref{eq:h0 ivp} we have
		\begin{equation}\label{eq:Jy}
			J(y) = 3V''(y)h_0(y) - V'(y)(3\mu_0^2 - h_0^2(y) + 3V(y)).
		\end{equation}
		Thus, we only have to prove the positivity of $J(y)$ to obtain \eqref{eq:bound V0'}.  Applying the bound \eqref{eq:h0 bound x1<x}, and the fact that $V(y)>0$,
		\[
		3\mu_0^2 - h_0^2(y) + 3V(y) \geq 3\mu_0^2 - \widetilde\mu_0^2 + 3V(y) ~\geq~ 2\mu_0^2 - \frac{8}{27} > 0.
		\]
		Bounding by below \eqref{eq:Jy} and using Lemma \ref{valor mu0}, since $V'(y)<0$,
		\[
		\begin{aligned}
			J(y) \geq &~{} - 3V''(y) - \left( 3\mu_0^2 - \widetilde\mu_0^2 + 3V(y) \right) V'(y) \\
			 \geq &~{} -3V''(y) - (1.0 + 3V(y))V'(y) \geq 0,
		\end{aligned}
		\]
		for all $y\geq x_{2,2}$, where we obtain the last inequality via the explicit expressions using \eqref{eq:L}, \eqref{Vp} and \eqref{Vpp}.
		Hence, recalling \eqref{eq:integral form decay}, we obtain that
		\[ 
		V_0'(x) = \frac12\underbrace{\frac{h_0(x)}{\phi_0^2(x)}}_{\leq 0}\underbrace{\int_x^\infty J(y) \left(\frac{\phi_0}{h_0}\right)^2(y) dy}_{\geq 0} + \frac12V'(x) \leq \frac12 V'(x) \leq 0. 
		\]
		This ends the proof of Lemma \ref{lem:decay of V0}.
	\end{proof}

\bibliographystyle{amsplain}

\begin{thebibliography}{10}

	\bibitem{AS1} M. Alammari, and S. Snelson, \emph{Existence and stability of near-constant solutions of variable-coefficient scalar-field equations}, to appear in Diff. Integral Eqns. (2023).

	\bibitem{AS2} M. Alammari, and S. Snelson, \emph{Linear and orbital stability analysis for solitary-wave solutions of variable-coefficient scalar field equations}, J. Hyperbolic Diff. Eqns., 19(1) 175--201, 2022.

	\bibitem{AMP} M. Alejo, C. Mu\~noz, and J. Palacios, \emph{On asymptotic stability of the sine-Gordon kink in the energy space}, Commun. Math. Phys. (2023). \url{https://doi.org/10.1007/s00220-023-04736-3}.

	\bibitem{Bam_Cucc} D.~Bambusi, and S.~Cuccagna, \emph{On dispersion of small energy solutions to the nonlinear Klein Gordon equation with a potential},  Amer. J. Math. \textbf{133} (2011), no. 5, 1421--1468.

	\bibitem{BizChmRosR07}
	P.~Bizo\'{n}, T.~Chmaj and A.~Rostworowski, \emph{Late-time tails of a Yang-Mills field on Minkowski and Schwarzschild backgrounds}, Classical and Quantum Gravity, \textbf{24} (2007), no.~13, F55.

	\bibitem{BCS} P. Bizo\'n, T. Chmaj, and  N. Szpak, \emph{Dynamics near the threshold for blow up in the one-dimensional focusing nonlinear Klein-Gordon equation}. J. Math. Phys. \textbf{52} (2011), 103703.
	
	\bibitem{BizKah16}
	P.~Bizo\'{n} and M.~Kahl, \emph{A Yang–Mills field on the extremal {R}eissner–{N}ordstr\"{o}m black hole}, Classical and Quantum Gravity, \textbf{33} (2016), no.~17, 175013.
	
	\bibitem{BerLio83}
	H. Berestycki, P.L. Lions, \emph{Nonlinear scalar field equations, II existence of infinitely many solutions}, Arch. Rational Mech. Anal. \textbf{82}, 347–375 (1983), \href{https://doi.org/10.1007/BF00250556}{https://doi.org/10.1007/BF00250556}.
		
	\bibitem{Chang08}
	S.~M.~Chang, S.~Gustafson, K.~Nakanishi, and T.~P.~Tsai, \emph{Spectra of linearized operators for NLS solitary waves}, SIAM J. Math. Anal., \textbf{39} (2007/8), no.~4, 1070-1111.
	
	\bibitem{GJ1} G. Chen, J. Jendrej, \emph{Strichartz estimates for Klein-Gordon equations with moving potentials}, \href{https://arxiv.org/abs/2210.03462}{arXiv:2210.03462}

	\bibitem{GJ2} G. Chen, J. Jendrej, \emph{Asymptotic stability and classification of multi-solitons for Klein-Gordon equations}, \href{https://arxiv.org/abs/2301.10279}{arXiv:2301.10279}
	
	\bibitem{CLL} G. Chen, J. Liu, and B. Lu, \emph{Long-time asymptotics and stability for the sine-Gordon equation}, Preprint arXiv:2009.04260.
	
	\bibitem{ChenLu} G. Chen, and J. Luhrmann, \emph{Asymptotic stability of the sine-gordon kink}, preprint https://arxiv.org/pdf/2411.07004 (2024). 
	
	\bibitem{Chris81}
	D.~Christodoulou, \emph{Solutions globales des equations de champ de Yang-Mills},  C. R.
	Acad. Sci. Paris A 293, \textbf{39} (1981).
	
	\bibitem{ChrSha97}
	P.~T. Chru{\'s}ciel and J.~Shatah, \emph{Global existence of solutions of the {Y}ang--{M}ills equations on globally hyperbolic four dimensional Lorentzian manifolds}, Asian journal of Mathematics, \textbf{1} (1997), no.~3, 530--548.
	
	\bibitem{Cuc} S. Cuccagna, \emph{On asymptotic stability in 3{D} of kinks for the  {$\phi^4$} model}, Trans. Amer. Math. Soc. \textbf{360} (2008), no.~5,  2581--2614. 

	\bibitem{CucMae} S. Cuccagna, M. Maeda, \emph{Asymptotic stability of kink with internal modes under odd perturbation}, Nonlinear Differ. Equ. Appl. 30, 1 (2023). \url{https://doi.org/10.1007/s00030-022-00806-y}.

	\bibitem{CucMae_survey} S. Cuccagna, M. Maeda, \emph{A survey on asymptotic stability of ground states of nonlinear Schr\"odinger equations. II} Discrete and Continuous Dynamical Systems - S, 14(5):1693--1716, 2021.

	\bibitem{Daly} R. A. Daly et al., \emph{New black hole spin values for Sagittarius $A*$ obtained with the outflow method}, MNRAS 527, 428–436 (2024).

	\bibitem{De85} K. Deimling. \emph{Solutions in Cones. Nonlinear Functional Analysis}, 217-255 (1985).

	\bibitem{DM} J.-M. Delort, and N. Masmoudi, \emph{Long-time dispersive estimates for perturbations of a kink solution of one-dimensional cubic wave equations}, Memoirs of the European Mathematical Society, 1. EMS Press, Berlin, (2022). ix+280 pp. ISBN: 978-3-98547-020-4; 978-3-98547-520-9.
	
	\bibitem{DunSch63}
	N.~Dunford and  J.~T Schwartz, \emph{Spectral theory: self adjoint operators in Hilbert space}, Interscience publishers, 1963.
	
	\bibitem{DKMM} T. Duyckaerts, C. Kenig, Y. Martel, and F. Merle, \emph{Soliton resolution for critical co-rotational wave maps and radial cubic wave equation}. Comm. Math. Phys. 391 (2022), no. 2, 779--871.
	
	\bibitem{EarMon82} D.~M. Eardley and V.~Moncrief, \emph{The global existence of {Y}ang-{M}ills-{H}iggs fields in 4-dimensional {M}inkowski space}, Comm. Math. Phys., \textbf{83} (1982), no.~2, 171--191.
	
	\bibitem{Germain} P. Germain, \emph{A review on asymptotic stability of solitary waves in nonlinear dispersive problems in dimension one}, https://arxiv.org/pdf/2410.04508 (2024).
	
	\bibitem{GPZ} P. Germain, F. Pusateri, and K. Z. Zhang, \emph{On 1d Quadratic Klein-Gordon Equations with a Potential and Symmetries}. Arch. Ration. Mech. Anal. 247 (2023), no. 2, 17.

	\bibitem{GP} P. Germain, and F. Pusateri, \emph{Quadratic Klein-Gordon equations with a potential in one dimension}. Forum Math. Pi 10 (2022), Paper No. e17, 172 pp.
	
	\bibitem{HN2} N. Hayashi and P. Naumkin, \emph{The initial value problem for the cubic nonlinear Klein-Gordon equation}, Z. Angew. Math. Phys. 59 (2008), no. 6, 1002--1028.
	
	\bibitem{HN3} N. Hayashi and P. Naumkin, \emph{The initial value problem for the quadratic nonlinear Klein-Gordon equation}, Adv. Math. Phys. (2010), Art. ID 504324, 35.

	\bibitem{HN4} N. Hayashi and P. Naumkin, \emph{Quadratic nonlinear Klein-Gordon equation in one dimension}, J. Math. Phys. 53 (2012), no. 10, 103711, 36.
	
	\bibitem{IMN} S. Ibrahim, N. Masmoudi, and K. Nakanishi, \emph{Scattering threshold for the focusing nonlinear Klein-Gordon equation}, Anal. PDE 4 (2011), no. 3, 405–460.
	
	\bibitem{HPW} D. B. Henry, J.F. Perez and W. F. Wreszinski, \emph{Stability Theory for Solitary-Wave Solutions of Scalar Field Equations}, Comm. Math. Phys. \textbf{85}, 351--361 (1982).

	\bibitem{JL} J. Jendrej, and A. Lawrie, \emph{Soliton resolution for energy-critical wave maps in the equivariant case.}, arXiv:2106.10738 (2021).
	
	\bibitem{KP} A. Kairzhan and F. Pusateri, \emph{Asymptotic stability near the soliton for quartic Klein-Gordon in 1D}, preprint arXiv:2206.15008 (2022).
	
	\bibitem{KK1} E.~Kopylova and A.~I. Komech, \emph{On asymptotic stability of kink for relativistic {G}inzburg-{L}andau equations}, Arch. Ration. Mech. Anal. \textbf{202} (2011), no.~1, 213--245. 

	\bibitem{KK2} E.~Kopylova, and  A.I. Komech, \emph{On asymptotic stability of moving kink for relativistic {G}inzburg-{L}andau equation.} Comm. Math. Phys. \textbf{302} (2011), no.~1, 225--252.
	
	\bibitem{KM2022} M. Kowalczyk, and Y. Martel, \emph{Kink dynamics under odd perturbations for (1+1)-scalar field models with one internal mode}, Preprint arXiv:2203.04143.

	\bibitem{KMM2017} M. Kowalczyk, Y. Martel, and C. Mu\~noz, \emph{Kink dynamics in the $ \phi^4$ model: Asymptotic stability for odd perturbations in the energy space}, J. Amer. Math. Soc. 30 (2017), 769-798.

	\bibitem{KMM19} M. Kowalczyk, Y. Martel, and C. Mu\~noz, \emph{Soliton dynamics for the 1D NLKG equation with symmetry and in the absence of internal modes}, J. Eur. Math. Soc. (JEMS) 24 (2022), no. 6, 2133--2167.

	\bibitem{KMM17_Survey}
	M.~Kowalczyk, Y.~Martel, and C.~Mu\~{n}oz, \emph{On asymptotic stability of nonlinear waves}, S\'{e}minaire
	{L}aurent {S}chwartz---\'{E}quations aux d\'{e}riv\'{e}es partielles et
	applications. {A}nn\'{e}e 2016--2017, Ed. \'{E}c. Polytech., Palaiseau, 2017,
	pp.~Exp. No. XVIII, 27.

	\bibitem{KMMV20} M. Kowalczyk, Y. Martel, C Mu\~noz, and H. Van Den Bosch, \emph{A Sufficient Condition for Asymptotic Stability of Kinks in General (1+1)-Scalar Field Models}. Ann. PDE 7, 10 (2021).

	\bibitem{KNS} J. Krieger, K. Nakanishi, and  W. Schlag, \emph{Global dynamics above the ground state energy for the one-dimensional NLKG equation.} Math. Z. \textbf{272} (2012), no. 1-2, 297--316.
	
	\bibitem{KST} J. Krieger, W. Schlag, and D. Tataru, \emph{Renormalization and blow up for the critical Yang-Mills problem}, Adv. Math. 221(2009), no.5, 1445--1521.
	
	\bibitem{KR48} M. Krein, and M. Rutman. \emph{Linear operators leaving invariant a cone in a Banach space}, Uspekhi mat. nauk, \textbf{3}(1), 3-95 (1948).
	
	\bibitem{LP} T. L\'eger, and F. Pusateri, \emph{Internal modes and radiation damping for quadratic Klein-Gordon in 3D}, preprint arXiv:2112.13163 (2021).
	
	\bibitem{LP2} T. L\'eger, and F. Pusateri, \emph{Internal mode-induced growth in 3d nonlinear Klein–Gordon equations}, Atti Accad. Naz. Lincei Rend. Lincei Mat. Appl. 33 (2022), no. 3, 695--727.
	
	\bibitem{LiLuhrmann22}
	Y.~Li and J. L\"uhrman, \emph{Soliton dynamics for the 1D quadratic Klein-Gordon equation with symmetry}, Journal of Differential Equations \textbf{344} (2023), 172-202.
	
	\bibitem{LLSS} H. Lindblad, J. L\"uhrmann, W. Schlag, and A. Soffer, \emph{On Modified Scattering for 1D Quadratic Klein-Gordon Equations With Non-Generic Potentials}, Int. Math. Res. Not. IMRN 2023, no. 6, 5118--5208.

	\bibitem{LLS} H. Lindblad, J. L\"uhrmann, and A. Soffer, \emph{Decay and asymptotics for the 1D Klein-Gordon equation with variable coefficient cubic nonlinearities}, SIAM J. Math. Anal. 52 (2020), no. 6, 6379–6411.

	\bibitem{LLS2}H. Lindblad, J. L\"uhrmann, and A. Soffer,  \emph{Asymptotics for 1D Klein-Gordon equations with variable coefficient quadratic nonlinearities}, Arch. Ration. Mech. Anal. 241 (2021), no. 3, 1459–1527.

	\bibitem{LS1} H. Lindblad, and A. Soffer, \emph{A remark on asymptotic completeness for the critical nonlinear Klein-Gordon equation}, Letters Math. Phys. 73 (2005), no. 3, 249--258.
	
	\bibitem{LS2} H. Lindblad, and A. Soffer, \emph{A remark on long range scattering for the nonlinear Klein-Gordon equation}, J. Hyperbolic Diff. Eqns. 2 (2005), no. 1, 77--89.
	
	\bibitem{LS3} H. Lindblad, and A. Soffer, \emph{Scattering for the Klein-Gordon equation with quadratic and variable coefficient cubic nonlinearities}, Trans. Amer. Math. Soc. 367 (2015), no. 12, 8861--8909.

	\bibitem{Lin_Tao} H. Lindbland, and T. Tao, \emph{Asymptotic decay for a one-dimensional nonlinear wave equation}, Anal. and PDE Vol. 5 (2012), No. 2, 411--422, doi: 10.2140/apde.2012.5.411.
	
	\bibitem{Lohe} M. A. Lohe, \emph{Soliton structures in $P(\varphi)_2$}, Physical Review D, 20 (1979), 3120–3130.

	\bibitem{LuS1} J. L\"uhrmann and W. Schlag, \emph{Asymptotic stability of the sine-Gordon kink under odd perturbations}. Duke Math. J. 172 (2023), no. 14, 2715--2820.

	\bibitem{LuS2} J. L\"uhrmann and W. Schlag, \emph{On codimension one stability of the soliton for the 1D focusing cubic Klein-Gordon equation}. Comm. Amer. Math. Soc. 4 (2024), 230--356.
	
	\bibitem{NS} K.~Nakanishi, and  W.~Schlag, \emph{Invariant manifolds and dispersive Hamiltonian evolution equations.} Z\"urich Lectures in Advanced Mathematics. European Mathematical Society (EMS), Z\"urich, 2011.
	
	\bibitem{PosTel33}
	G.~P{\"o}schl and E.~Teller, \emph{Bemerkungen zur Quantenmechanik des anharmonischen Oszillators}, Zeitschrift f{\"u}r Physik, \textbf{83} (1933), no.~3, 143--151.
	
	\bibitem{ManSut04}
	N.~Manton and P.~Sutcliffe, \emph{Topological solitons}, Cambridge Monographs on Mathematical Physics, Camb. U. Press, Cambridge, 2004.
	
	\bibitem{Martel1} Y. Martel, \emph{Asymptotic stability of solitary waves for the 1D cubic-quintic Schr\"odinger equation with no internal mode}. Probab. Math. Phys. 3 (2022), no. 4, 839--867.
	
	\bibitem{Martel2} Y. Martel, \emph{Asymptotic stability of small standing solitary waves of the one-dimensional cubic-quintic Schr\"odinger equation}, preprint arXiv:2312.11016 (2023). \url{https://arxiv.org/abs/2312.11016}.
	
	\bibitem{MM1} Y. Martel, and F. Merle, \emph{A Liouville theorem for the critical generalized Korteweg--de Vries equation}, J. Math. Pures Appl. 79 (2000), 339--425.
	
	\bibitem{Mau23}
	C. Maul\'en, \emph{Asymptotic stability manifolds for solitons in the generalized Good Boussinesq equation}, J. Math. Pures Appl. (9) 177 (2023), 260--328.
	
	\bibitem{MM23}
	C. Maul\'en, and C. Mu\~noz, \emph{Asymptotic stability of the fourth order $\phi^4$ kink for general perturbations in the energy space}, Annales IHP Analyse Nonlin\'eaire (online first) \url{https://doi.org/10.4171/aihpc/112}.
	
	\bibitem{MR} F. Merle, and P. Rapha\"el, \emph{The blow-up dynamic and upper bound on the blow-up rate for critical nonlinear Schr\"odinger equation}. Ann. of Math. (2) 161 (2005), no. 1, 157--222.
	
	\bibitem{MP} C. Mu\~noz, and J.M. Palacios, \emph{Nonlinear stability of 2-solitons of the Sine-Gordon equation in the energy space}, Ann. IHP C Analyse Nonlin\'eaire. 36 (2019), no. 4, 977--1034.
	
	\bibitem{OT1}  S. J. Oh, and D. Tataru, \emph{The threshold conjecture for the energy critical hyperbolic Yang-Mills equation}. Ann. of Math. (2) 194 (2021), no. 2, 393--473.
	
	\bibitem{OT2} S. J. Oh, and D. Tataru, \emph{The hyperbolic Yang-Mills equation in the caloric gauge: local well-posedness and control of energy-dispersed solutions}. Pure Appl. Anal. 2 (2020), no. 2, 233--384.
	
	\bibitem{OT3} S. J. Oh, and D. Tataru, \emph{The hyperbolic Yang-Mills equation for connections in an arbitrary topological class}. Comm. Math. Phys. 365 (2019), no. 2, 685--739.
	
	\bibitem{PP} J. M. Palacios, and F. Pusateri, \emph{Local Energy control in the presence of a zero-energy resonance}, arXiv:2401.02623 (2024).
	
	\bibitem{RR}  P. Rapha\"el, and I. Rodnianski, \emph{Stable blow up dynamics for the critical co-rotational wave maps and equivariant Yang-Mills problems}. Publ. Math. Inst. Hautes \'Etudes Sci. 115 (2012), 1--122.

	\bibitem{ReedSimon78}
	M.~Reed, B.~Simon. \emph{Methods of modern mathematical physics: Functional analysis. IV: Analysis of Operators}, Elsevier  \textbf{4} 1978.

	\bibitem{SW2} A.~Soffer and M.~I. Weinstein, \emph{Time dependent resonance theory}, Geom.  Funct. Anal. \textbf{8} (1998), no.~6, 1086--1128. 

	\bibitem{SW3} A. Soffer, and  M.I. Weinstein, \emph{Resonances, radiation damping and instability in Hamiltonian nonlinear wave equations}, Invent. Math. \textbf{136}, no. 1, 9--74   (1999).

	\bibitem{Snelson} S. Snelson, \emph{Asymptotic stability for odd perturbations of the the stationary kink in the variable-speed $\phi^4$ model}, Transactions of the AMS 370(10) 7437-7460, 2018.

	\bibitem{Sre07}  M.~Srednicki. \emph{Quantum field theory}. Cambridge University Press, 2007.
	
	\bibitem{Vac06} T. Vachaspati, \emph{Kinks and domain walls}, Cambridge University Press,   New York, 2006, An introduction to classical and quantum solitons.
			
	\bibitem{YangMills54} C. N. Yang and R.~L Mills, \emph{Conservation of isotopic spin and isotopic gauge invariance}, Physical review, \textbf{96} (1954), no.~1, 191.
				
	\bibitem{Weinstein} M. I. Weinstein, \emph{Modulational stability of ground states of nonlinear Schr\"odinger equations}, SIAM J. Math. Anal., 16 (1985), 472--491. 		


	\end{thebibliography}

\end{document}